\documentclass[lettersize,journal]{IEEEtran}

\usepackage{amsmath,amsfonts}
\usepackage{algorithmic}
\usepackage{algorithm}
\usepackage{array}
\usepackage[caption=false,font=normalsize,labelfont=sf,textfont=sf]{subfig}
\usepackage{textcomp}
\usepackage{stfloats}
\usepackage{url}
\usepackage{verbatim}
\usepackage{graphicx}
\usepackage{cite}

\usepackage{amssymb,amsthm,mathabx,bm,mathtools,stmaryrd,nicefrac}
\usepackage{aliascnt}
\mathtoolsset{showonlyrefs}
\usepackage{csquotes}
\usepackage{booktabs}
\usepackage[shortlabels]{enumitem}
\usepackage{multirow}

\usepackage[dvipsnames]{xcolor}
\usepackage[hidelinks]{hyperref}

\graphicspath{{../images/},{../matlab/images/}}
\usepackage{tikz,pgfplots}
\pgfplotsset{compat=1.18}

\usepackage{bbm, dsfont}
\newtheorem{theorem}{Theorem}

\newaliascnt{lemma}{theorem}
\newtheorem{lemma}[lemma]{Lemma}
\aliascntresetthe{lemma}

\newaliascnt{corollary}{theorem}
\newtheorem{corollary}[corollary]{Corollary}
\aliascntresetthe{corollary}

\newaliascnt{proposition}{theorem}

\aliascntresetthe{proposition}

\theoremstyle{definition}

\newaliascnt{remark}{theorem}
\newtheorem{remark}[remark]{Remark}
\aliascntresetthe{remark}

\newtheorem*{remark*}{Remark}

\newcommand{\E}{\ensuremath{\mathbb{E}}}

\newcommand{\N}{\ensuremath{\mathbb{N}}}

\newcommand{\R}{\ensuremath{\mathbb{R}}}
\renewcommand{\S}{{\mathbb{S}}}

\newcommand{\X}{\ensuremath{\mathbb{X}}}

\newcommand{\zba}{{\zb{a}}}
\newcommand{\zbb}{{\zb{b}}}
\newcommand{\zbc}{{\zb{c}}}
\newcommand{\zbe}{{\zb{e}}}
\newcommand{\zbk}{{\zb{k}}}
\newcommand{\zbn}{{\zb{n}}}
\newcommand{\zbv}{{\zb{v}}}
\newcommand{\zbw}{{\zb{w}}}
\newcommand{\zbx}{{\zb{x}}}
\newcommand{\zby}{{\zb{y}}}
\newcommand{\zbA}{{\zb{A}}}
\newcommand{\bftheta}{\boldsymbol{\theta}}
\newcommand{\bfvartheta}{\boldsymbol{\vartheta}}

\newcommand{\Fourier}[1]{\mathcal{F}_{#1}}

\newcommand{\Radon}[1]{\mathcal{R}_{#1}}

\renewcommand{\P}{\ensuremath{\mathcal{P}}}

\newcommand{\e}{\mathrm{e}}

\newcommand{\I}{\mathrm{i}}

\newcommand{\zb}[1]{\ensuremath{\boldsymbol{#1}}}

\DeclareMathOperator*{\argmin}{arg\,min}

\DeclareMathOperator*{\supp}{supp}

\DeclareMathOperator{\sinc}{sinc}

\renewcommand{\d}{\, \mathrm{d}}

\newcommand{\sphere}{\mathbb S}

\newcommand{\ii}[1]{\llbracket #1\rrbracket}

\newcommand{\gray}[1]{{\color{gray}{#1}}}
\newcommand{\tT}{\top}

\def\Lebesgue	{\mathrm L}
\def\Wasserstein	{\mathrm W}
\def\SWasserstein	{\mathrm{SW}}

\def\Radon		{\mathcal R}
\def\NRCDT		{\mathcal N}
\def\mNRCDT			{\textsubscript{m}NR-CDT}

\def\maxNRCDT	{\NRCDT_{\mathrm{m}}}

\def\GL			{\mathrm{GL}}

\begin{document}

\title{A Discrete Radon Transform Based on\\ the Area of Cube-Plane Intersection}

\author{Robert Beinert, Jonas Bresch, Michael Quellmalz
\thanks{R. Beinert, J. Bresch, and M. Quellmalz are with 
Institute of Mathematics, Technische Universität Berlin, Berlin,
Germany. {beinert,bresch,quellmalz}@math.tu-berlin.de 
(all authors contributed equally).}%
}



\maketitle

\begin{abstract}
The Radon transform is a fundamental tool 
for analyzing data
in tomographic imaging,
optimal transport, crystallography, and geometric analysis.
Numerical computations require an accurate discretization.
To deal with voxelized images and objects,
we derive a closed-form, piecewise polynomial expression 
for the Radon transform of an axis-aligned cube in arbitrary dimension $\bm d$.
Building on this formula,
we propose a discrete Radon transform in $\bm{\R^d}$ that is both
analytically exact for voxelized data
and computationally efficient. 
For improved numerical stability, 
we introduce a regularized variant replacing the Radon transform of a cube, 
i.e.\ the $\bm{(d-1)}$-dimensional area of the intersection between that cube and a hyperplane,
by the $\bm d$-dimensional volume of the intersection between the cube and a thin slab around the hyperplane.
Numerical experiments demonstrate the effectiveness of the proposed approach in several applications 
including 3D shape matching, classification, and sliced Wasserstein barycenters.
The computational efficiency in higher dimensions is verified by a comparison with Monte Carlo integration.
\end{abstract}

\begin{IEEEkeywords}
Radon transform,
Fourier transform,
polynomial splines,
discretization,
shape matching,
classification,
sliced Wasserstein 
distance and barycenter.
\end{IEEEkeywords}

\section{Introduction}

\IEEEPARstart{I}{n} imaging, inverse problems, and geometric analysis 
the Radon transform holds a pivotal role
by linking functions to their integrals over hyperplanes.
Thereby,
it encodes geometric information 
about a function
that can be probed analytically 
and reconstructed algorithmically.
Understanding its behavior
for simple geometric bodies 
is thus both practically and theoretically important. 
In particular, 
the problem of computing the 
$(d-1)$-dimensional volume 
of the intersection between a hyperplane 
and an axis-aligned hypercube 
has a long history: 
classical closed-form results for special directions 
date back to {Sommerfeld} \cite{Sommerfeld1904}
and {Pólya} \cite{Polya1913}.
The latter article ``Berechnung eines bestimmten Integrals'' (``Computation of a certain integral'')
seems not to be well-known
and apparently unknown to \cite{Barrow1979}.
More recently, sharp and asymptotic bounds 
and extremal results were obtained 
in \cite{Hensley1979,Ball1986,Konig2011,Konig2021,Moody2013}.

From the geometric-analysis side, 
a detailed understanding of cube–hyperplane intersections 
informs local extrema of section volumes \cite{Pournin2024a,Oneil1971,Deloera2025} 
and combinatorial bounds on intersection vertices 
\cite{Melo2019,Groenland2020,Brandenburg2025}.
While prior works give sharp bounds 
and treat a number of special configurations (notably the main diagonal and central slices),
explicit expressions 
for general directions and offsets,
which are amenable to numerical evaluation,
have been missing in a single, unified form, cf.\ Remark \ref{rem:related}.

From the computational imaging side, 
there is an extensive literature on discrete Radon transforms,
including {Fourier}-based approaches 
and numerical approximations for tomographic imaging \cite{Averbuch2003,Averbuch2008a,Averbuch2008b,Averbuch2016,Beylkin1987}.
Applications of 3D {Radon} representations include rendering, 
shape matching and object retrieval \cite{Daras2004,Daras2006a,Daras2006b,katayama2023content,ma20203d}, 
crystallography, medical/seismic imaging, biometric authentication \cite{Meduna2022,Shi2024,Lanzavecchia1998,mahmoud20063d},
and kernel-based methods \cite{RuxQueSte24,RuxQueSte25}. 
Practical algorithms related to voxel-friendly approximations,
spline convolutions and nonuniform {Fourier} methods are discussed in \cite{Horbelt2002,Knopp2007,Wechsler2024,Marichal2020,Nikolaev2023}.
However, many existing discrete schemes either reduce the 3D problem
to repeated 2D line integrals, rely on {Fourier} interpolation 
that introduces artifacts for voxelized data, 
or employ Monte Carlo approximations \cite{Agarwal2019,ChyManRei2008,DeHoop1996} 
that are costly for high resolution. 
Further studies on the discretization 
are done via the operator norm \cite{Bresch2024}
and adjoint mismatch \cite{Bresch2025}.

Motivated by these limitations, 
we derive an explicit formula 
for the $d$-dimensional {Radon} transform 
of a hypercube.
Building on this, 
we propose a voxel-aware discrete Radon transform 
that uses the exact cube–plane intersection area 
for summing voxel contributions.  
This discretization 
is both more accurate than simple center-projection binning, cf.\ \cite{Huber25} for $d=2$,
and substantially faster than high-quality Monte Carlo sampling.

\paragraph*{Main contributions}
\begin{enumerate}[(i)]
  \item An explicit closed form expression
  for the {Radon} transform 
  of a hypercube in $\R^d$ (Theorem~\ref{thm:area})
  in terms of a piecewise polynomial of degree $d-1$ that is  ${d-2}$ times differentiable.
  \item An explicit formula for a regularized Radon transform computing the 
  $d$-dimensional volume of the intersection of the cube with a thin slab,
  which is the $\varepsilon$-neighborhood of a hyperplane (Corollary~\ref{cor:slab}).
  \item A voxel-aware discrete {Radon} transform
  utilizing the exact formulas from (i) and (ii).
  \item Numerical simulations
  demonstrating the approach on 3D shape matching and classification tasks, 
  sliced Wasserstein minimization and barycenters, 
  and comparisons with Monte Carlo integration.
\end{enumerate}

\paragraph*{Organization}
Section~\ref{sec:radon_transform} provides explicit formulas 
for the Radon transform of a hypercube in $\R^d$,
special cases $d\in\{2,3\}$,
and the cube--slab intersection volume.
Two application areas are considered:
Firstly, Section~\ref{sec:discrete} 
introduces a voxel-aware discrete {Radon} transform,
and performs shape matching 
via variants of the trace transform \cite{Daras2004} 
and the normalized Radon cumulative distribution transform \cite{Beckmann2025a}.
Secondly,
Section~\ref{sec:continuous}
provides proof-of-concept experiments 
for the approximation of empirical measures
and Wasserstein barycenters, both based on a free-support Radon model.
Finally, Section~\ref{sec:MC} compares
Monte Carlo sampling for approximating 
the 4D Radon transform 
with our explicit formula.

\IEEEpubidadjcol

\section{Multi-Dimensional {Radon} transform}
\label{sec:radon_transform}
  
\subsection{Basic definitions}
\label{sec:definitions}

For an absolutely integrable function $f\in \Lebesgue^1(\R^d)$
and a fixed direction $\bftheta \in \S^{d-1}\coloneqq\{\zbx\in\R^d:\|\zbx\|=1\}$,
the \emph{{Radon} transform} is defined as 
\begin{equation}
    \label{eq:radon}
  R_{\bftheta} [f](t)
  \coloneqq
  \int_{H_{\bftheta}(t)} f(\zbx) \d \sigma ,
  \qquad t \in \R,
\end{equation}
where $\sigma$ denotes the surface measure on the \emph{hyperplane }
\begin{equation} \label{eq:hyperplane}
    H_{\bftheta}(t)
    \coloneqq
    \{\zbx\in\R^d : \langle \zbx,\bftheta\rangle = t\}.
\end{equation}
The {Radon} transform is well-defined for almost all $t \in \R$
and obeys the symmetry 
$R_{-\bftheta}[f](-t)=R_{\bftheta} [f](t)$.

Furthermore,
the \emph{{Fourier} transform} is defined by
\begin{equation} \label{eq:FT}
  \mathcal F_d[f](\zbv)
  \coloneqq
  (2\pi)^{-d/2} \int_{\R^d} f(\zbx)\, \e^{\I \langle \zbx, \zbv\rangle} \d \zbx
  ,\qquad \zbv\in\R^d.
\end{equation}
The {Fourier} and {Radon} transform are related 
via the {Fourier} slice theorem \cite{Palamodov2016},
which 
for all $s \in \R$ and $\bftheta \in \S^{d-1}$
states
\begin{equation} \label{eq:Fourier-slice}
  \mathcal F_1[R_{\bftheta} [f]](s)
  =
  (2\pi)^{\frac{d-1}{2}} \mathcal F_d[f](s \bftheta).
\end{equation}

For $ a >0$, 
we denote the \emph{indicator function} of the interval $(-a,a]$ by
\begin{equation}
    \mathbf 1_{a}(x)
    \coloneqq
    \begin{cases}
        1, & \text{if } x \in (-a,a],\\
        0, & \text{otherwise},\\
    \end{cases}
\end{equation}
and 
the indicator function of the  \emph{hypercube}
\begin{equation}
    \label{eq:box}
    (-\zba,\zba] 
    \coloneqq 
    \bigtimes_{i=1}^d(-a_i,a_i] ,
    \qquad
    \zb a \in\R^d_{>0},
\end{equation}
by
\begin{equation}
    \label{eq:char}
    \mathbf 1_{\zb a}(\zb x)
    \coloneqq
    \mathbf 1_{a_1}(x_1) \cdots \mathbf 1_{a_d}(x_d)
    ,\qquad \zb x\in\R^d.
\end{equation}
Furthermore, we define the \emph{positive part} of $x\in\R$ by
\begin{equation}
    (x)_+
    \coloneqq
    \begin{cases}
        x, & \text{if }x>0,\\
        0, & \text{if }x\le0.\\
    \end{cases}
\end{equation}
We frequently use its $\ell$-th power $(x)_+^\ell$ for $\ell\in\N_0$,
where we set $0^0\coloneqq0$.
Moreover,
we define the vector 
$$ 
    \zbx^\circ \coloneqq \zbx|_{\supp(\zbx)} \in\R^{\ell}
    ,\qquad
    \zbx\in\R^d,
$$
whose $j$-th entry is the $j$-th non-zero entry of~$\zbx$ 
and where 
$$\ell=\|\zbx\|_0 \coloneqq |\supp(\zbx)|$$
denotes the number of non-zero entries of~$\zb x$.
The product of all entries of $\zbx\in\R^d$ is denoted by
\begin{equation} \label{eq:sign}
    P (\zbx) \coloneqq
    \prod_{j=1}^d x_j,
\end{equation}
and the componentwise product of two vectors $\zbx,\zby\in\R^d$
by $\zbx\varodot \zby\in\R^d$.

The {Fourier} transform of $\mathbf 1_{\zba}$ 
with $\zba \in \R_{>0}^d$
can be expressed 
as
\begin{equation} \label{eq:Fourier-cube}
  \mathcal F_d[ \mathbf 1_{\zba}](\zbv)
  =
  \left(\frac2\pi\right)^{\nicefrac{d}{2}} \prod_{j=1}^d a_j \sinc(a_jv_j),
\end{equation}
see \cite[Ex.~2.3]{PlPoStTa23},
with the \emph{cardinal sine} function
$$
\sinc(x)
\coloneqq \begin{cases}
  \frac{\sin(x)}{x},& x\in\R\setminus\{0\},\\
  1,&x=0.
\end{cases}
$$

\subsection{The {Radon} transform of a hypercube}
\label{sec:thm_area}

The next theorem
provides a closed form for the {Radon} transform
of the indicator function $\mathbf 1_{\zba}$ of the hypercube $(-\zba, \zba]$.
This coincides with the surface area of $(-\zba, \zba] \cap H_{\bftheta}(t)$.

\begin{theorem}\label{thm:area}
  Let $\bftheta\in\S^{d-1}$, $t\in\R$, 
  and $\zba\in \R^d_{>0}$.
  Furthermore,
  let $\ell\coloneqq \|\bftheta\|_0$. 
  Then, 
  for the {Radon} transformation 
  of the hypercube, 
  $A_{\bftheta}^{\zba}(t)
  \coloneqq
  \mathcal R_{\bftheta}[\mathbf 1_{\zba}](t)$,
  holds
  \begin{align} \label{eq:Radon-cube}
    \hspace{-10pt}A_{\bftheta}^{\zba}(t)
    &= \frac{2^{d-\ell} P(\zba)}{P((\zba\varodot\bftheta)^\circ)\, (\ell-1)!} 
    \\
    &\qquad\qquad\times
    \smashoperator{\sum_{\zbk\in\{-1,1\}^\ell}} P(\zbk) \,
    \big(t+\langle \zbk,(\zba\varodot\bftheta)^\circ\rangle\big)_+^{\ell-1}
  \end{align}
\end{theorem}

The proof 
requires the following lemma on convolutions.
Related results about splines 
can be found in \cite{Killmann2001}, \cite[§~1]{Boor1993}.
The \emph{convolution} of two functions $g,h\in \Lebesgue^1(\R)$ 
is defined as 
\begin{equation}
    (g*h)(x)=
    \frac{1}{\sqrt{2\pi}}
    \int_\R g(y) h(y-x) \d y,
    \qquad x \in \R.
\end{equation}
Iteratively,
we denote the \emph{$k$-fold convolution} of $k$ functions $g_j \in \Lebesgue^1(\R)$
by 
\begin{equation}
    \label{eq: k fold}
    \bigast_{j = 1}^k g_j \coloneqq g_1 \ast ... \ast g_k.
\end{equation}
The {Fourier} convolution theorem \cite[Thm.~2.15]{PlPoStTa23}
relates the {Fourier} transform 
and the convolution
of two function 
by
\begin{equation}
    \label{eq:Fourier convolution}
    \mathcal F_1 [g * h] = \mathcal F_1[g] \cdot \mathcal F_1[h],
\end{equation}
which directly extends to the $k$-fold convolution from \eqref{eq: k fold}.

\begin{lemma}
    \label{lem:conv-inds}
    Let $k \in \N$ and $\zbb \in \R^k_{>0}$. 
    Then the $k$-fold convolution of indicator functions is
    \begin{equation} \label{eq:conv-inds}
    \bigast_{j=1}^k \mathbf 1_{b_j} (t)
    =
    \tfrac{(2\pi)^{\nicefrac{(1-k)}{2}}}{(k-1)!} 
    \smashoperator{\sum_{\zbk \in \{-1,1\}^k}}
    P(\zbk) (t + \langle \zbk, \zbb \rangle)_+^{k-1},
    \quad t \in \R.
    \end{equation}
\end{lemma}

\begin{proof}
    For $k=1$, 
    we have
    \begin{equation}
        \mathbf 1_{b_1} (t)
        =
        (t + b_1)_+^0 - (t - b_1)_+^0,
        \qquad t \in \R.
    \end{equation}
    Now consider $\zbb \in \R^k_{\ge0}$ 
    and define for $\zbc \in \R^d$ 
    the by the last component reduced vector $\zbc' \coloneqq (c_1, \dots,c_{k-1})^\tT \in \R^{d-1}$.
    Then, we have for any $t \in \R$ that
    \begin{align}
    &\Bigl(\mathbf 1_{b_k}
    *
    \tfrac{(2\pi)^{\nicefrac{(2-k)}{2}}}{(k-2)!} 
    \hspace{-5pt}\sum_{\zbk' \in \{-1,1\}^{k-1}}\hspace{-5pt}
    P(\zbk') (t + \langle \zbk', \zbb' \rangle)_+^{k-2}
    \Bigr)(t)
    \\
    &=
    \tfrac{(2\pi)^{\nicefrac{(1-k)}{2}}}{(k-2)!} 
    \sum_{\zbk' \in \{-1,1\}^{k-1}}
    \int_{t - b_k}^{t + b_k}
    P(\zbk') (s + \langle \zbk', \zbb' \rangle)_+^{k-2} \d s
    \\
    &=
    \tfrac{(2\pi)^{\nicefrac{(1-k)}{2}}}{(k-1)!} 
    \Bigg[\sum_{\zbk' \in \{-1,1\}^{k-1}}
    P(\zbk') (t + b_k + \langle \zbk', \zbb' \rangle)_+^{k-1}
    \\[-10pt]
    &\hspace{4cm}
    -P(\zbk') (t - b_k + \langle \zbk', \zbb' \rangle)_+^{k-1}\Bigg]
    \\
    &=
    \tfrac{(2\pi)^{\nicefrac{(1-k)}{2}}}{(k-1)!} 
    \sum_{\zbk \in \{-1,1\}^{k}}
    P(\zbk) (t + \langle \zbk, \zbb \rangle)_+^{k-1}.
    \qedhere
    \end{align}
\end{proof}

\begin{proof}[Proof of Theorem \ref{thm:area}]
  By \eqref{eq:Fourier-cube} and the {Fourier} slice theorem~\eqref{eq:Fourier-slice},
  we obtain for almost all $s \in \R$ that
  \begin{align}
  \mathcal F_1[\Radon_{\bftheta} [\mathbf 1_{\zba}]](s)
  &=
  (2\pi)^{\frac{d-1}{2}} \mathcal F_d[\mathbf 1_{\zba}](s \bftheta)
  \\
  &=
  \frac{2^{d}}{\sqrt{2\pi}}
  \prod_{j=1}^d a_j\sinc(a_j s \theta_j)
  \\
  &= 
  \frac{2^{d}}{\sqrt{2\pi}P(\bftheta^\circ)}
  \prod_{\theta_j \neq 0} a_j \theta_j\sinc(a_j s \theta_j)
  \prod_{\theta_j=0} a_j
  \\[-10pt]
  &= 
  \frac{2^{d-\ell}P(\zba)}{P((\zba\varodot\bftheta)^\circ)}
  (2\pi)^{\frac{\ell-1}{2}}
  \prod_{\theta_j \neq 0} \Fourier{1}[\mathbf 1_{a_j\theta_j}](s) 
  \\
  &= 
  \frac{2^{d-\ell}P(\zba)}{P((\zba\varodot\bftheta)^\circ)}
  (2\pi)^{\frac{\ell-1}{2}}
  \Fourier{1}\Bigl[\bigast_{j = 1}^\ell \mathbf 1_{(\zba\varodot\bftheta)^\circ_j}\Bigr] (s)
  \end{align}
  by
  the {Fourier} convolution theorem \eqref{eq:Fourier convolution}.
  Lemma~\ref{lem:conv-inds}
  yields a representative of the {Radon} transform of $\mathbf 1_{\zba}$
  in the space of continuous functions
  by the right-hand-side of \eqref{eq:conv-inds}.
  This finishes the proof.
\end{proof}

The summation index $\zbk\in\{-1,1\}^\ell$ for $\ell=d$ in \eqref{eq:Radon-cube} 
corresponds to the $2^d$ vertices of the hypercube 
$(-\zba,\zba]$,
which has the corners $\zba\varodot\zbk$, 
where $\zbk\in\{-1,1\}^d$.

The intersection of $(-\zba, \zba]$ with the hyperplane $H_{\bftheta}(t)$
is a convex polytope
with at most 
$\lfloor \frac{d+1}{2}\rfloor \binom{d}{\lfloor d/2\rfloor}$
vertices, and this bound is attained~\cite{Oneil1971}.
These can be used for computing the area by triangulation,
but their number 
grows asymptotically like $2^{d+1/2}$,
whereas our summation has at most $2^d$ summands.

\subsection{Explicit formulas from Theorem~\ref{thm:area}}
\label{sec:cor_lower_dimensional}

First,
we note the degenerate case  $\ell = 1$, i.e. 
the {Radon} transform 
of the hypercube in direction of the axes $\zb\theta=\zbe_i$, which denotes the $i$th unit vector in $\R^d$:
\begin{equation} \label{eq:l1}
    A_{\zbe_i}^{\zba}(t)
    = 
    2^{d-1} \mathbf 1_{a_i}(t) \tfrac{P(\zba)}{a_i},
    \quad t \in \R,
    \quad i \in \ii{d},
\end{equation} 
where $\ii{n} \coloneqq \{1,...,n\}$ for $n \in \N$.

Second,
we state explicit versions of Theorem~\ref{thm:area} 
for the important cases of dimension $d\in\{2,3\}$.
For a 2D rectangle, we have the following.

\begin{corollary}[2D {Radon} transform] \label{cor:2d} 
    Let $\bftheta\in\S^1$ with $\ell = \|\bftheta\|_0$ and $\zba\in\R^2_{>0}$.
    If $\ell=1$,
    the area $A_{\bftheta}^{\zba}$ is given in \eqref{eq:l1}.
    If $\ell=2$,
    we set  $(t_i)_{i = 1}^4$ as a nondecreasing rearrangement 
    of $(\langle \bftheta, \zba\varodot \zbk\rangle)_{\zbk\in\{-1,1\}^2}$,
    and it holds
    \begin{equation} \label{eq:radon2d_t}
        A_{\bftheta}^{\zba}(t) 
        =
        \frac{1}{|\theta_1\theta_2|}
        \begin{cases}
        0,
        & t\le t_1,
        \\ 
        t - t_1,
        & t \in (t_1,t_2], 
        \\
        t_2 - t_1,
        & t \in (t_2,t_3],
        \\
        t_4-t,
        & t \in (t_3, t_4], 
        \\
        0,
        & t \ge t_4. 
        \end{cases}
    \end{equation}
\end{corollary}

In the 3D case,
the intersection of a plane with a cube is a polygon with 3, 4, 5, or 6 edges 
(or in the degenerate case a line segment, singleton, or empty), 
see Fig.~\ref{fig:radon3d}~(b).

\begin{corollary}[3D {Radon} transform] \label{cor:3d}
    Let $\bftheta\in\S^2$ with $\ell=\|\bftheta\|_0$
    and $\zba\in\R^3_{>0}$.
    If $\ell=1$,
    the area $A_{\bftheta}^{\zba}(t)$ is given in \eqref{eq:l1}.    
    If $\ell = 2$,
    we have 
    $A_{\bftheta}^{\zba}(t) = 2a_j A_{\bftheta^\circ}^{\zba^\circ}(t)$
    given in Corollary~\ref{cor:2d},
    where $j$ is the index with $\theta_j=0$.
    If $\ell=3$,
    we set  $(t_i)_{i = 1}^8$ as a nondecreasing rearrangement 
    of $(\langle\bftheta,\zba\varodot \zbk\rangle)_{\zbk\in\{-1,1\}^3}$.
    Then it holds 
    \begin{equation}
        A_{\bftheta}^{\zba}(t)
        =
        \begin{cases}
            0 
            & t\le t_1,
            \\
            \frac{(t - t_1)^2}{2|\theta_1\theta_2\theta_3|}
            & t \in (t_1, t_2],
            \\
            \frac{(t_2 - t_1)(2t - t_2 - t_1)}{2|\theta_1\theta_2\theta_3|}
            & t \in (t_2, t_3],
            \\
            \frac{t(t + 2(-t_3 + t_2 - t_1)) + t_3^2 - t_2^2 + t_1^2}{2|\theta_1\theta_2\theta_3|}
            & t \in (t_3, t_4],
            \\
            \frac{2t(t_4 - t_3 + t_2 - t_1) - t_4^2 + t_3^2 - t_2^2 + t_1^2}{2|\theta_1\theta_2\theta_3|},
            & t \in (t_4,0].
        \end{cases}
    \end{equation}
    For $t>0$, 
    $A_{\bftheta}^{\zba}$ can be obtained from $A_{\bftheta}^{\zba}(t) = A_{\bftheta}^{\zba}(-t)$, 
    and 
    the symmetry
    $t_i = -t_{9-i} \le 0$ for all $i \in \ii 4$.
\end{corollary}

The points $t_i$ 
are the projections of the vertices $\zba\varodot \zbk$ 
to the line in direction $\bftheta$, 
see Fig.~\ref{fig:radon3d}~(a).
If some of the $t_i$ coincide, 
the respective intervals
in the last corollary are empty.

In general, 
the function $A_{\bftheta}^{\zba}$ is a piecewise polynomial in $t$ of degree $\ell-1$ 
and it is $\ell-2$ times continuously differentiable if $\ell\ge2$.
Fig.~\ref{fig:radon3d_function} indicates a continuous behavior in 3D
between the non-degenerate  
and the degenerate case,
the latter coincides with the non-degenerate case 
from the 2D Radon transform, cf.~Corollary~\ref{cor:2d}.

\begin{figure}
\resizebox{\linewidth}{!}{%
\begin{tabular}{@{} l @{} l @{}}
    \includegraphics[width=0.65\linewidth, clip=true, trim=20pt 12pt 14pt 12pt]{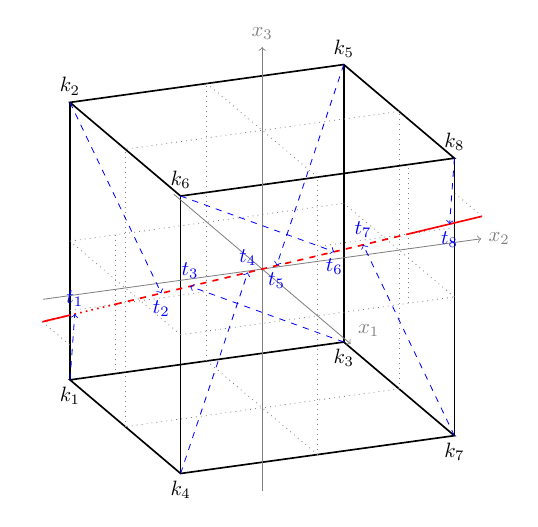} 
    & \includegraphics[width=0.65\linewidth, clip=true, trim=20pt 12pt 14pt 12pt]{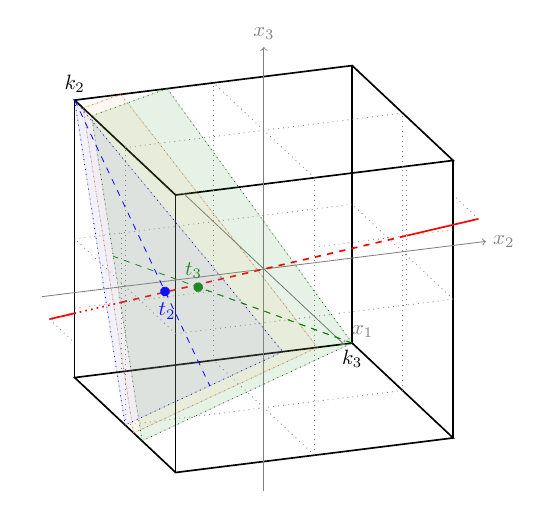} \\
    \raisebox{8mm}[0mm][0mm]{\quad (a)} & \raisebox{8mm}[0mm][0mm]{\quad (b)}
\end{tabular}}
\vspace{-15pt}
\caption{
{(a)}: 
Projections $t_i$ of the corners $k_i$ of the cube to the red line in direction~$\theta$.
{(b)}: Planes $H_{\bftheta}(t)$ for $t=t_2$ in blue, $t=t_3$ in green, and $t\in(t_2,t_3)$ in yellow,
with three or five corners.}
\label{fig:radon3d}
\end{figure}

\begin{figure}
\resizebox{\linewidth}{!}{%
\begin{tabular}{@{} c @{} c @{} c @{}}
    \includegraphics[width=0.4\linewidth]{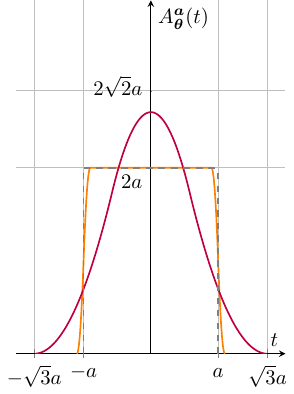}
    & \includegraphics[width=0.4\linewidth]{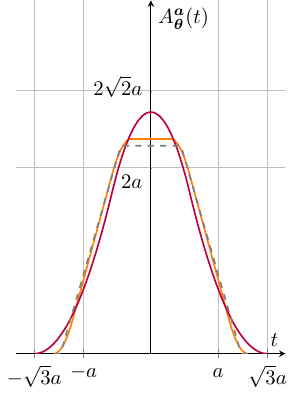}
    & \includegraphics[width=0.4\linewidth]{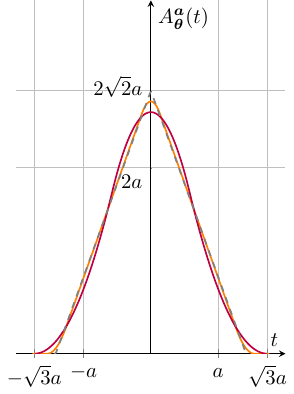} \\
    {\gray{$\bftheta = (0,0,1)^\top$}}
    & {\gray{$\bftheta = \nicefrac{1}{\sqrt{5}}(0,1,2)^\top$}}
    & {\gray{$\bftheta = \nicefrac{1}{\sqrt{2}}(0,1,1)^\top$}} \\
    {\color{orange}{$\bftheta = \nicefrac{\sqrt{26}}{}(0,1,5)^\top$}}
    & {\color{orange}{$\bftheta = \nicefrac{1}{\sqrt{48}}(1,3,6)^\top$}}
    & {\color{orange}{$\bftheta = \nicefrac{1}{\sqrt{51}}(1,5,5)^\top$}} \\
\end{tabular}}
\caption{The function $A_{\bftheta}^{\zba}(t)$ for $\zba=(a,a,a)^\top$,
$a>0$, is the 3D {Radon} transform of the cube.
Functions depending on $t$ 
for $\bftheta = \nicefrac{1}{\sqrt{3}} (1,1,1)^\top \in \sphere^2$ (purple) 
as reference, 
and varying $\bftheta\in\sphere^{2}$ in the degenerate (dashed gray) and general case (orange). 
}
\label{fig:radon3d_function}
\end{figure}

\begin{remark}[Optimal directions]
    Fig.~\ref{fig:radon3d_sphere}
    visualizes the spherical function 
    \[
        \sphere^{2} \ni \bftheta \mapsto A_{\bftheta}^{\zba}( t),
        \quad \textrm{for some} \quad  t \in \R,
    \]
    for $\zba = \nicefrac{1}{2}(1,1,1)^\top \in \R^3$.
    In \cite{Konig2021} it was shown that it has 12 global maxima
    $\nicefrac{1}{\sqrt{3}}(\pm1,\pm1,\pm1)^\top$
    for all $\nicefrac{1}{2\sqrt{3}} <  t \leq \nicefrac{\sqrt{3}}{2}$,
    which remain at least local maxima if $ t < \nicefrac{1}{2\sqrt{3}}$,
    see Fig.~\ref{fig:radon3d_sphere}~(b).
    Moreover,
    $\bftheta \mapsto A_{\bftheta}^{\zba}(0)$ has 12 global maxima \cite{Moody2013},
    namely $\bftheta\in\nicefrac{1}{\sqrt2} \{(\pm1,\pm1,0)^\top,(\pm1,0,\pm1)^\top,(0,\pm1,\pm1)^\top\}$,
    and 6 global minima $\bftheta\in\{\pm\e_i:i\in\ii 3\}$, 
    see Fig.~\ref{fig:radon3d_sphere}~(a).
\end{remark}

\begin{figure}\centering
\begin{minipage}{.49\linewidth}    
    \includegraphics[width=0.8\linewidth, clip=true, trim=10pt 80pt 50pt 50pt]{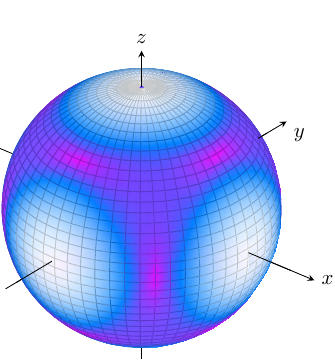}
    \scalebox{0.5}{\includegraphics[width=0.18\linewidth, clip=true, trim=180pt 2pt 2pt 2pt]{plots/sphere_diag.pdf}} \\
    \raisebox{5mm}[0mm][0mm]{(a)}
\end{minipage}
\begin{minipage}{.49\linewidth}    
    \includegraphics[width=0.8\linewidth, clip=true, trim=10pt 80pt 50pt 50pt]{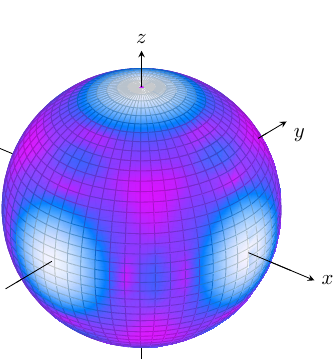}
    \scalebox{0.5}{\includegraphics[width=0.18\linewidth, clip=true, trim=180pt 2pt 2pt 2pt]{plots/sphere_neben.pdf}} \\
    \raisebox{5mm}[0mm][0mm]{(b)}
\end{minipage}
\caption{The function $\sphere^2 \ni \bftheta \mapsto A_{\bftheta}^{\zba}(t)$ 
for $\zba=(\nicefrac12,\nicefrac12,\nicefrac12)^\top$, 
which is the 3D {Radon} transform of the cube.
(a): $t = 0$.
(b): $t = \nicefrac{1}{7}$.
}
    \label{fig:radon3d_sphere}
\end{figure}

\subsection{Volume of a slab}
\label{sec:thm_slab}

Closely related to Theorem~\ref{thm:area},
we replace the hyperplane by a $d$-dimensional slab.
Denote by $V_{\bftheta}^{\zba}(t_1,t_2)$
the volume 
of the intersection of the hypercube $(-\zba,\zba]$ from \eqref{eq:box} 
with the \emph{slab}
\begin{equation}
    \label{eq:slab}
    S_{\bftheta}(t_1,t_2)
    \coloneqq
    \left\{\zbx\in\R^d : \langle \zbx,\bftheta\rangle \in [t_1,t_2]\right\}
\end{equation}
for $\bftheta\in\S^{d-1}$ and $t_1,t_2\in\R$.

\begin{corollary}
    \label{cor:slab}
    For $\zba\in\R^d_{>0}$, 
    $\bftheta\in\S^{d-1}$, 
    and $t_1,t_2 \in \R$ with $t_1<t_2$,
    it holds with the notation from Theorem~\ref{thm:area} 
    that
    \begin{equation} 
    \label{eq:volume-slab-general} 
    \begin{split}
        &\hspace{-5pt}V_{\bftheta}^{\zba}(t_1,t_2)
        =
        \frac{2^{d-\ell} P(\zba)}{P((\zba\varodot\bftheta)^\circ)\, \ell!}
        \sum_{\zbk\in\{-1,1\}^\ell} P(\zbk)
        \\
        &\times
        \Big( \,
        \big(t_2+\langle \zbk,(\zba\varodot\bftheta)^\circ\rangle\big)_+^{\ell}
        -
        \big(t_1+\langle \zbk,(\zba\varodot\bftheta)^\circ\rangle\big)_+^{\ell}
        \Big).
    \end{split}
    \end{equation}
\end{corollary}
\begin{proof}
    The slab $S_{\bftheta}(t_1,t_2)$ from \eqref{eq:slab}
    is the union of hyperplanes \eqref{eq:hyperplane}, 
    namely
    \begin{equation*}
        S_{\bftheta}(t_1,t_2)
        =
        \bigcup_{s\in[t_1,t_2]} H_{\bftheta}(s).
    \end{equation*}
    Since $\|\bftheta\|=1$, 
    we have
    \begin{align*}
        V_{\bftheta}^{\zba}(t_1,t_2)
        &=
        \int_{(-\zba,\zba] \cap S_{\bftheta}(t_1,t_2)} 1 \d \zbx
        =
        \int_{t_1}^{t_2}\!\!\!\! A_{\bftheta}^{\zba}(s) \d s.
    \end{align*}
    Integrating \eqref{eq:Radon-cube} for $t$,
    we obtain \eqref{eq:volume-slab-general}.
\end{proof}

The volume converges to the area $A_{\bftheta}^{\zba}$ if the width of the slab approaches zero, in particular it holds
\begin{equation}
    \label{eq:limSlap}
    \lim_{\varepsilon\searrow 0} 
    \; \tfrac{1}{2\varepsilon} V_{\bftheta}^{\zba}(t-\varepsilon,t+\varepsilon) 
    = A_{\bftheta}^{\zba}(t),
    \qquad \forall t\in\R.
\end{equation}

\begin{remark}[Related results]
\label{rem:related}
    Explicit volume formulas have been stated 
    in different forms in \cite{Polya1913,Barrow1979,Pournin2023}.
    However, 
    this literature considered the special cases of \eqref{eq:volume-slab-general},
    where either $\bftheta=\zbe/\sqrt{d}\in\S^{d-1}$ 
    and $\zba\in\R^d_{>0}$ is arbitrary \cite{Polya1913},
    or $\bftheta\in\S^{d-1}$ is arbitrary 
    and $\zba=\zbe$ is fixed \cite{Barrow1979,Konig2011,Pournin2023},
    where $\zbe \coloneqq (1,\dots,1)^\top\in\R^d$.
    These two cases are related via a coordinate transform:
    Let $\zba\in\R^d_{>0}$
    and note that \eqref{eq:volume-slab-general} remains valid for $\bftheta\in\R^d$.
    For $\zbx\in\R^d$, 
    we perform the coordinate transform $\zbx = \zba\varodot \zby$
    with $\d \zbx = P(\zba) \d \zba$. 
    Then $\zbx\in (-\zba,\zba]$ 
    if and only if $\zby\in (-\zbe,\zbe]$ 
    and
    \begin{equation*}
        \zbx \in H_{\zbe}(t)
        \Leftrightarrow
        \langle \zbe,\zbx \rangle = t
        \Leftrightarrow
        \langle \zba,\zby \rangle = t
        \Leftrightarrow
        \zby \in H_{\zba}(t),
    \end{equation*}
    for any $t \in \R$.
    Hence, 
    we have
    \begin{align}
        V_{\zbe}^{\zba}(t_1,t_2)
        =
        \int_{(-\zba,\zba] \cap S_{\zbe}(t_1,t_2)} 1 \d \zbx
        &=
        \int_{(-\zbe,\zbe] \cap S_{\zba}(t_1,t_2)} 1 \d \zby
        \\&=
        P(\zba)
        \cdot V_{\zba}^{\zbe}(t_1,t_2).
    \end{align}
    and therewith, by \eqref{eq:limSlap}, it follows
    \begin{equation*}
        A_{\zbe}^{\zba}(t) = P(\zba) A_{\zba}^{\zbe}(t),
        \qquad t \in \R.
    \end{equation*}
\end{remark}

\section{Discrete Radon Transform and Shape Matching}
\label{sec:discrete}

In this section,
we substantiate the theoretical findings 
from Theorem~\ref{thm:area}
with numerical applications.
Therefore, 
we consider the task 
of object recognition and classification 
in 3D utilizing the discrete Radon transform
of voxelized images outlined in Section~\ref{sec:discrete_model}.
First, in Section~\ref{sec:matching},
we perform feature extraction 
utilizing the {Radon} shape matching \cite{Daras2004}
based on the \textit{trace transform} \cite{Kadyrov2001}.
Second, in Section~\ref{sec:nrcdt},
we study the recently proposed
\textit{Normalized {Radon} Cumulative Distribution Transform} (NR-CDT) \cite{Beckmann2024a,Beckmann2025,Beckmann2025a}.
Finally, we consider the classification 
of affinely transformed 3D objects in Section~\ref{sec:affine_objects}.

\subsection{Datasets}
\label{sec:datasets}

The ModelNet10 dataset%
\footnote{Online available: \url{https://modelnet.cs.princeton.edu}} \cite{Wu2015} contains ten classes of 3D objects:
bathtubs, beds, chairs, desks, dressers, monitors, 
nightstands, sofas, tables, and toilets,
see Fig.~\ref{fig:modelnet10_1}~and~\ref{fig:modelnet10_2}.
\begin{figure}
    \vspace{-20pt}
    \resizebox{\linewidth}{!}{%
    \begin{tabular}{c c c c c c}
        \multicolumn{2}{c}{\includegraphics[width=0.5\linewidth, clip=true, trim=200pt 50pt 200pt 50pt]{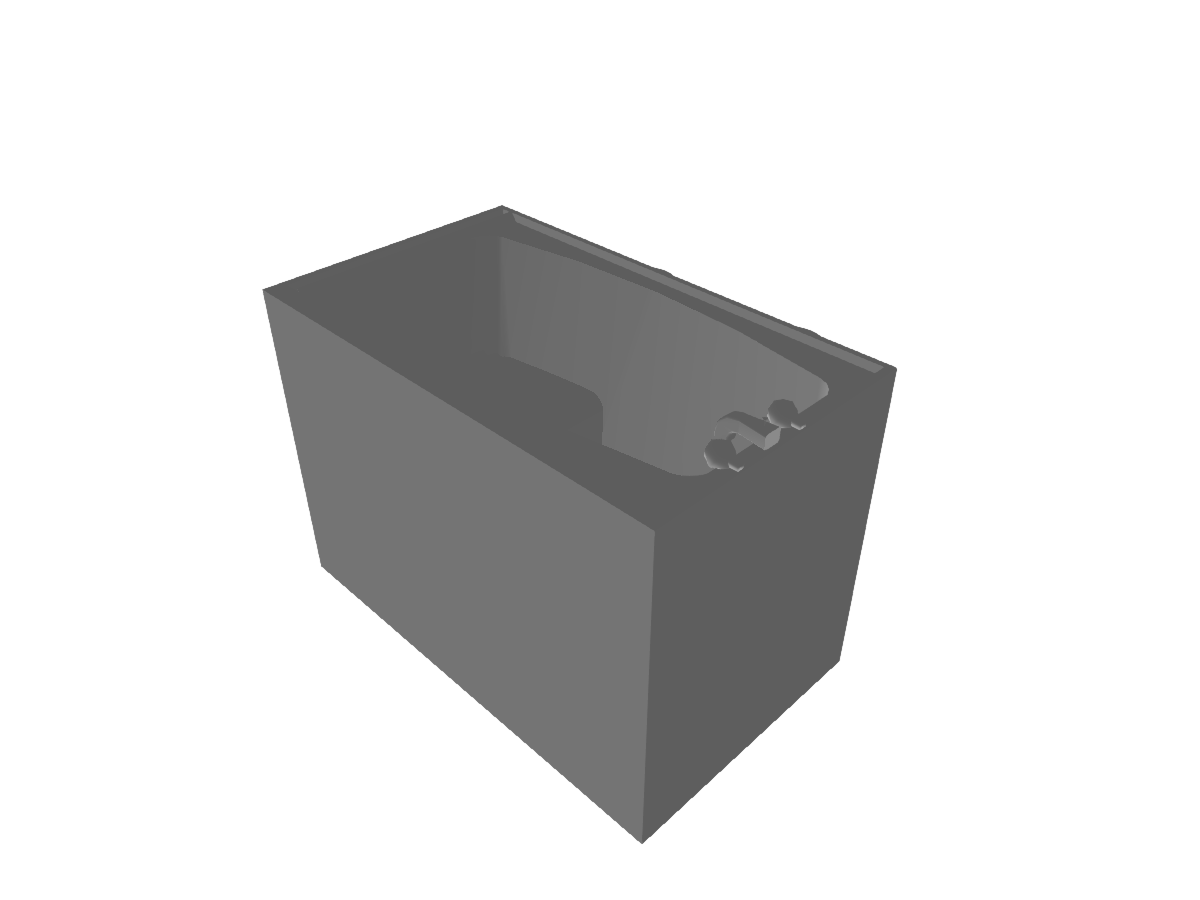}}
        &\multicolumn{2}{c}{\includegraphics[width=0.5\linewidth, clip=true, trim=240pt 20pt 160pt 80pt]{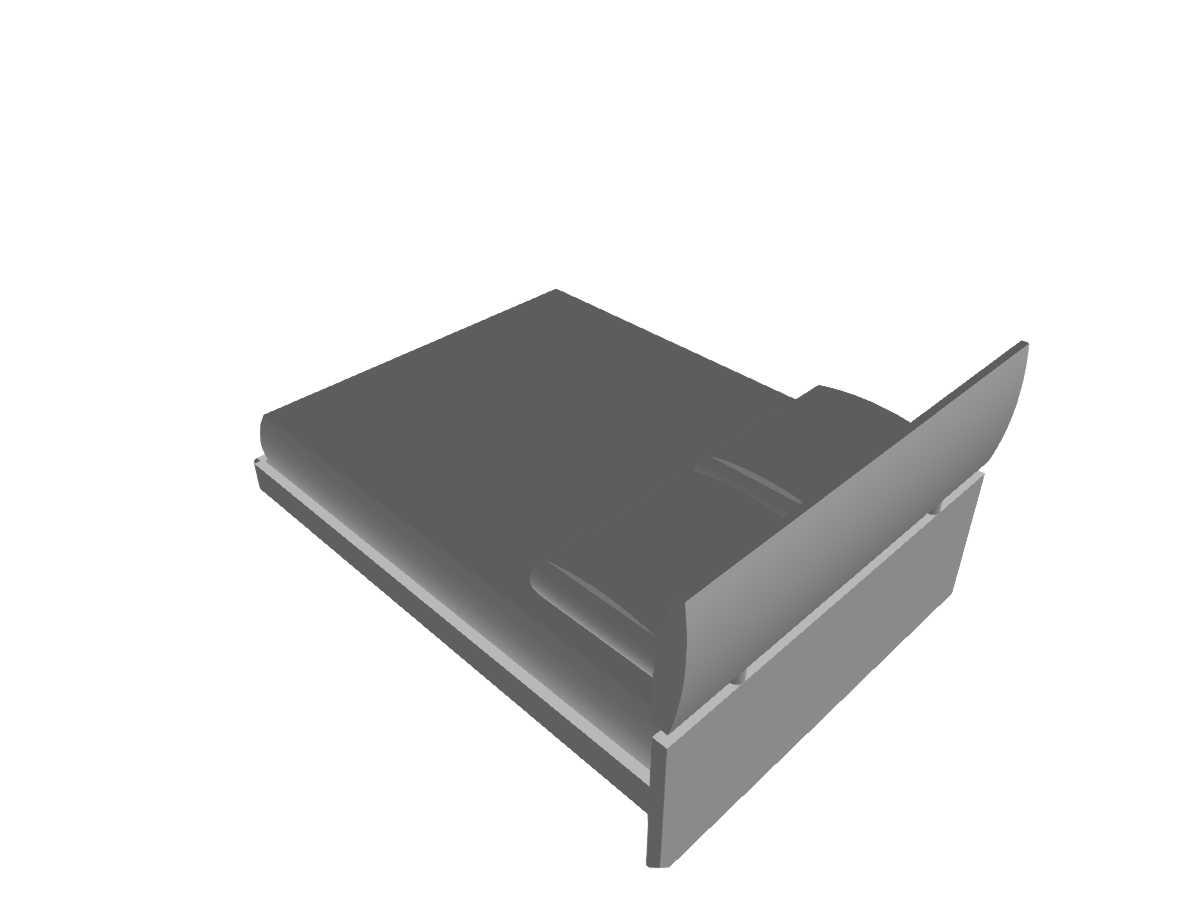}} \\
        \raisebox{6mm}[0mm][0mm]{(a)} & & \raisebox{6mm}[0mm][0mm]{(b)}\\[-5pt]
        \includegraphics[width=0.245\linewidth, clip=true, trim=200pt 50pt 200pt 50pt]{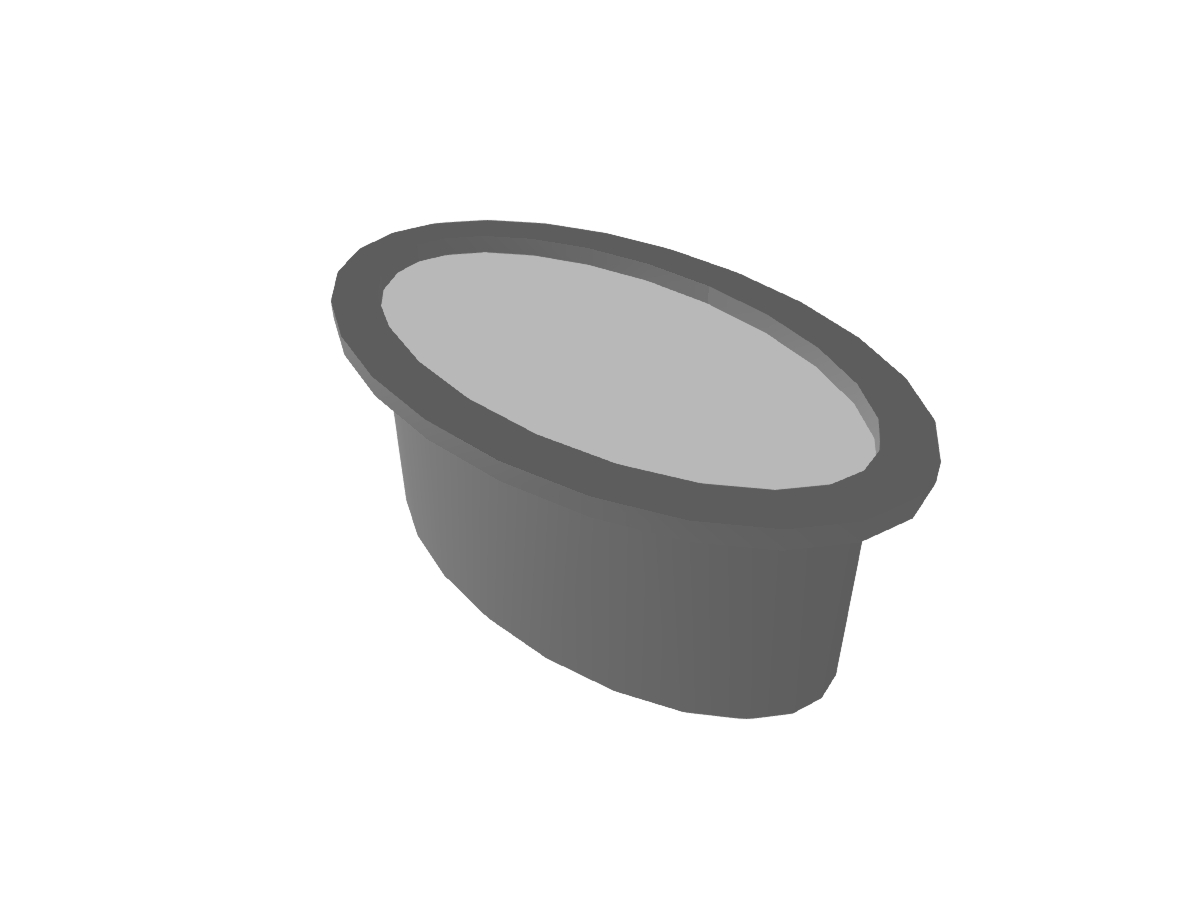}
        &\includegraphics[width=0.245\linewidth, clip=true, trim=300pt 50pt 100pt 50pt]{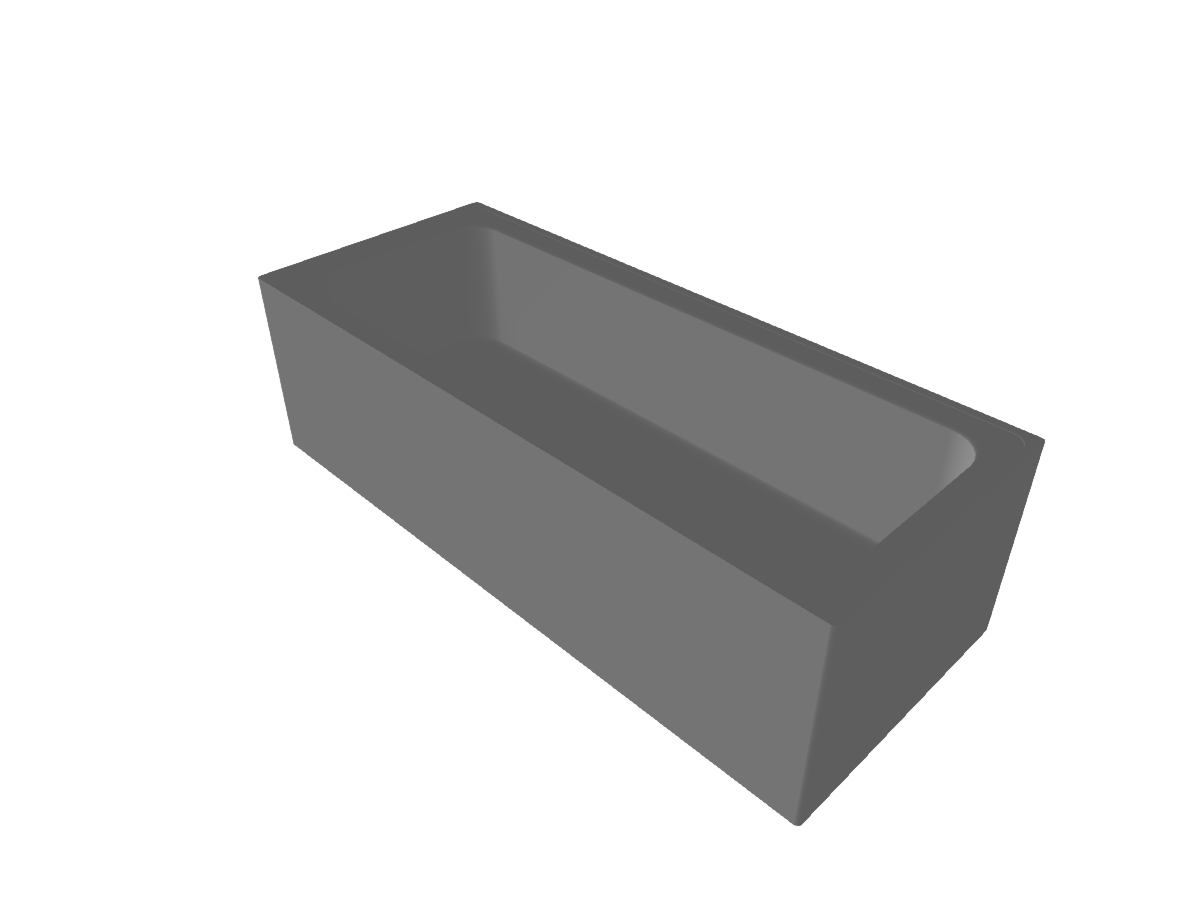} 
        &\includegraphics[width=0.245\linewidth, clip=true, trim=200pt 50pt 200pt 50pt]{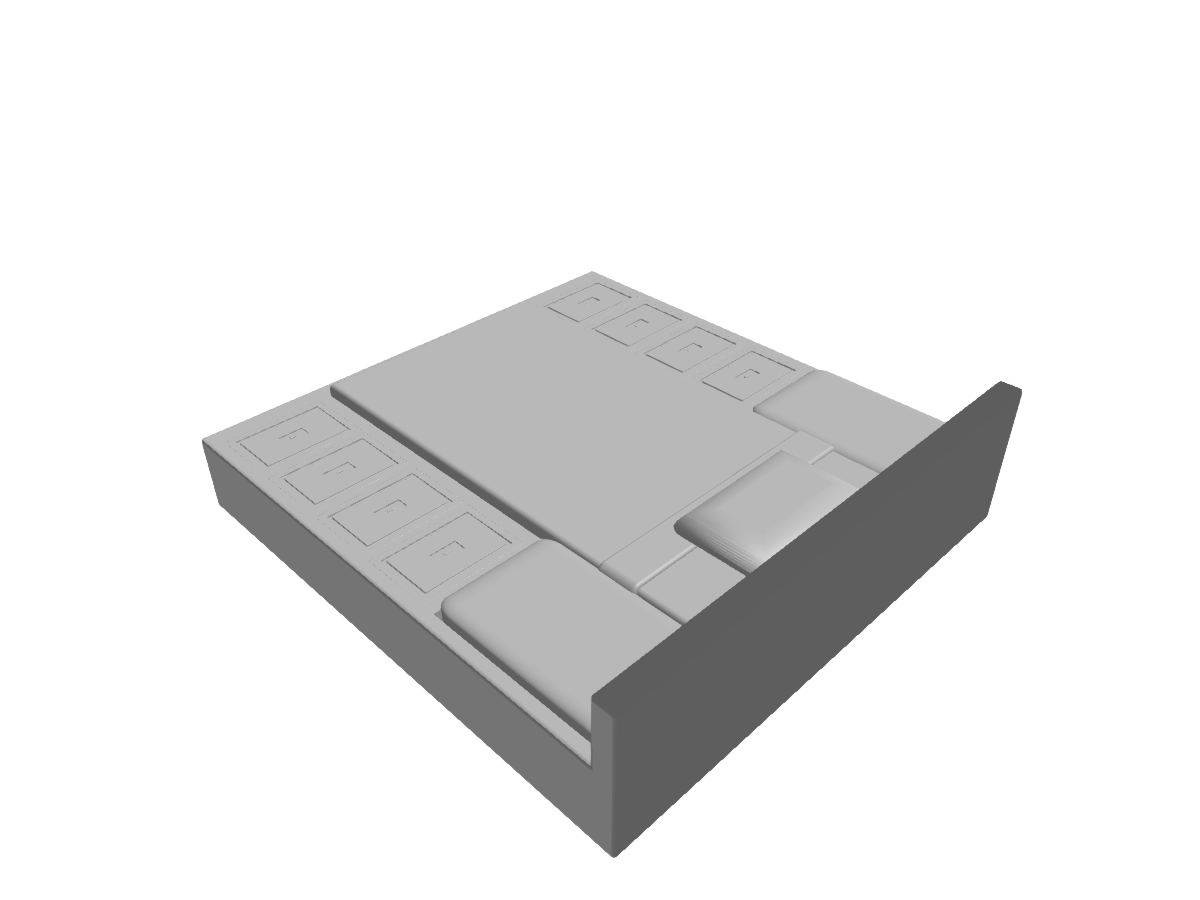}
        &\includegraphics[width=0.245\linewidth, clip=true, trim=200pt 50pt 200pt 50pt]{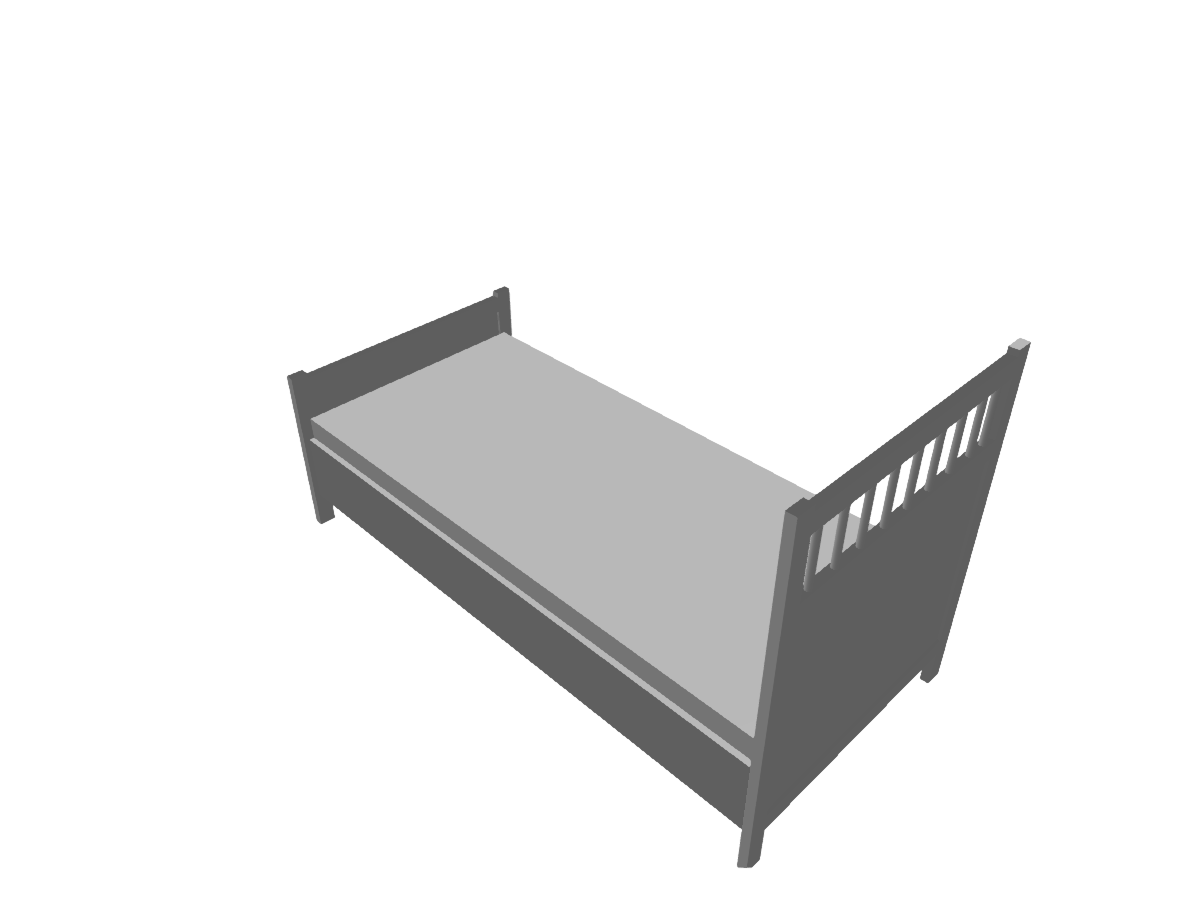} \\
        \includegraphics[width=0.245\linewidth, clip=true, trim=200pt 50pt 200pt 50pt]{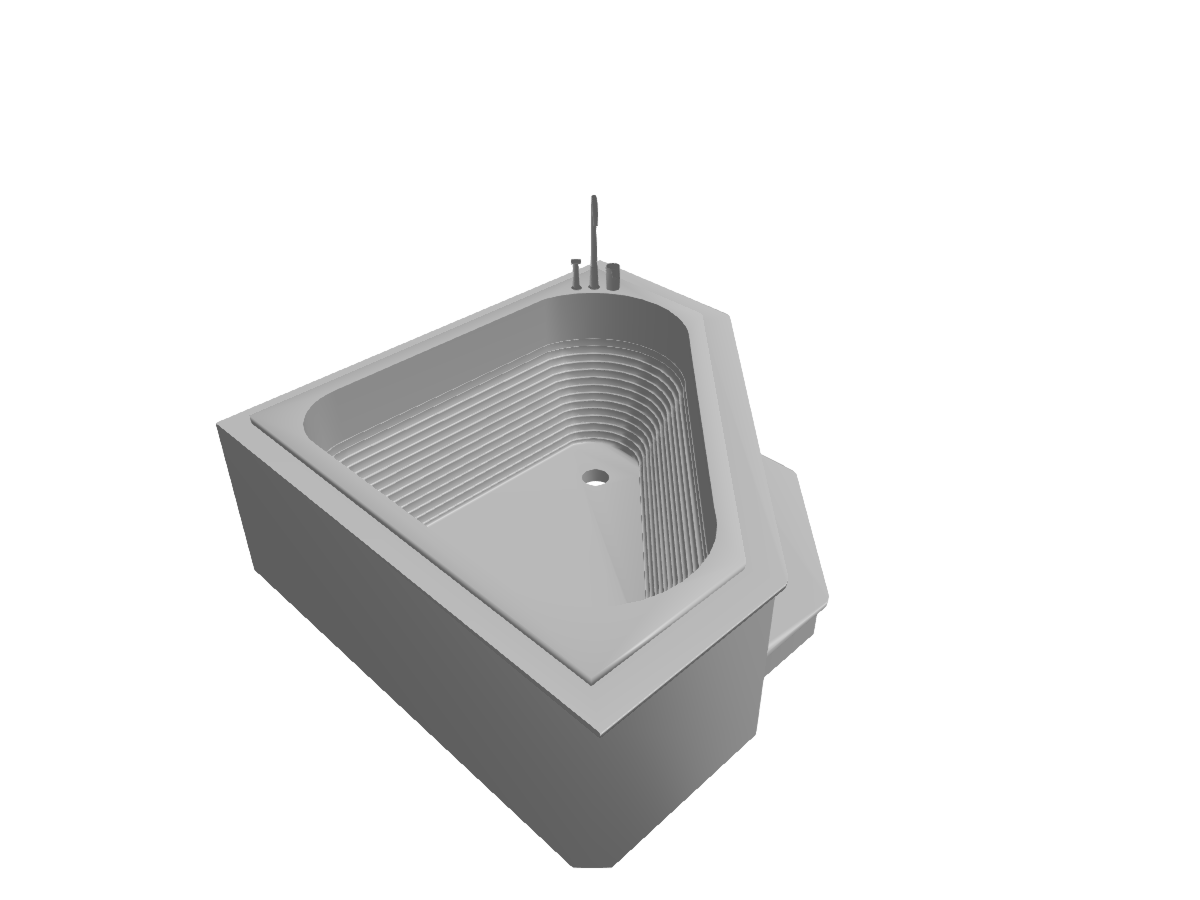}
        &\includegraphics[width=0.245\linewidth, clip=true, trim=200pt 50pt 200pt 50pt]{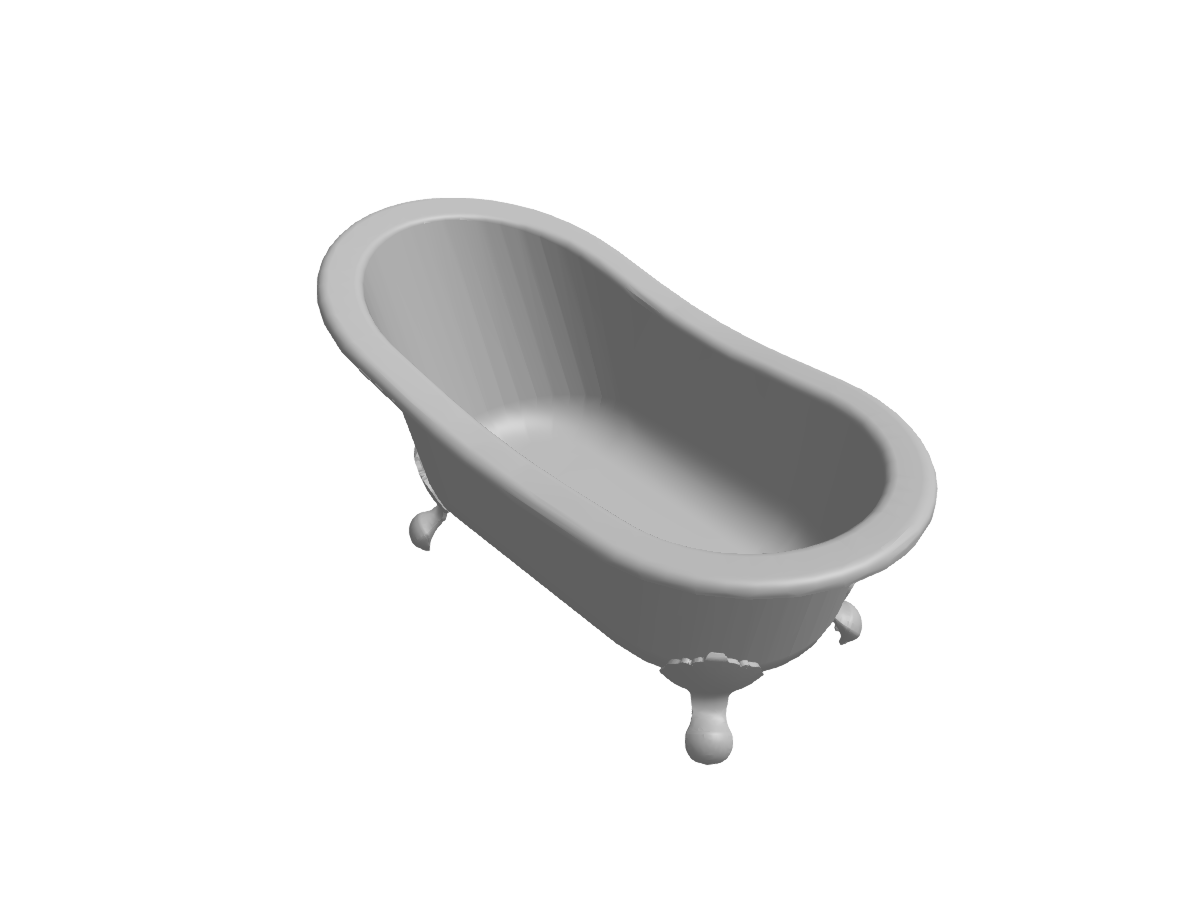} 
        &\includegraphics[width=0.245\linewidth, clip=true, trim=200pt 50pt 200pt 50pt]{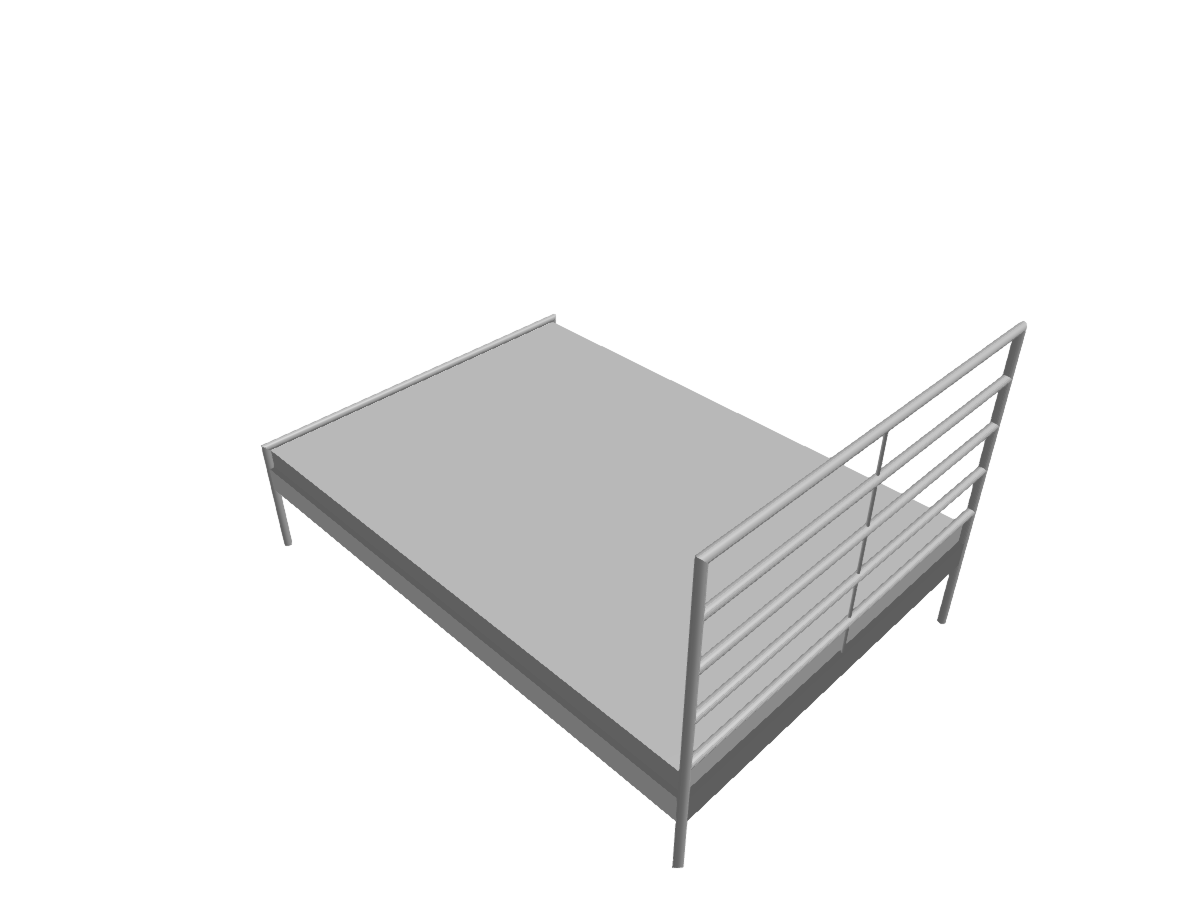}
        &\includegraphics[width=0.245\linewidth, clip=true, trim=200pt 50pt 200pt 50pt]{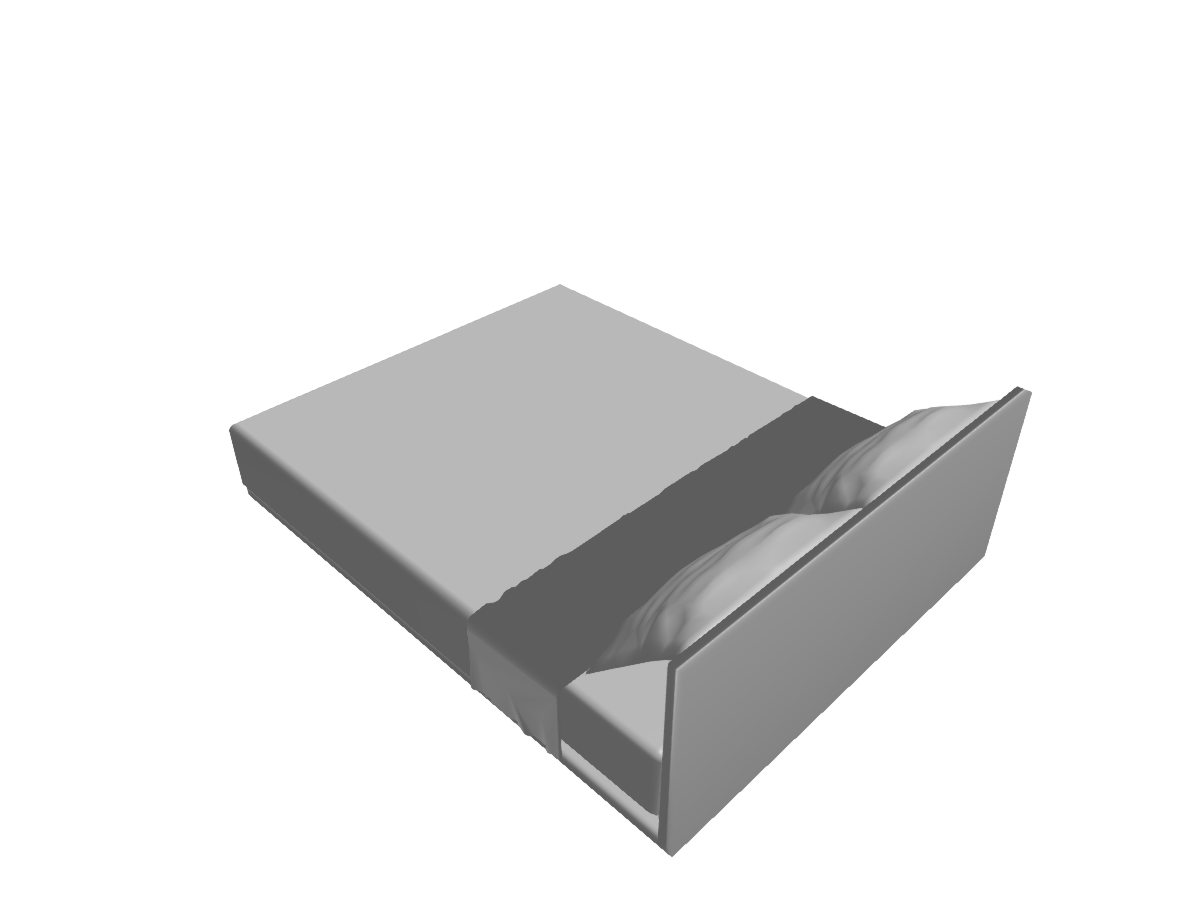} \\
        \multicolumn{2}{c}{\includegraphics[width=0.5\linewidth, clip=true, trim=200pt 50pt 200pt 50pt]{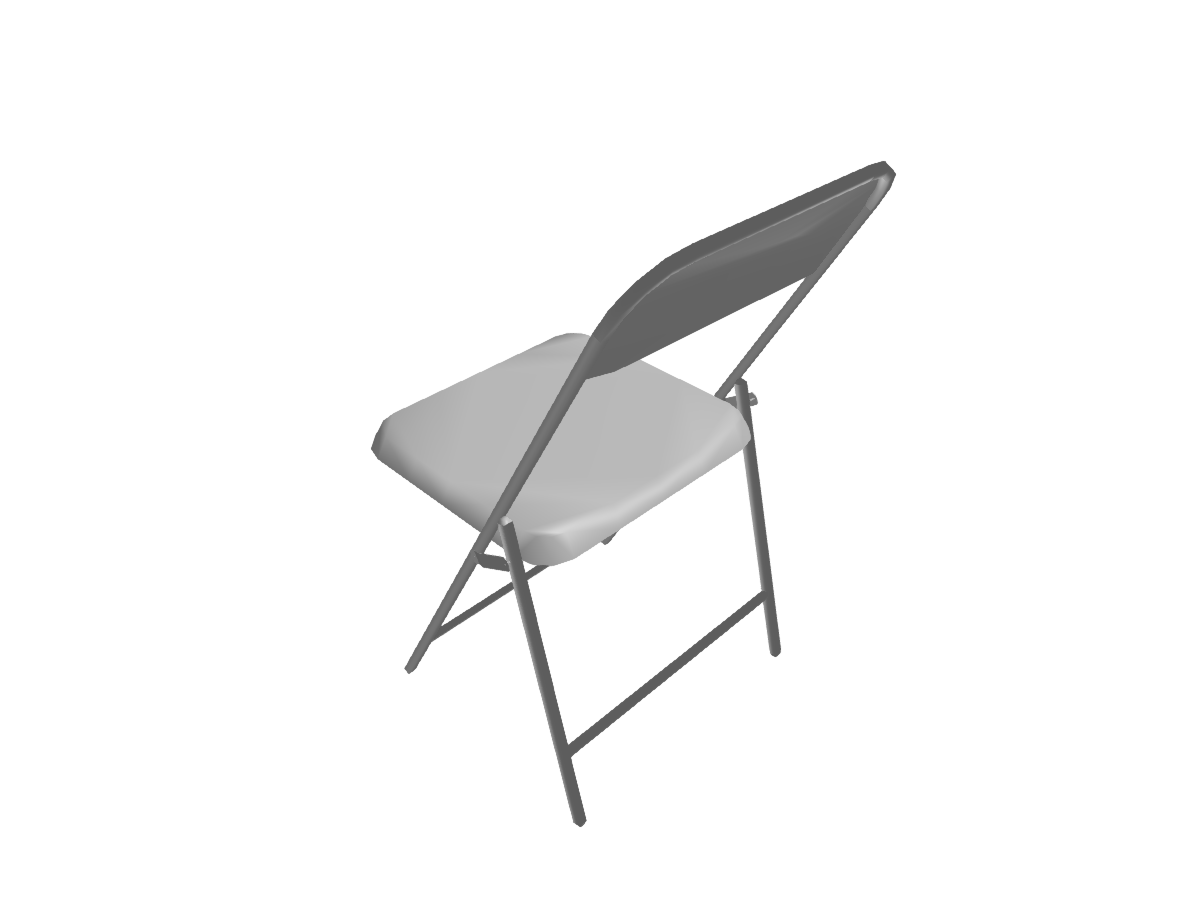}} 
        &\multicolumn{2}{c}{\includegraphics[width=0.5\linewidth, clip=true, trim=250pt 50pt 150pt 50pt]{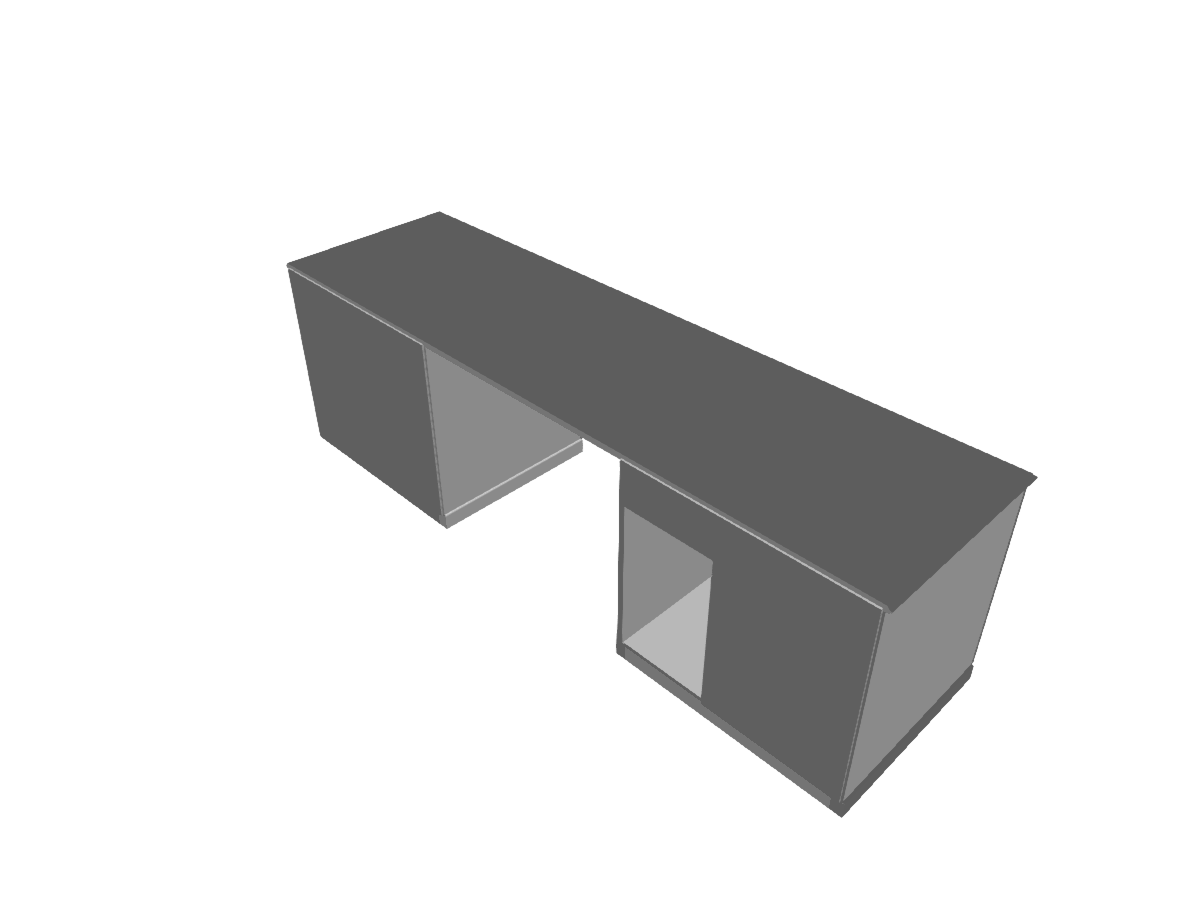}} \\
        \raisebox{6mm}[0mm][0mm]{(c)} & & \raisebox{6mm}[0mm][0mm]{(d)}\\[-5pt]
        \includegraphics[width=0.245\linewidth, clip=true, trim=200pt 50pt 200pt 50pt]{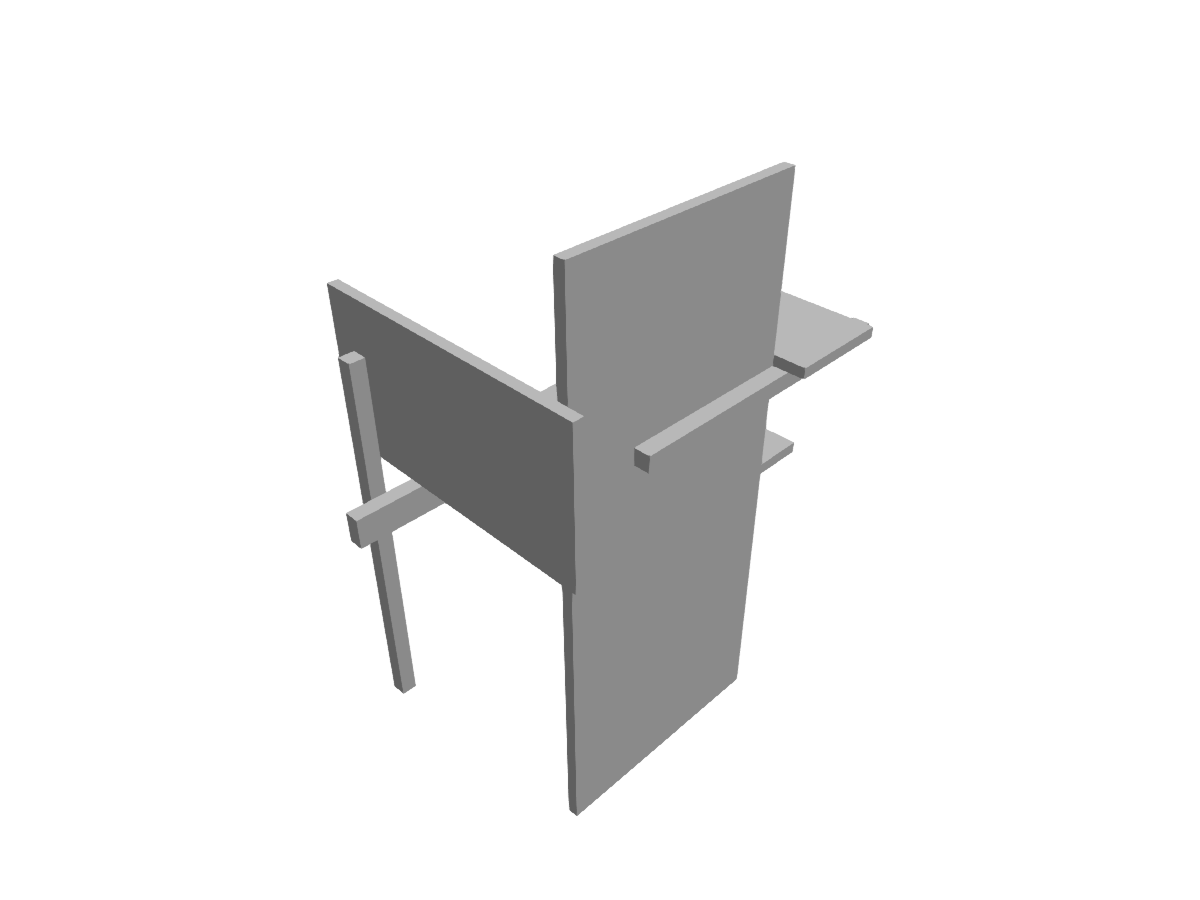}
        &\includegraphics[width=0.245\linewidth, clip=true, trim=200pt 50pt 200pt 50pt]{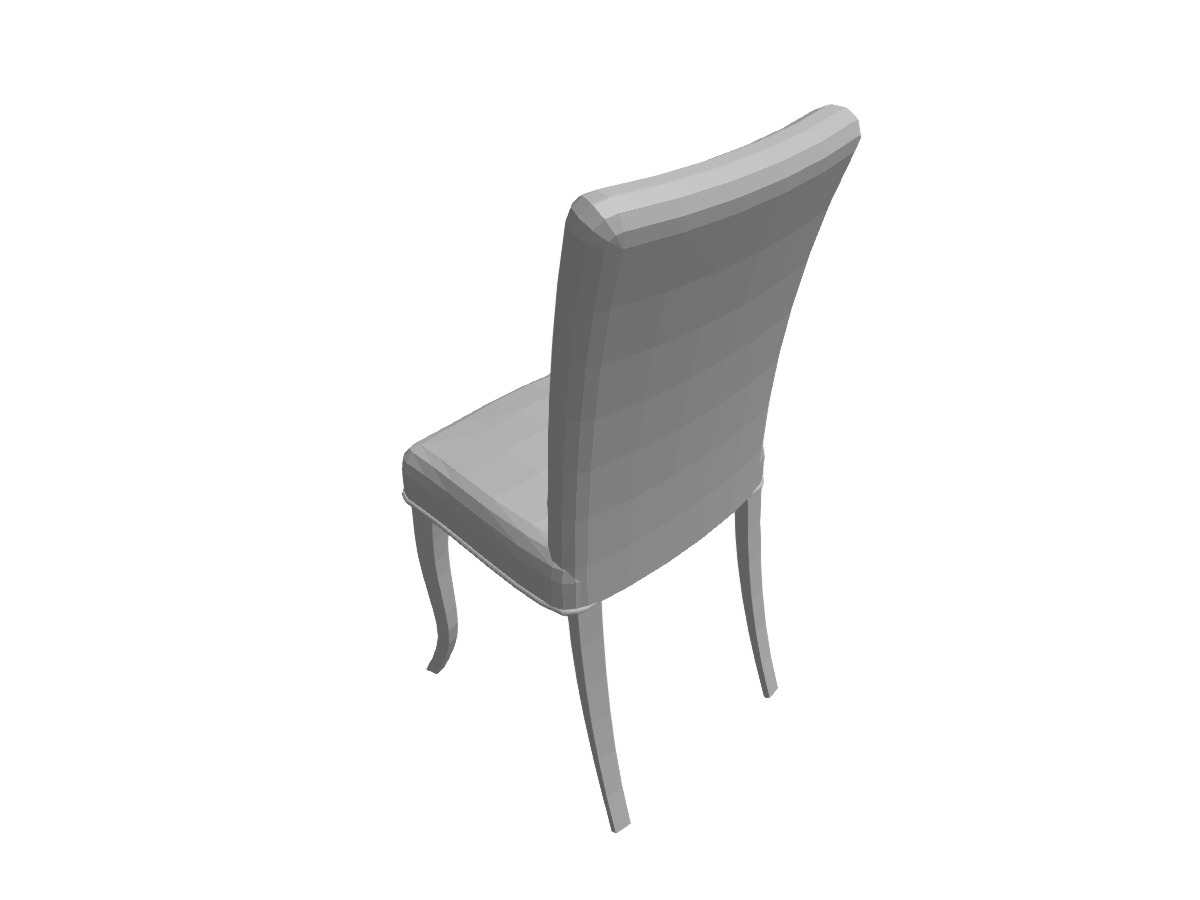} 
        &\includegraphics[width=0.245\linewidth, clip=true, trim=200pt 50pt 200pt 50pt]{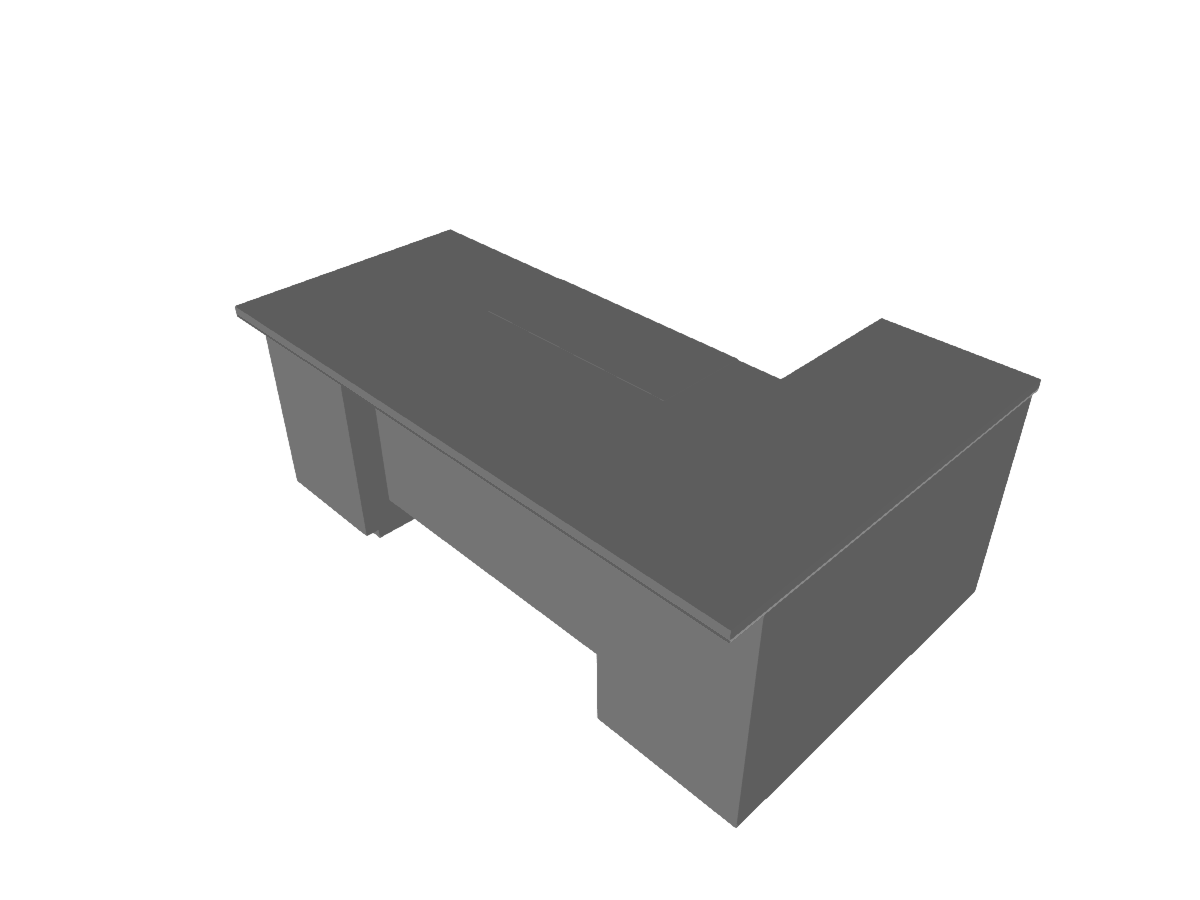}
        &\includegraphics[width=0.245\linewidth, clip=true, trim=200pt 50pt 200pt 50pt]{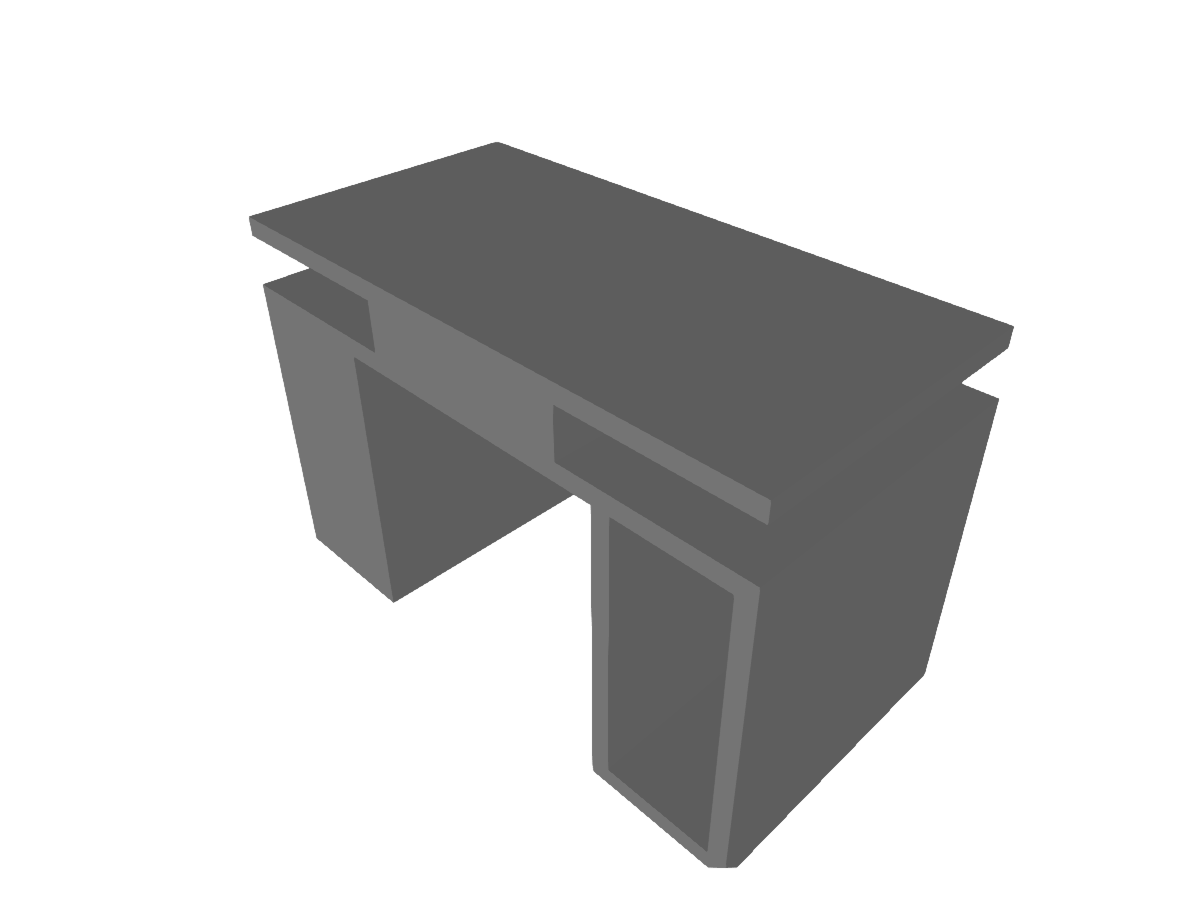} \\
        \includegraphics[width=0.245\linewidth, clip=true, trim=200pt 50pt 200pt 50pt]{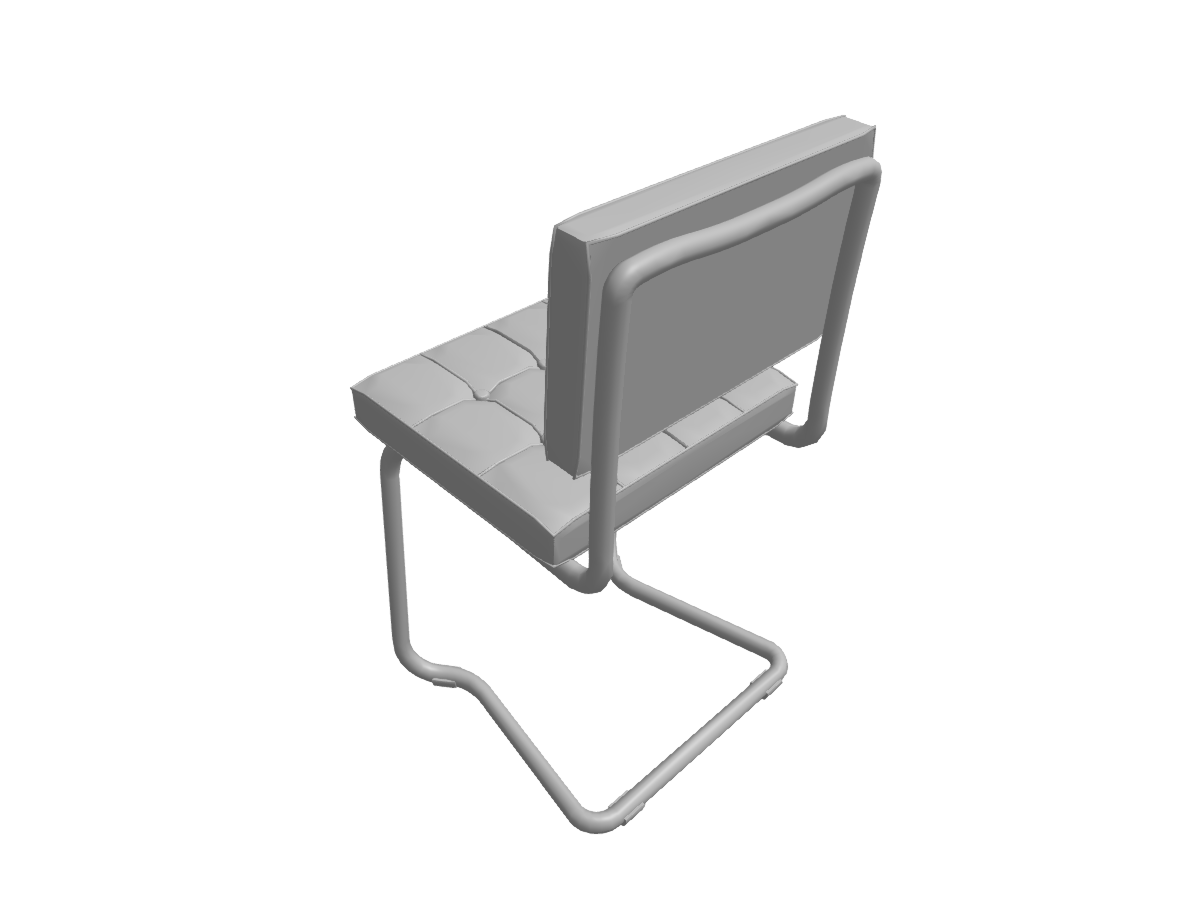}
        &\includegraphics[width=0.245\linewidth, clip=true, trim=200pt 50pt 200pt 50pt]{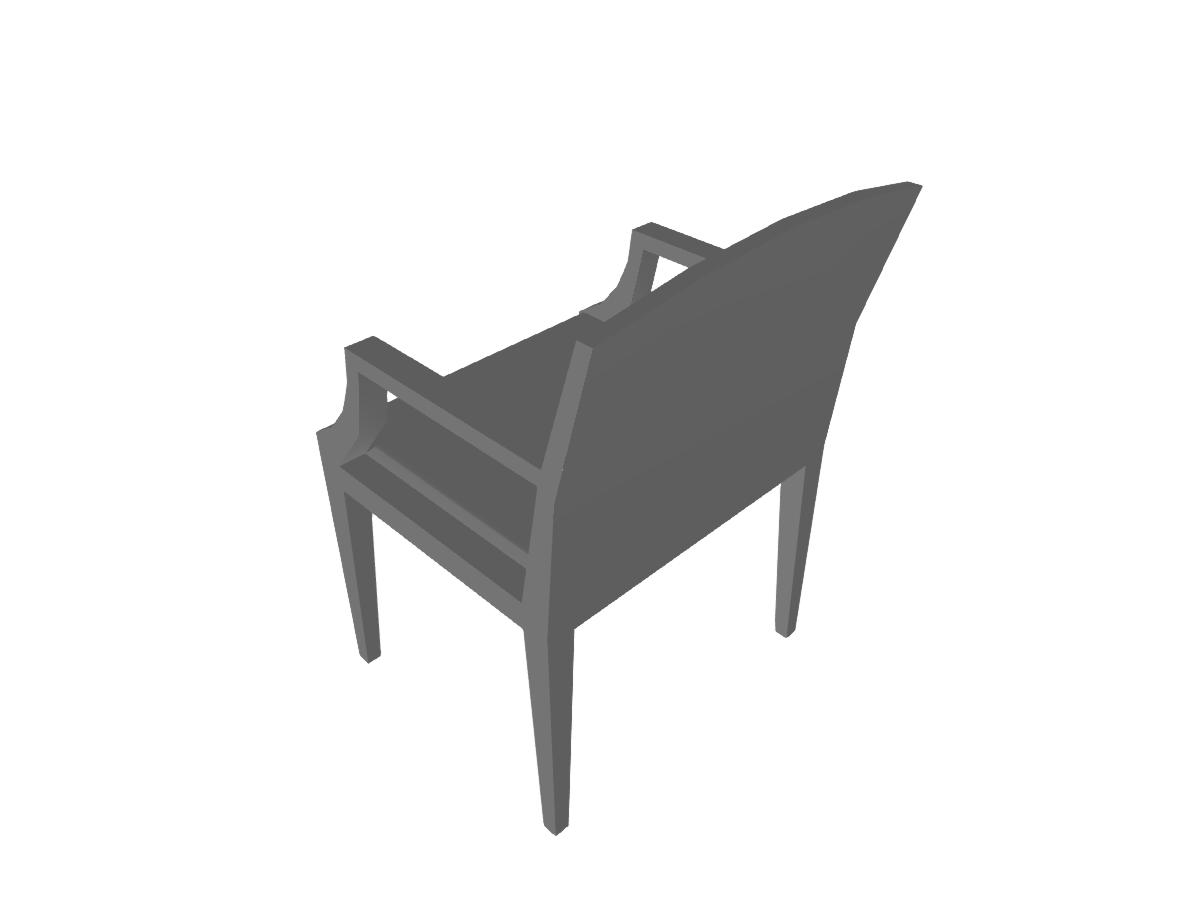} 
        &\includegraphics[width=0.245\linewidth, clip=true, trim=200pt 50pt 200pt 50pt]{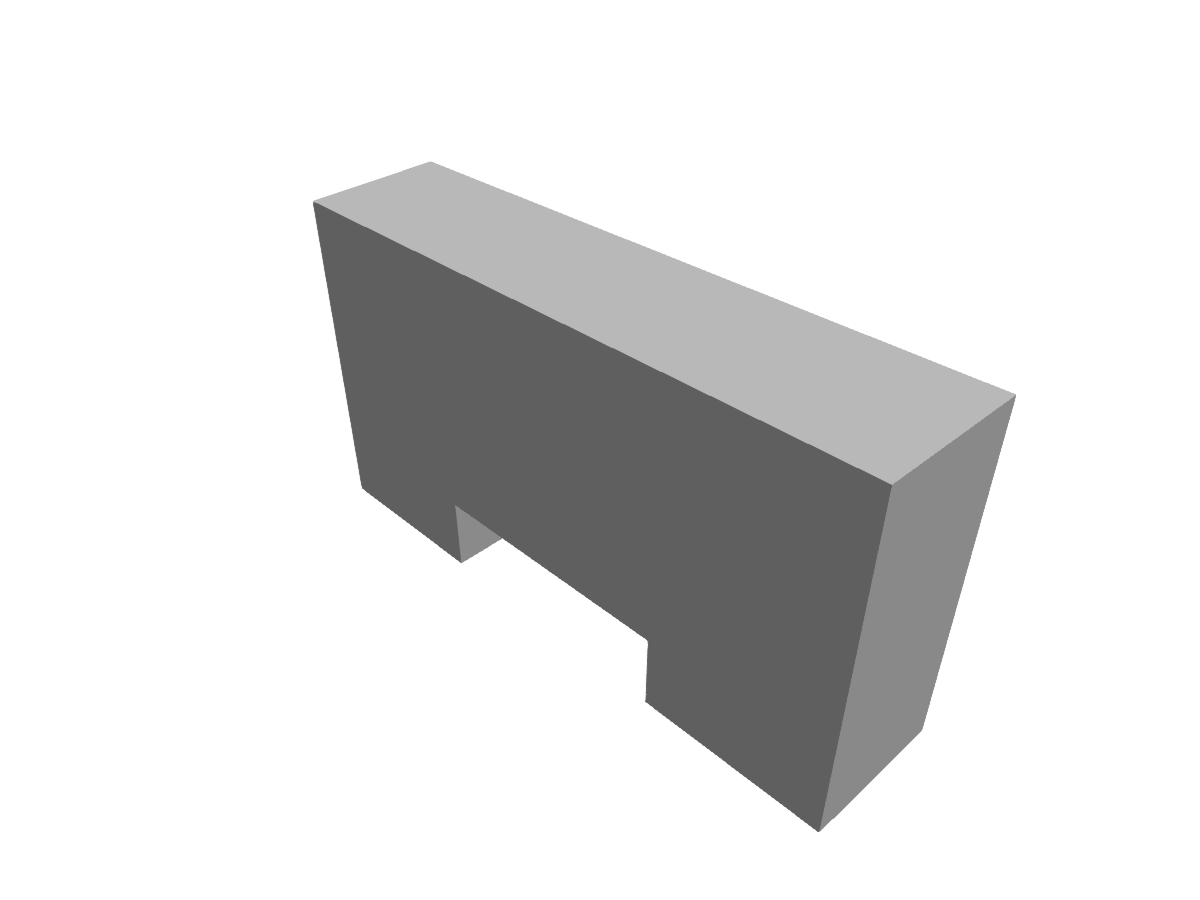}
        &\includegraphics[width=0.245\linewidth, clip=true, trim=200pt 50pt 200pt 50pt]{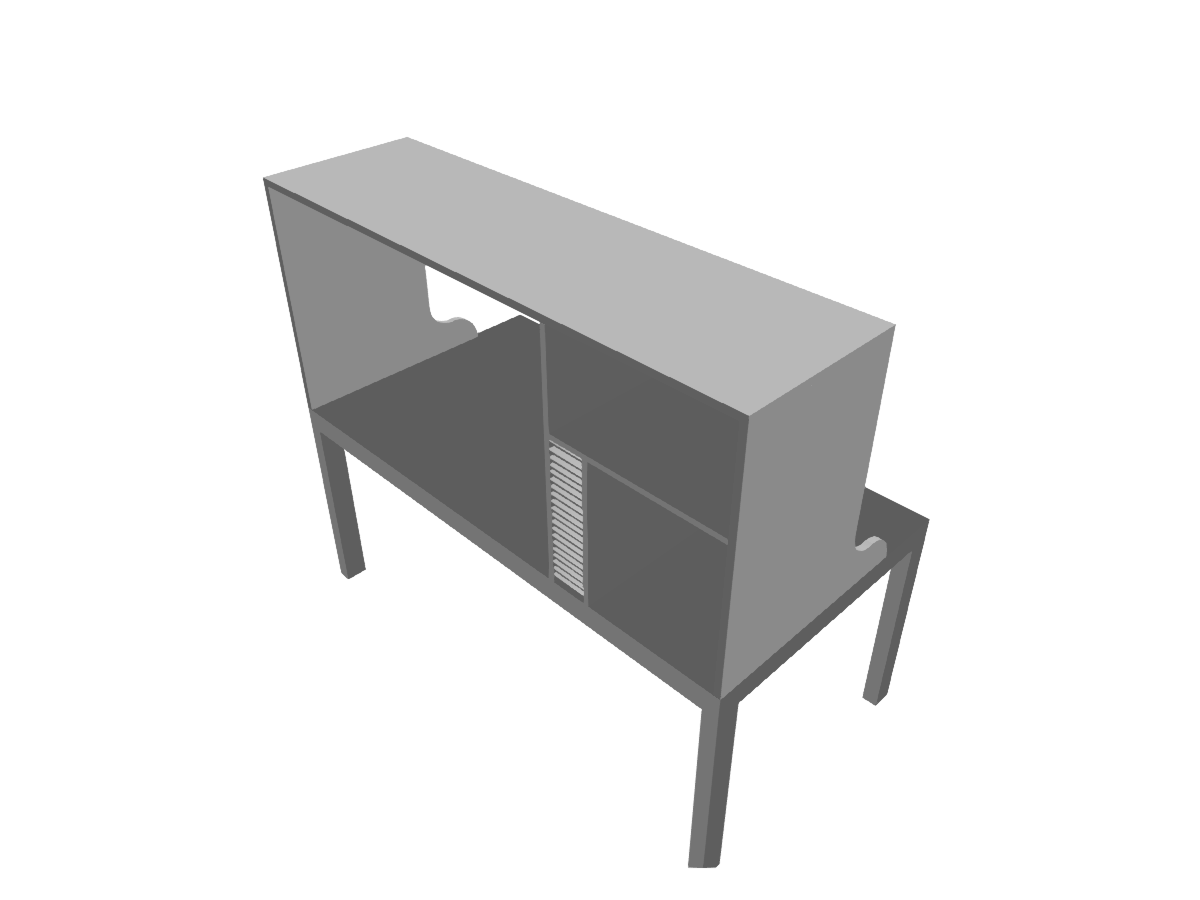} \\
        \multicolumn{2}{c}{\includegraphics[width=0.5\linewidth, clip=true, trim=200pt 50pt 200pt 0pt]{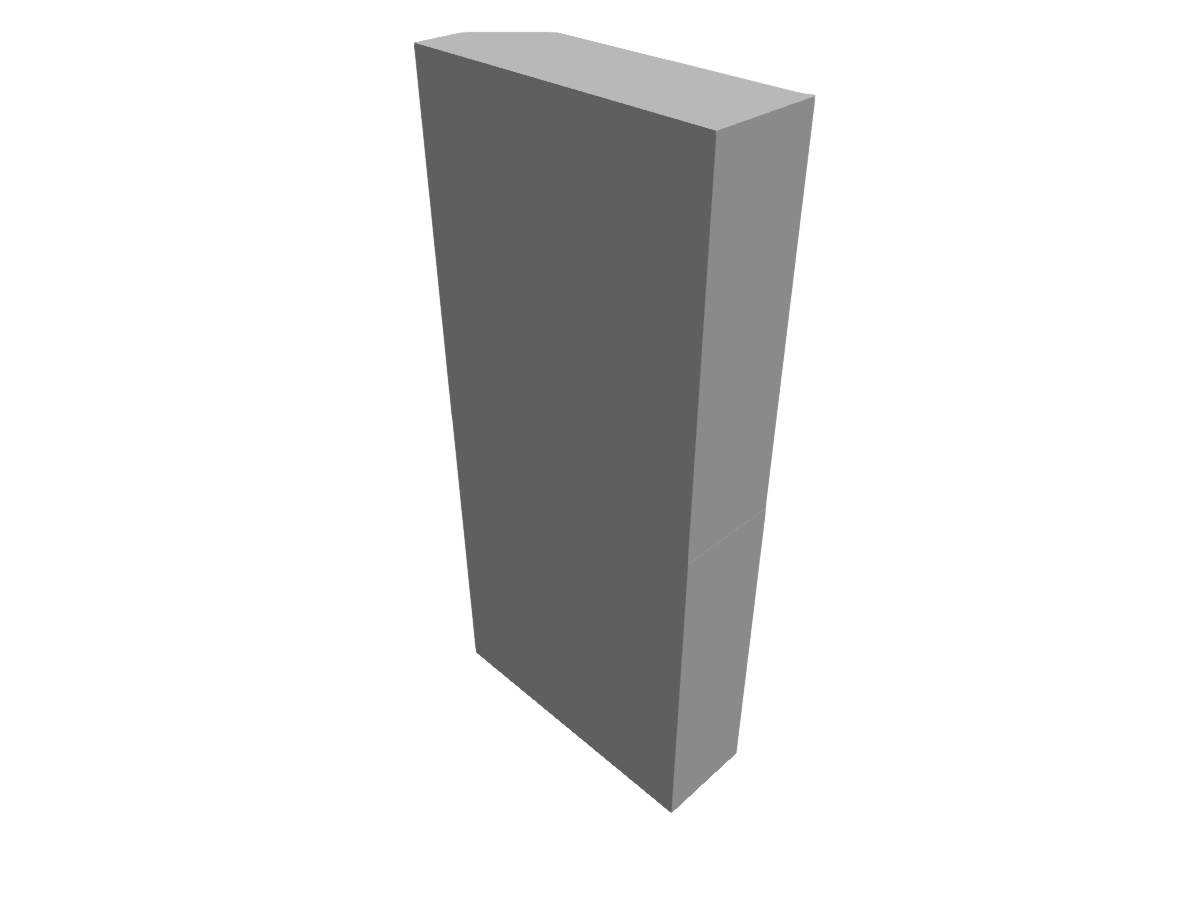}}
        &\multicolumn{2}{c}{\includegraphics[width=0.5\linewidth, clip=true, trim=200pt 50pt 200pt 50pt]{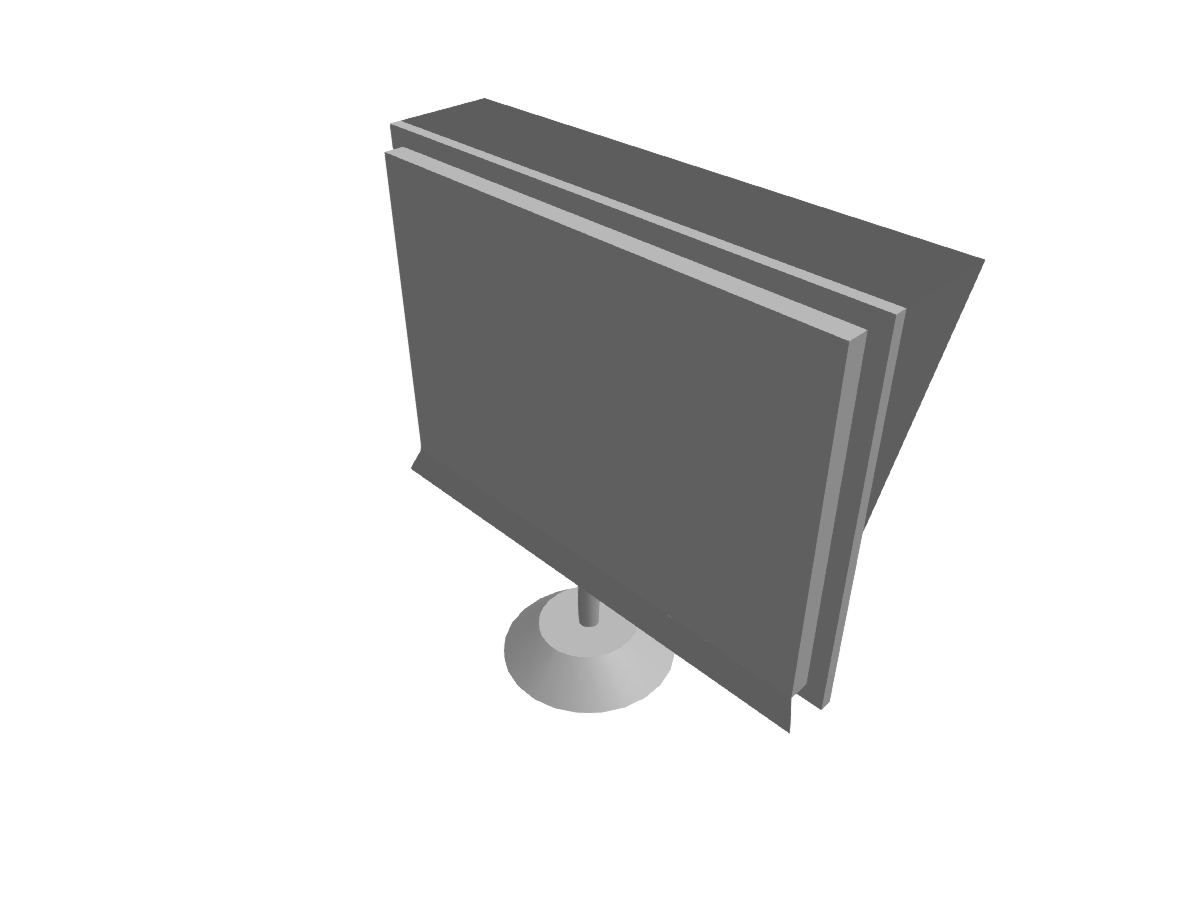}} \\
        \raisebox{6mm}[0mm][0mm]{(e)} & & \raisebox{6mm}[0mm][0mm]{(f)}\\[-5pt]
        \includegraphics[width=0.245\linewidth, clip=true, trim=200pt 50pt 200pt 50pt]{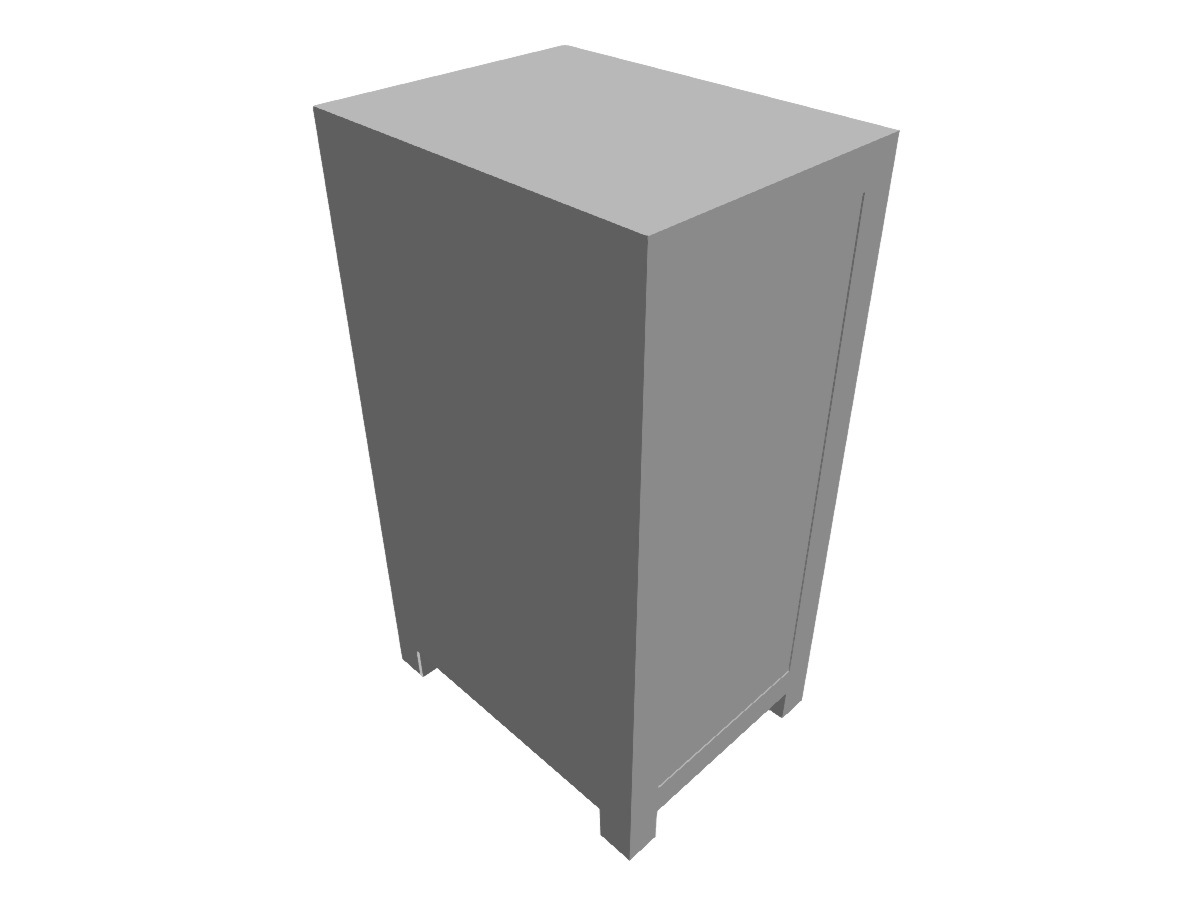}
        &\includegraphics[width=0.245\linewidth, clip=true, trim=200pt 50pt 200pt 50pt]{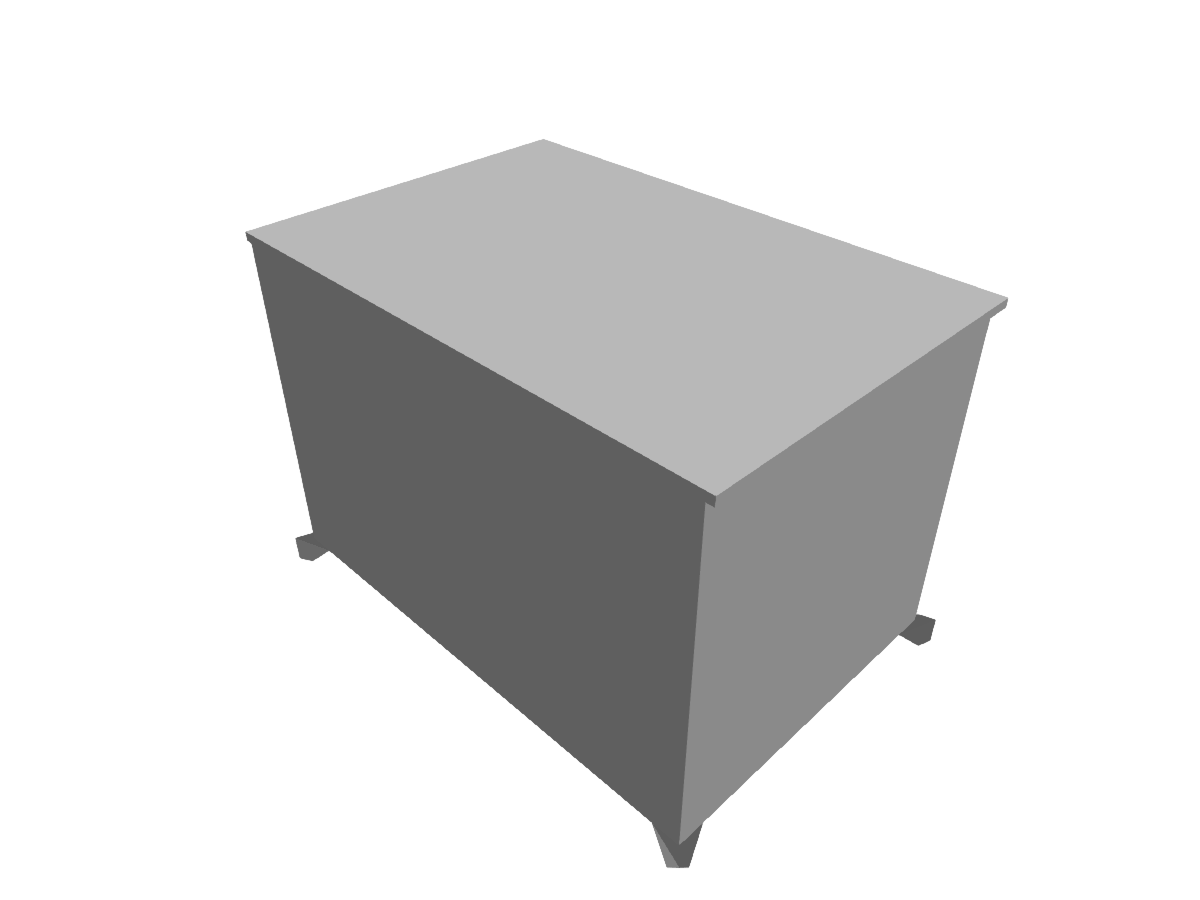}
        &\includegraphics[width=0.245\linewidth, clip=true, trim=200pt 50pt 200pt 50pt]{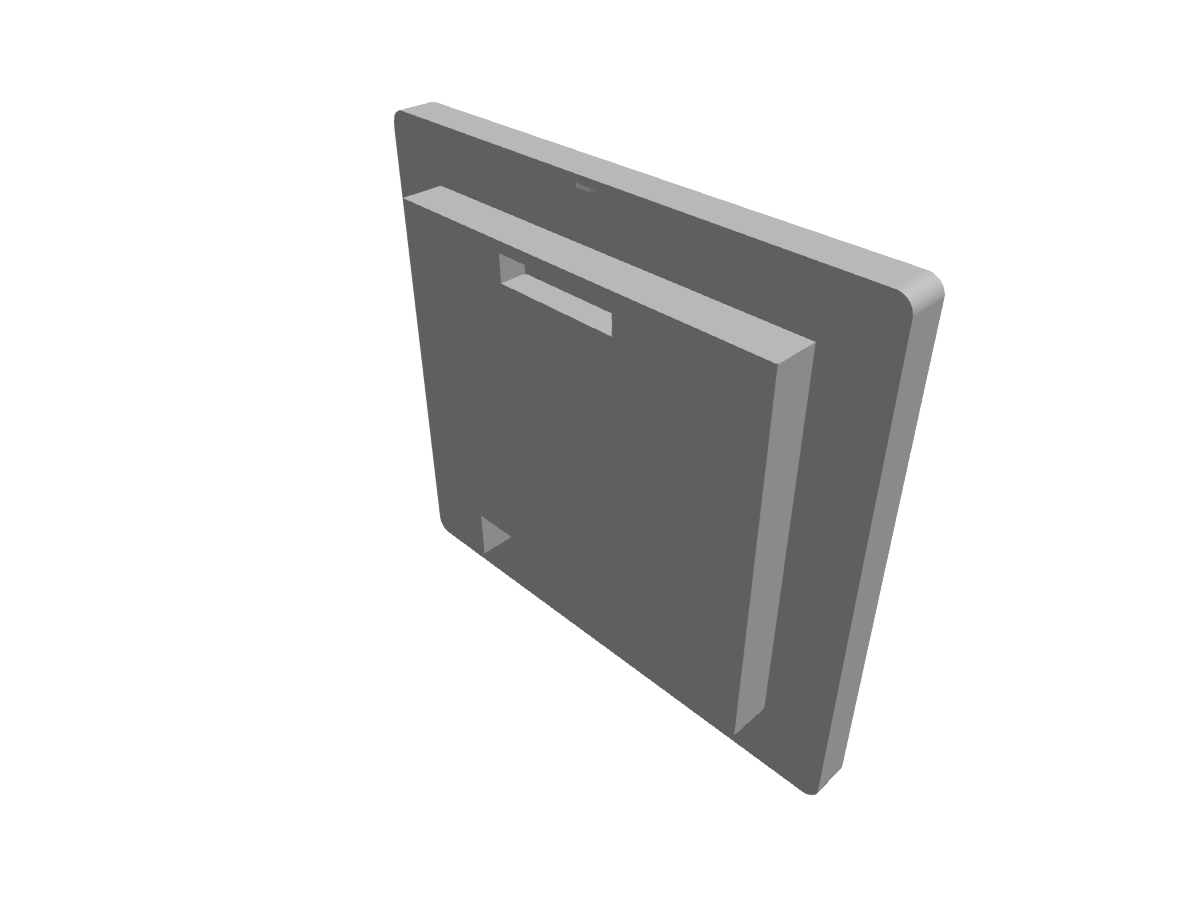}
        &\includegraphics[width=0.245\linewidth, clip=true, trim=200pt 50pt 200pt 50pt]{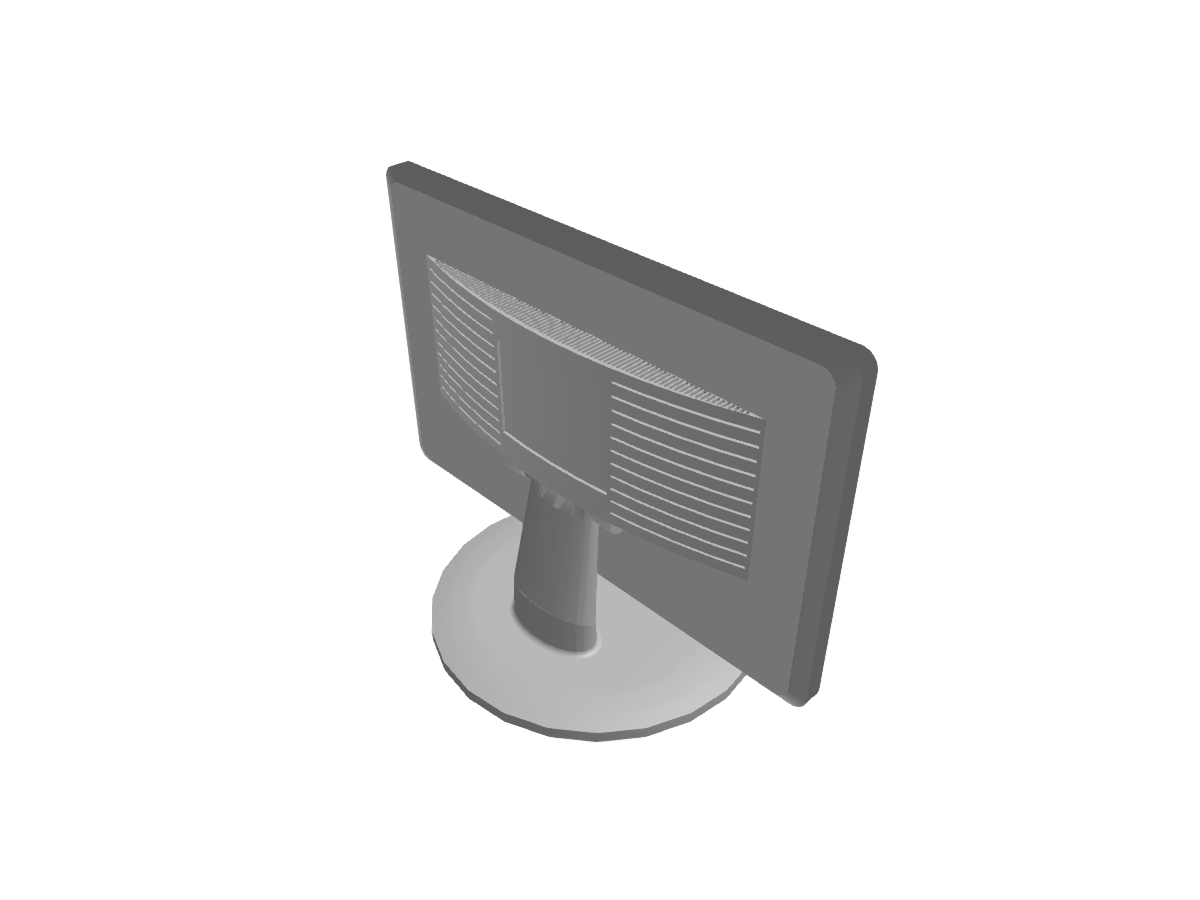} \\
        \includegraphics[width=0.245\linewidth, clip=true, trim=200pt 50pt 200pt 50pt]{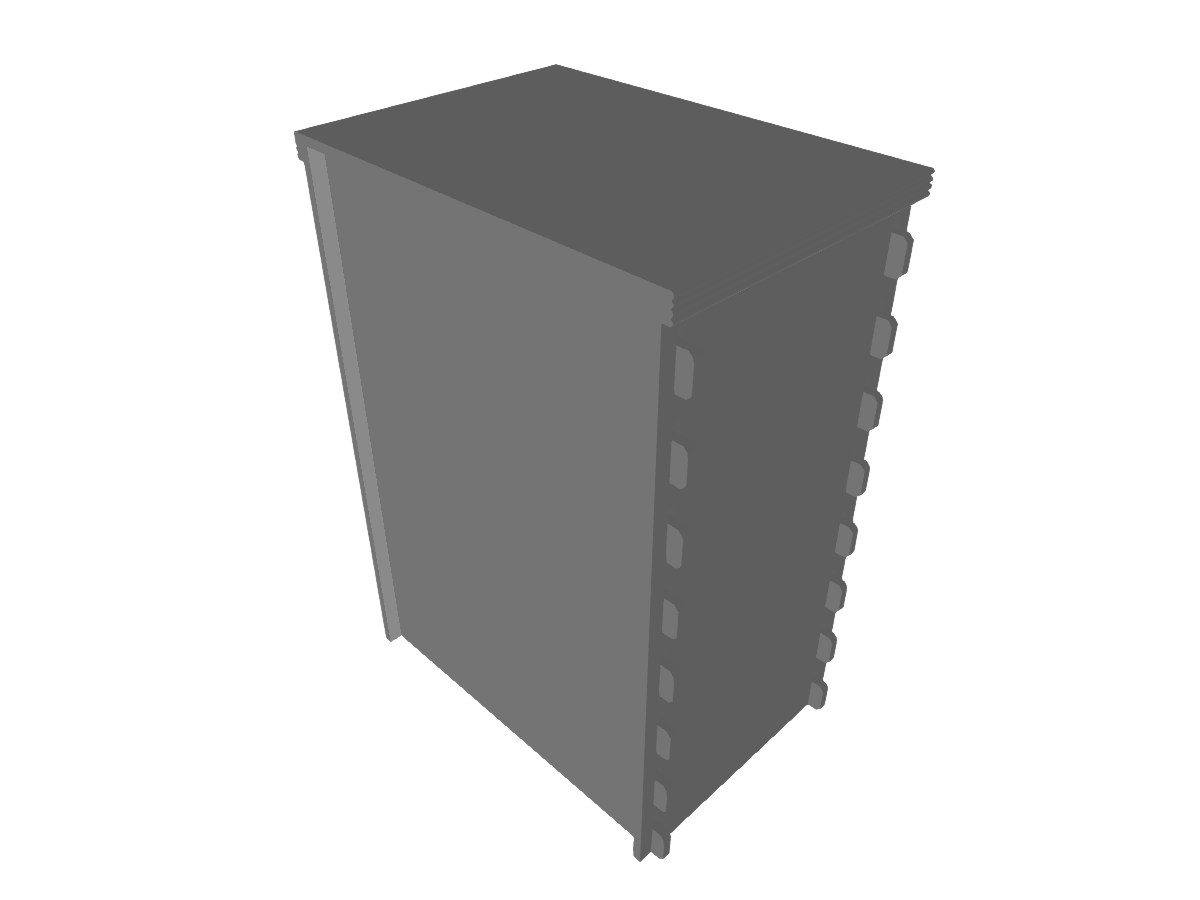}
        &\includegraphics[width=0.245\linewidth, clip=true, trim=200pt 50pt 200pt 50pt]{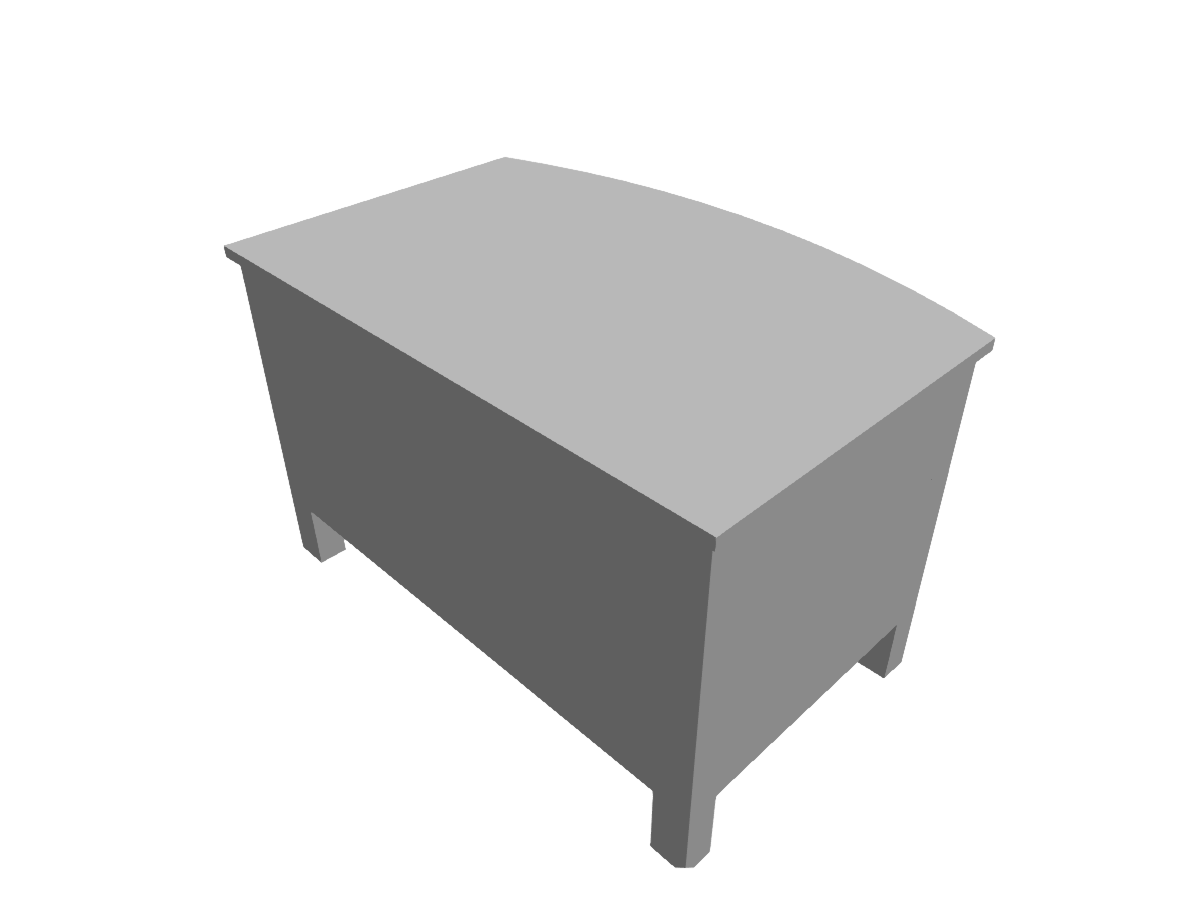}
        &\includegraphics[width=0.245\linewidth, clip=true, trim=200pt 50pt 200pt 50pt]{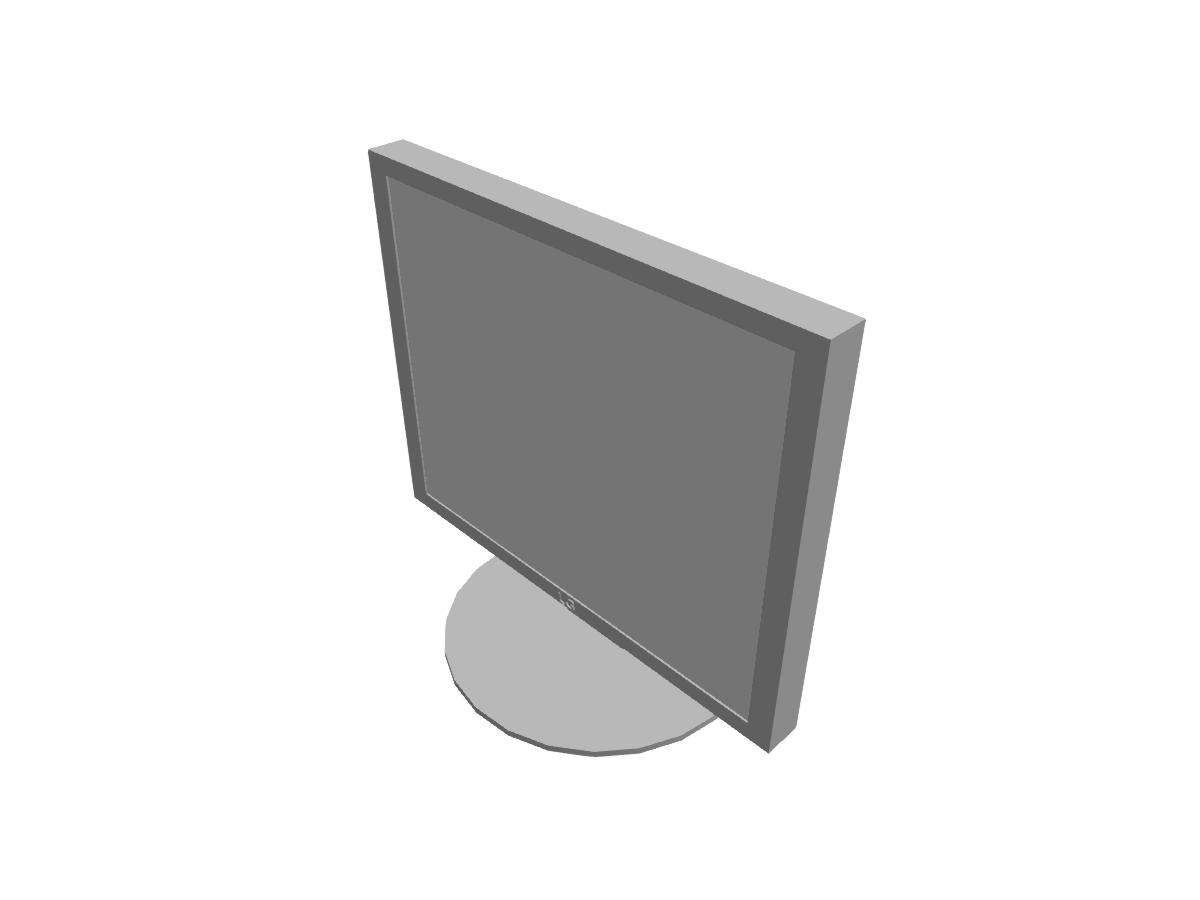}
        &\includegraphics[width=0.245\linewidth, clip=true, trim=200pt 50pt 200pt 50pt]{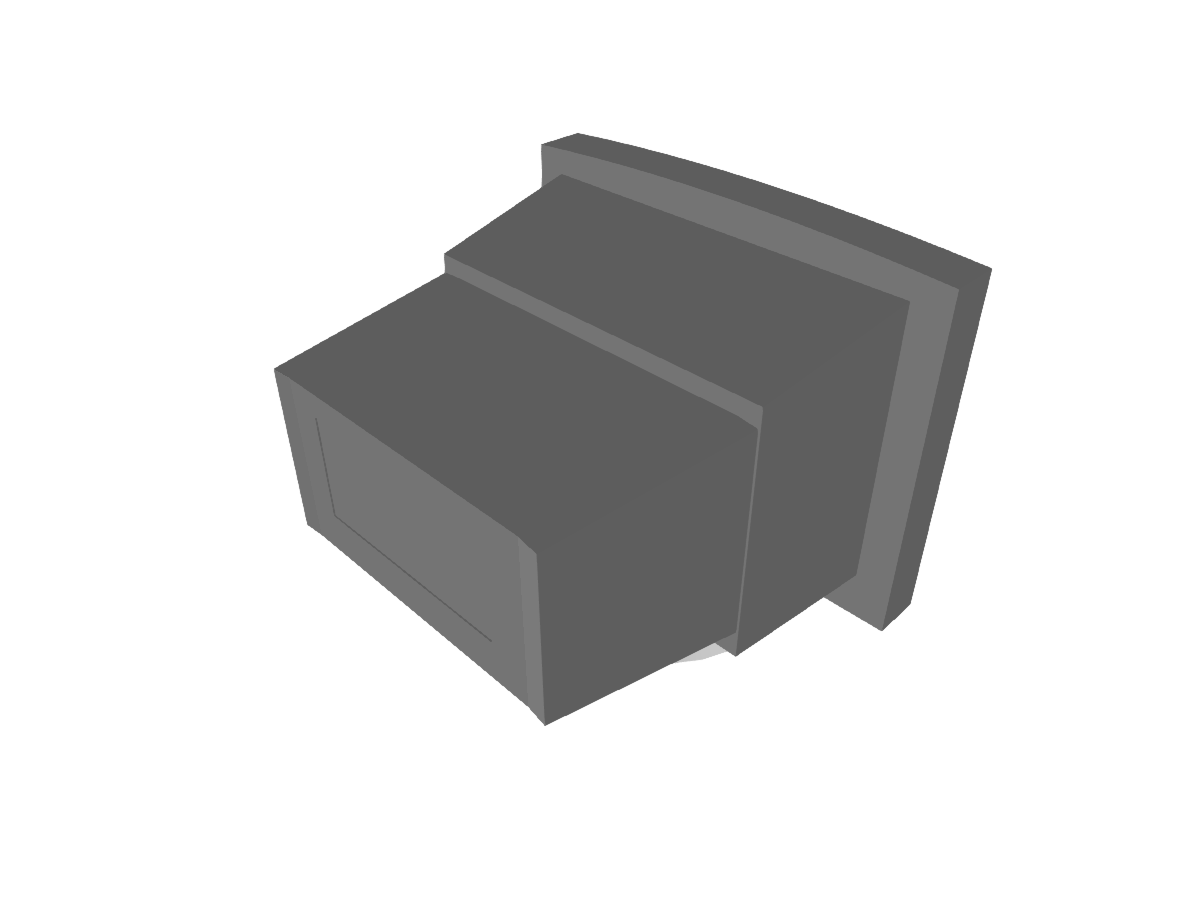} 
    \end{tabular}}
    \caption{Selection of the first to fives samples 
    of the ModelNet10 dataset 
    for (a) bathtubs, (b) beds, (c) chairs, (d) desks, (e) dresser, and (f) monitors.}
    \label{fig:modelnet10_1}
\end{figure}
\begin{figure}
    \resizebox{\linewidth}{!}{%
    \begin{tabular}{c c c c c c}
        \multicolumn{2}{c}{\includegraphics[width=0.5\linewidth, clip=true, trim=200pt 50pt 200pt 50pt]{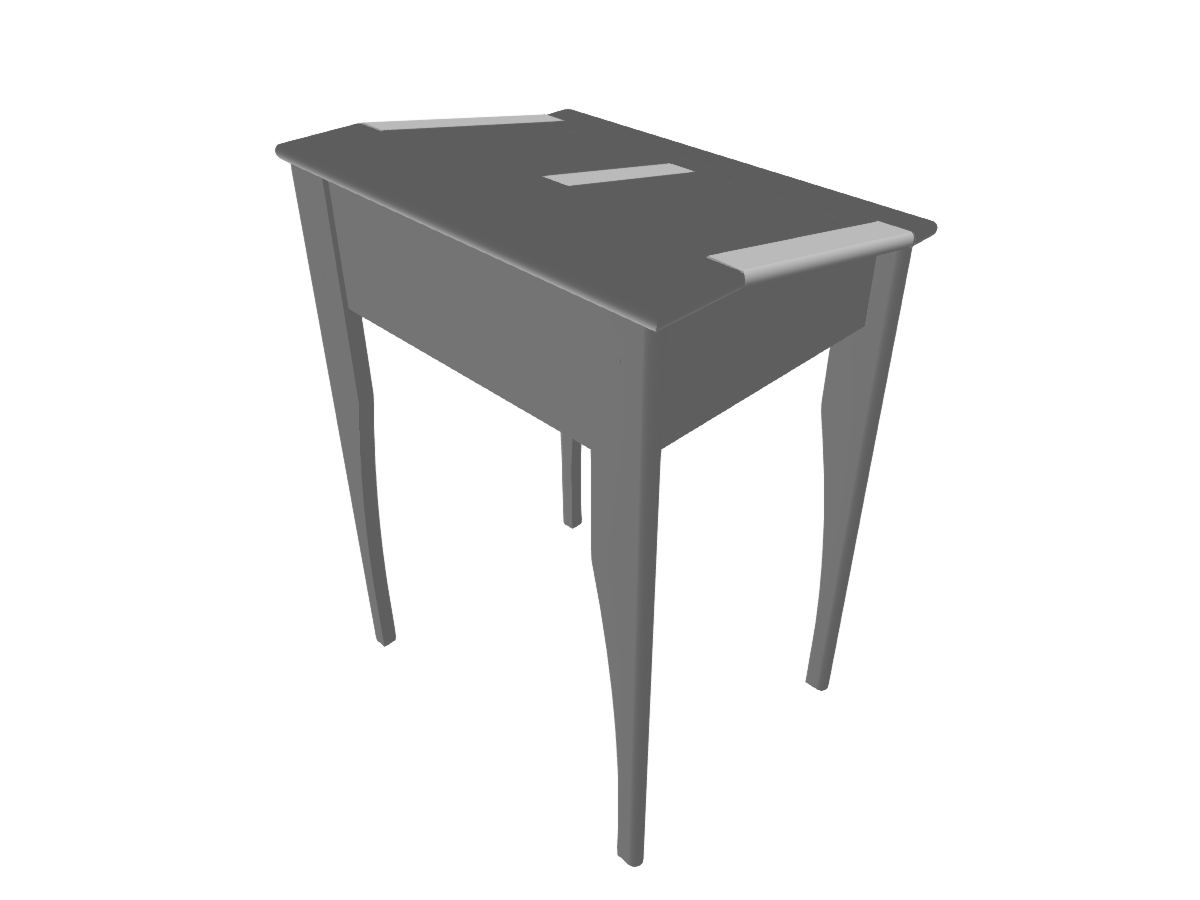}}
        &\multicolumn{2}{c}{\includegraphics[width=0.5\linewidth, clip=true, trim=200pt 50pt 200pt 50pt]{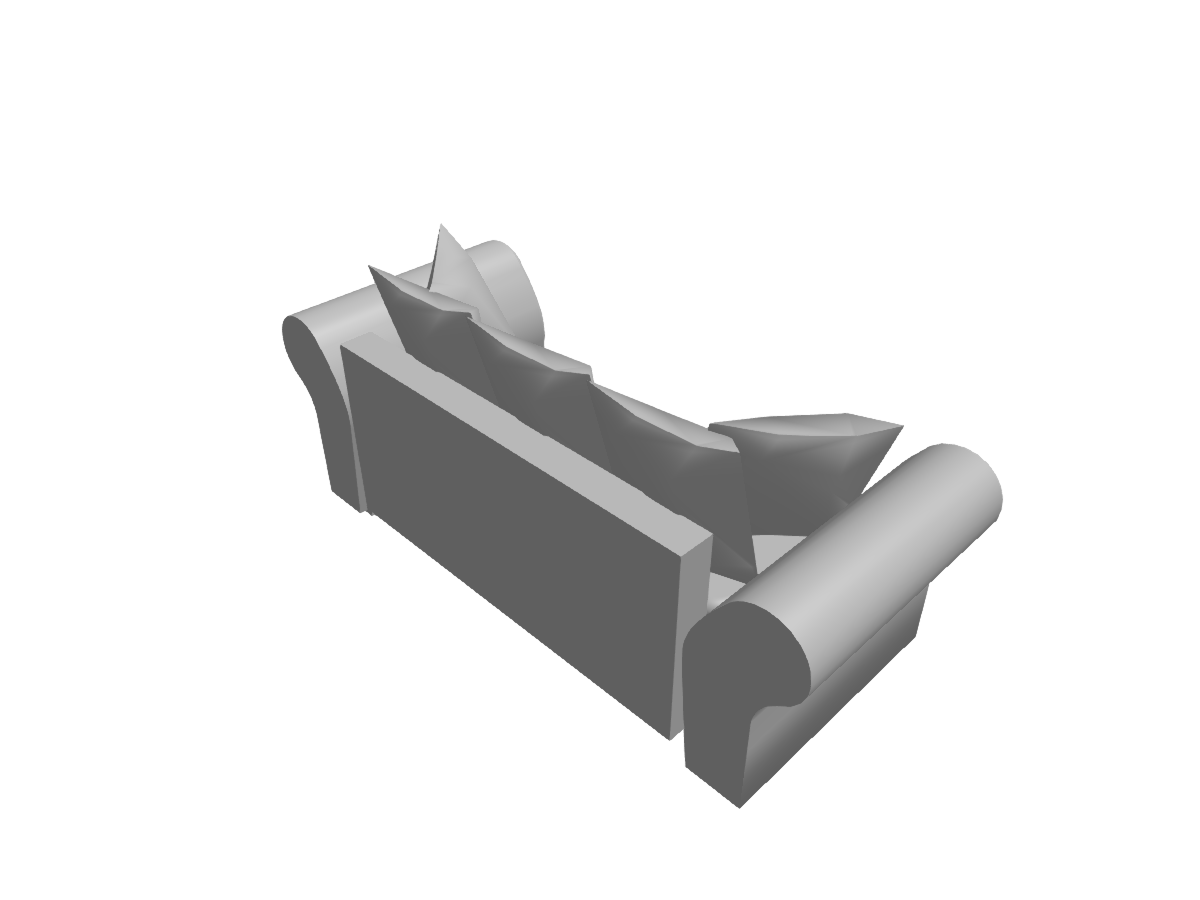}} \\
        \raisebox{6mm}[0mm][0mm]{(a)} & & \raisebox{6mm}[0mm][0mm]{(b)}\\[-10pt]
        \includegraphics[width=0.245\linewidth, clip=true, trim=200pt 50pt 200pt 50pt]{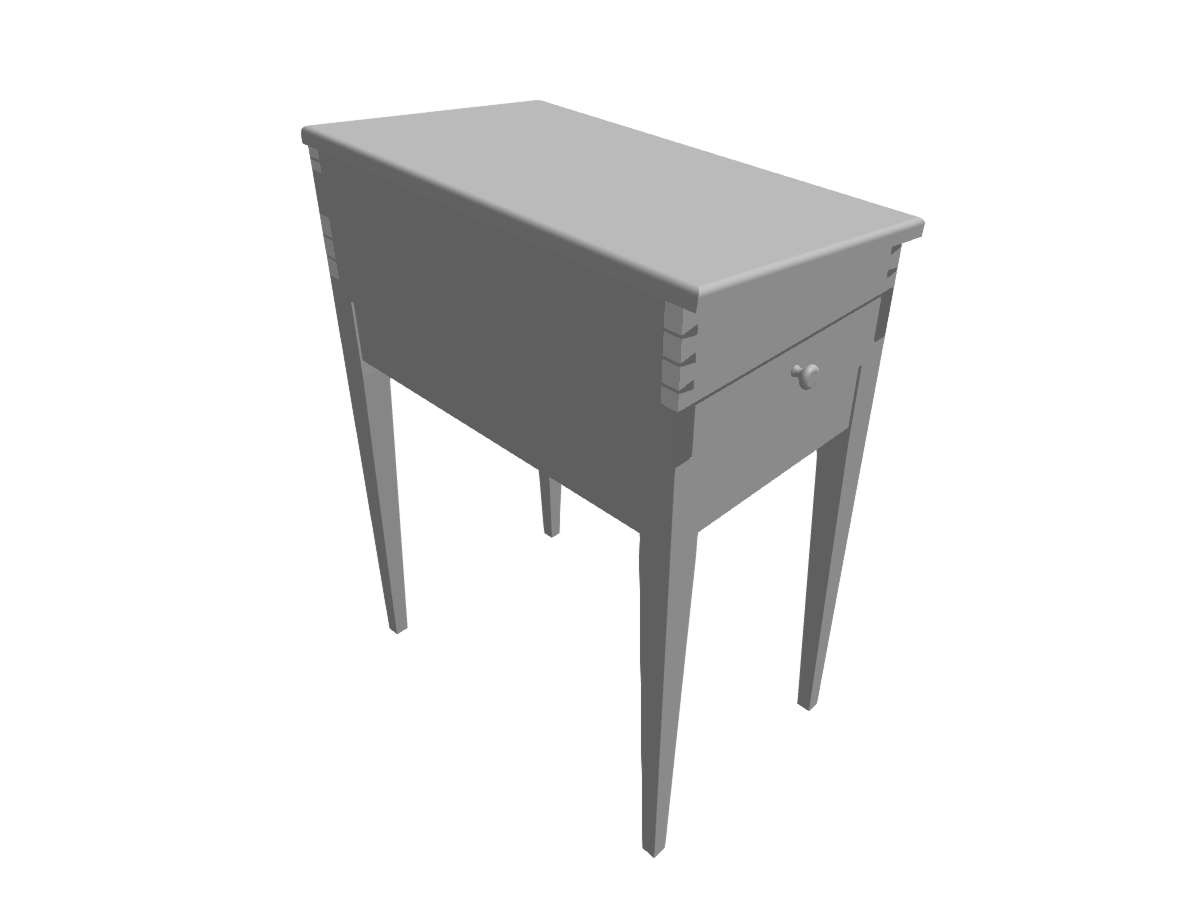}
        &\includegraphics[width=0.245\linewidth, clip=true, trim=200pt 50pt 200pt 50pt]{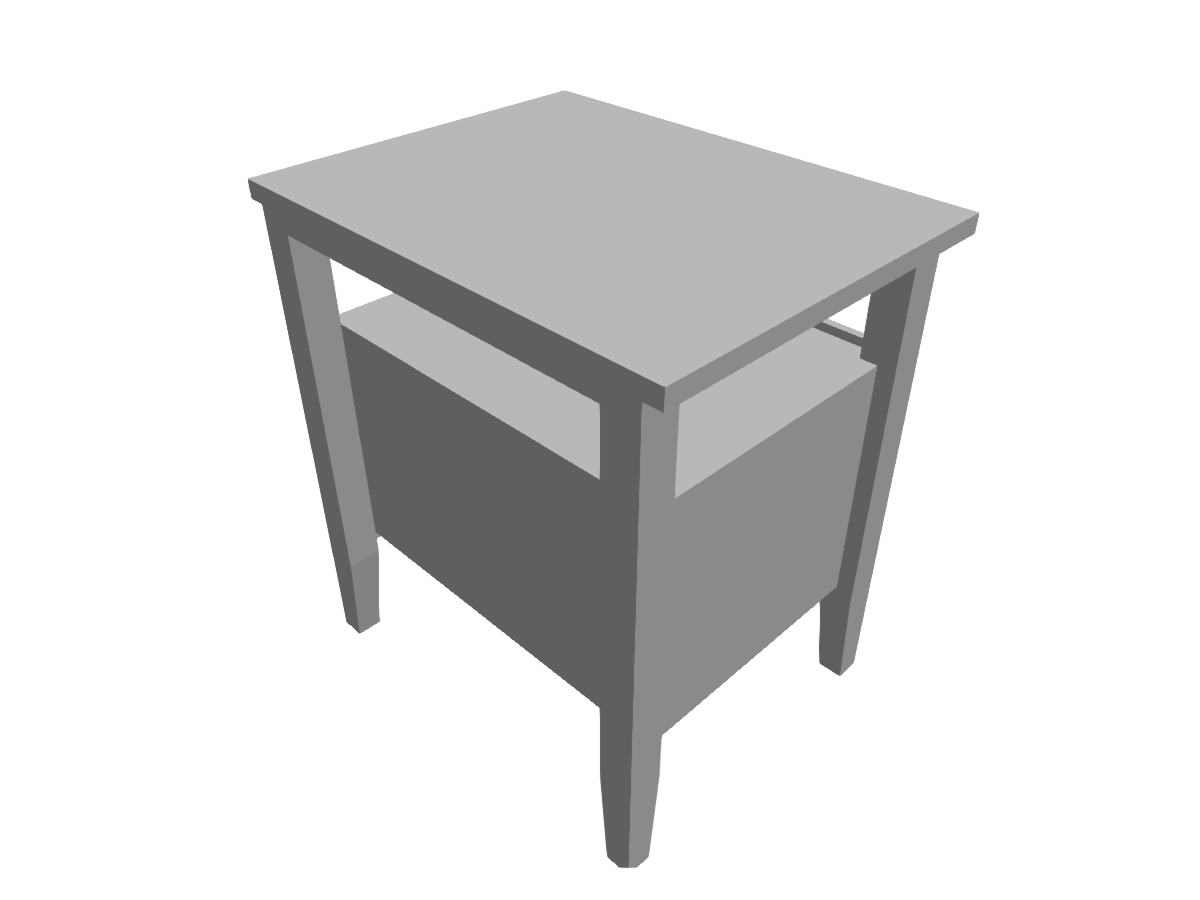} 
        &\includegraphics[width=0.245\linewidth, clip=true, trim=200pt 50pt 200pt 50pt]{numerics/sofa_2.png}
        &\includegraphics[width=0.245\linewidth, clip=true, trim=200pt 50pt 200pt 50pt]{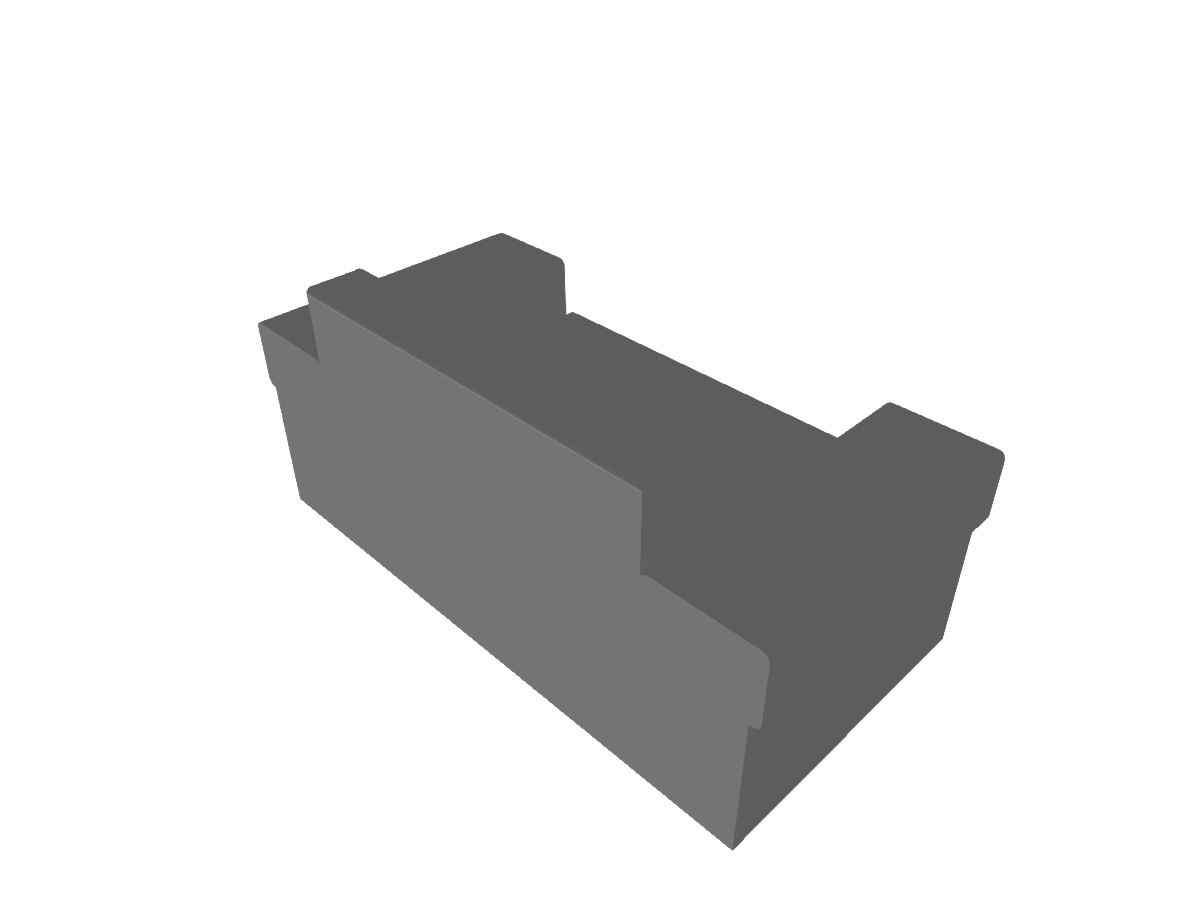} \\
        \includegraphics[width=0.245\linewidth, clip=true, trim=200pt 50pt 200pt 50pt]{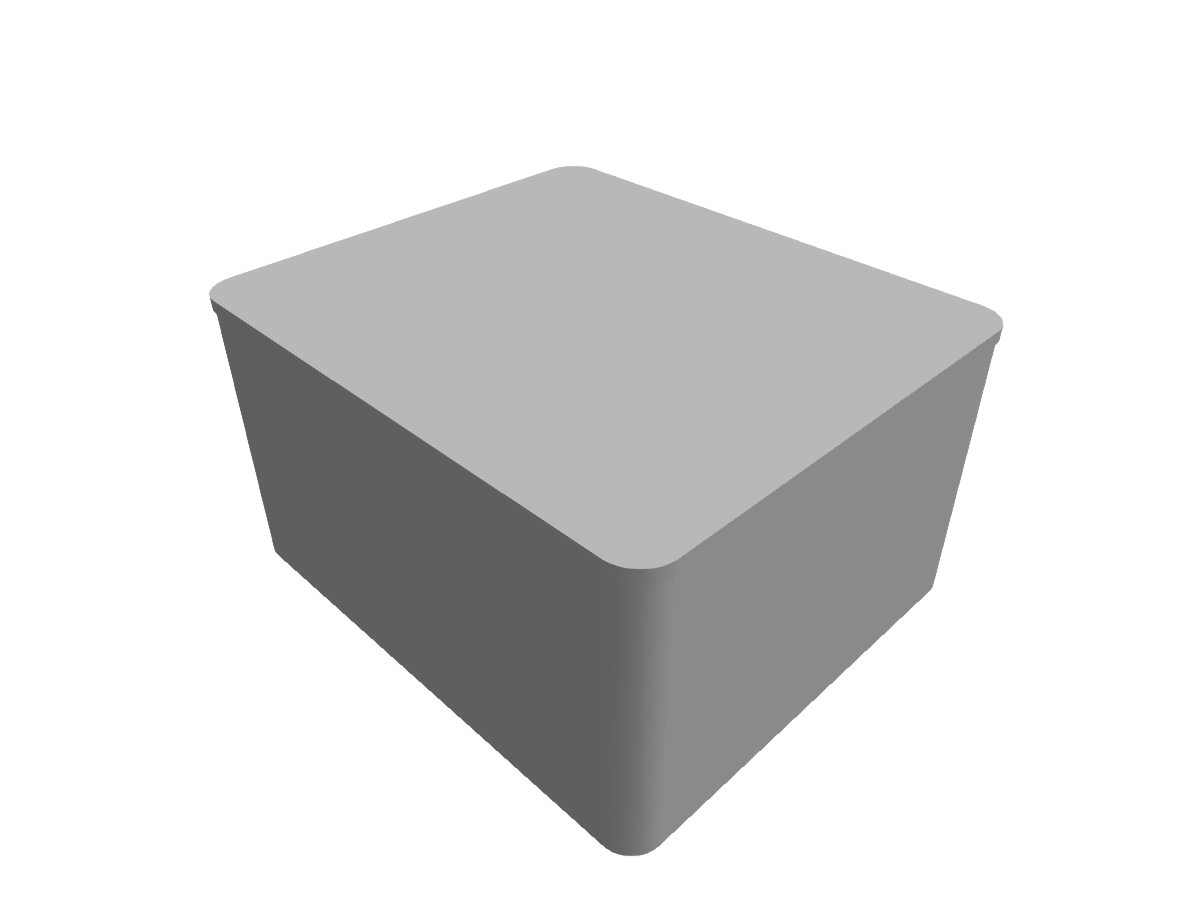}
        &\includegraphics[width=0.245\linewidth, clip=true, trim=200pt 50pt 200pt 50pt]{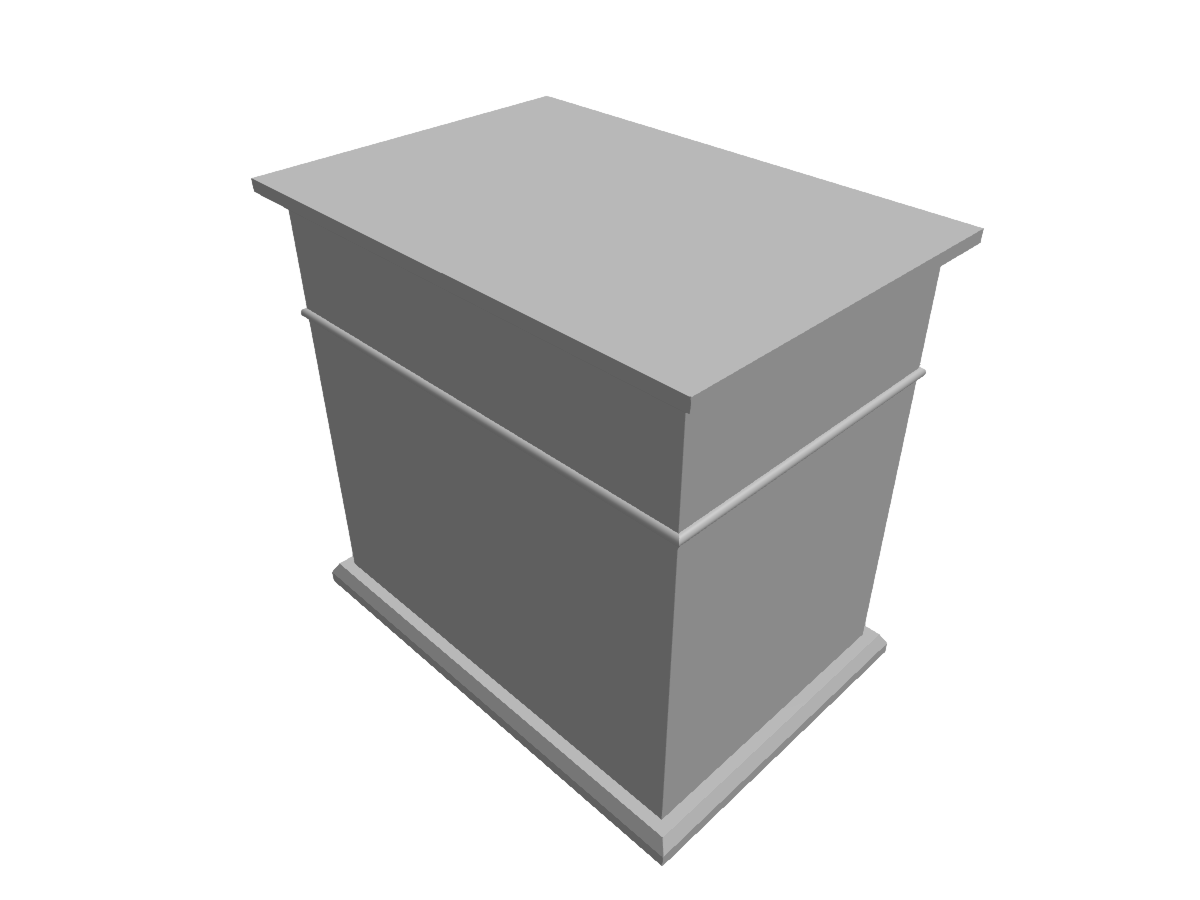} 
        &\includegraphics[width=0.245\linewidth, clip=true, trim=200pt 50pt 200pt 50pt]{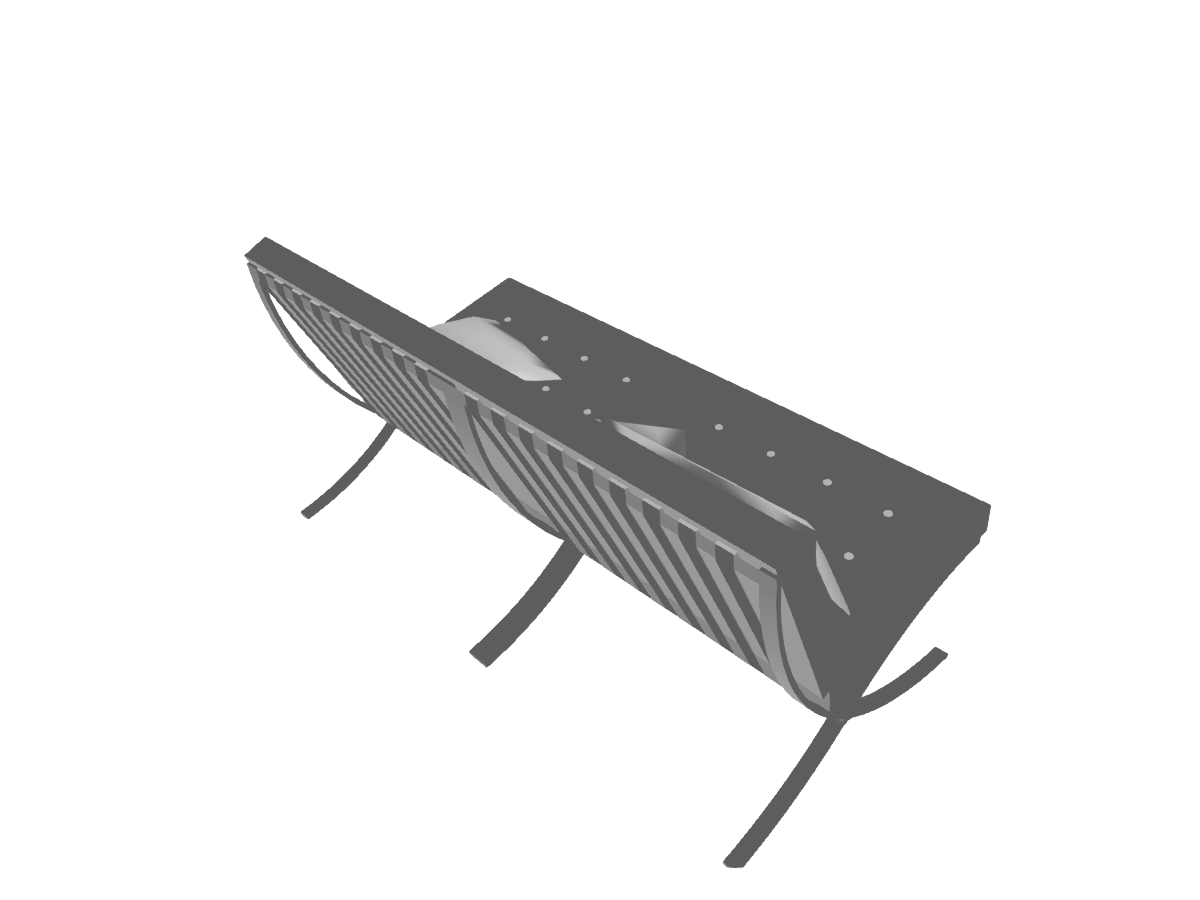}
        &\includegraphics[width=0.245\linewidth, clip=true, trim=200pt 50pt 200pt 50pt]{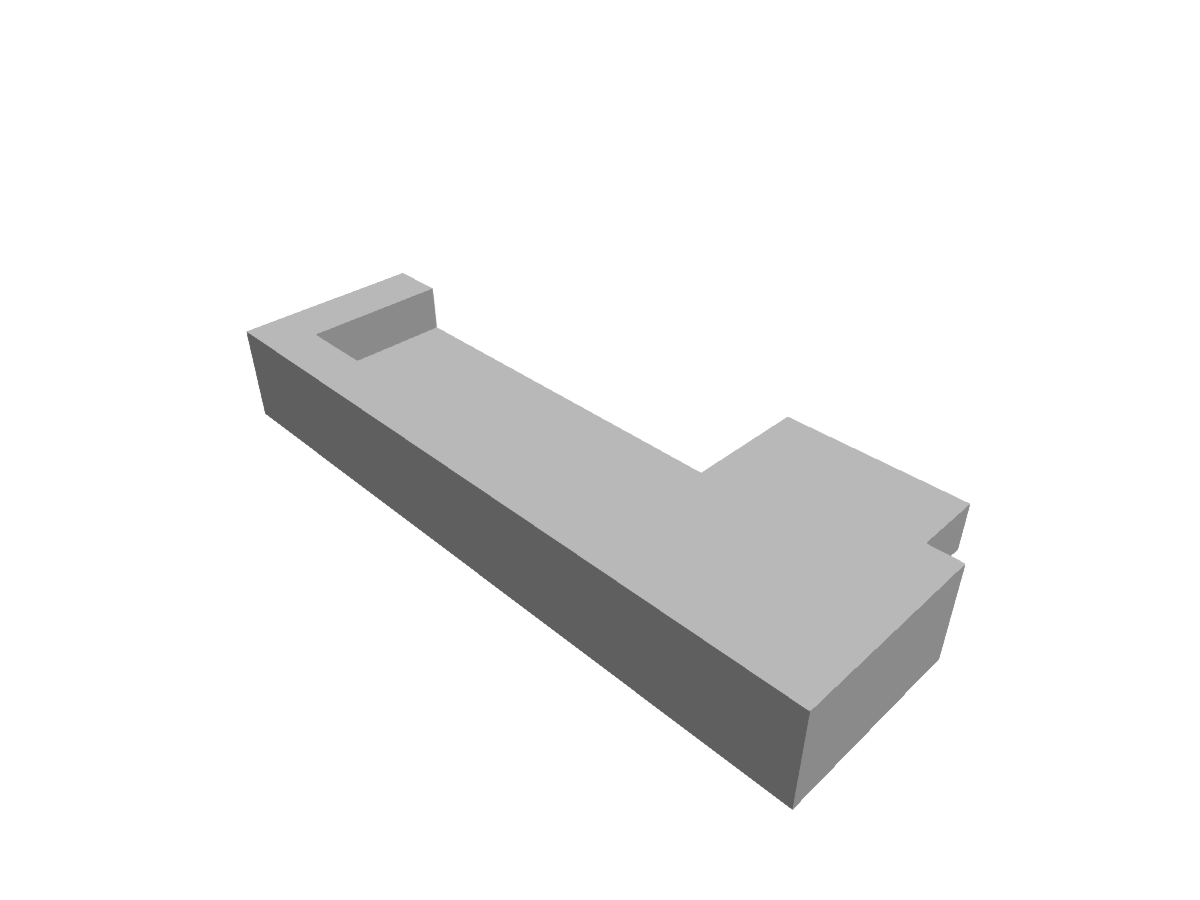} \\
        \multicolumn{2}{c}{\includegraphics[width=0.5\linewidth, clip=true, trim=200pt 50pt 200pt 50pt]{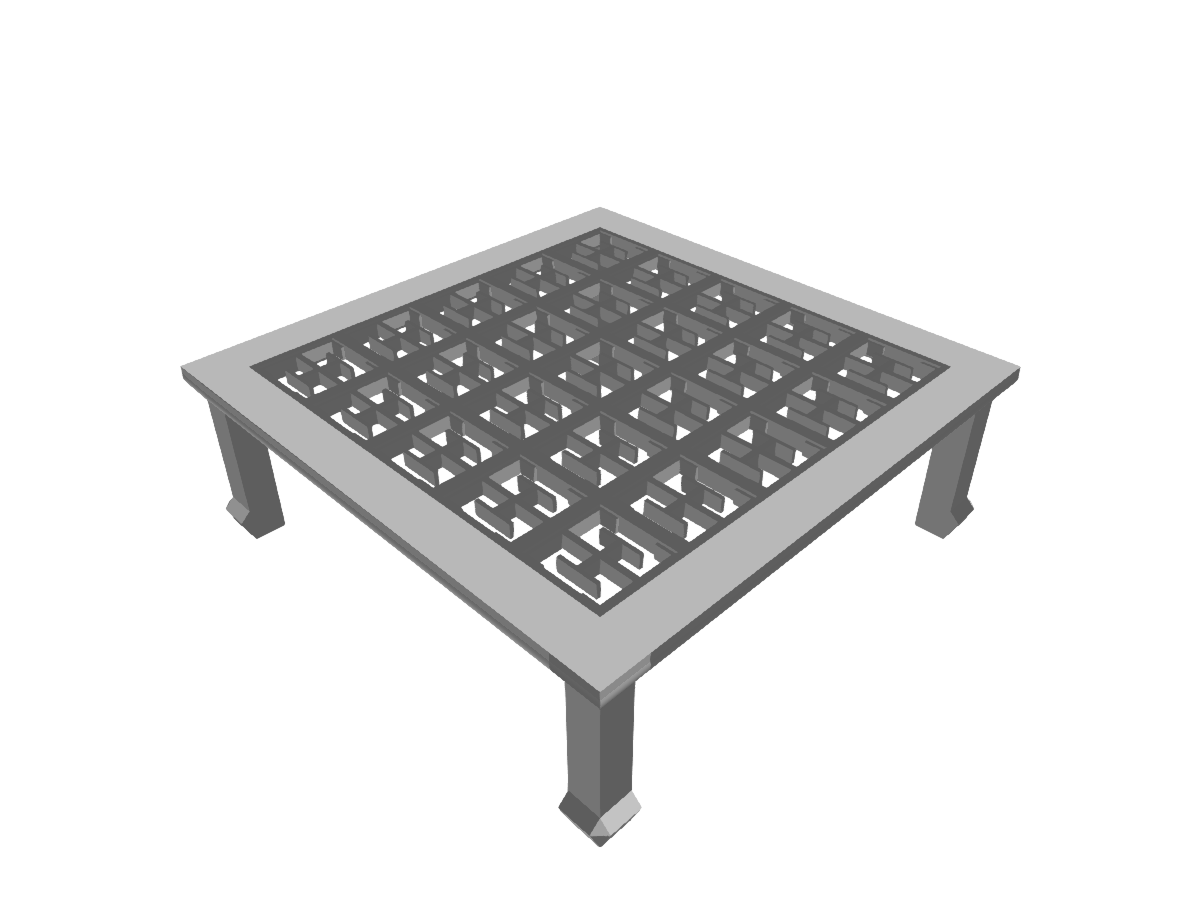}} 
        &\multicolumn{2}{c}{\includegraphics[width=0.5\linewidth, clip=true, trim=200pt 50pt 200pt 50pt]{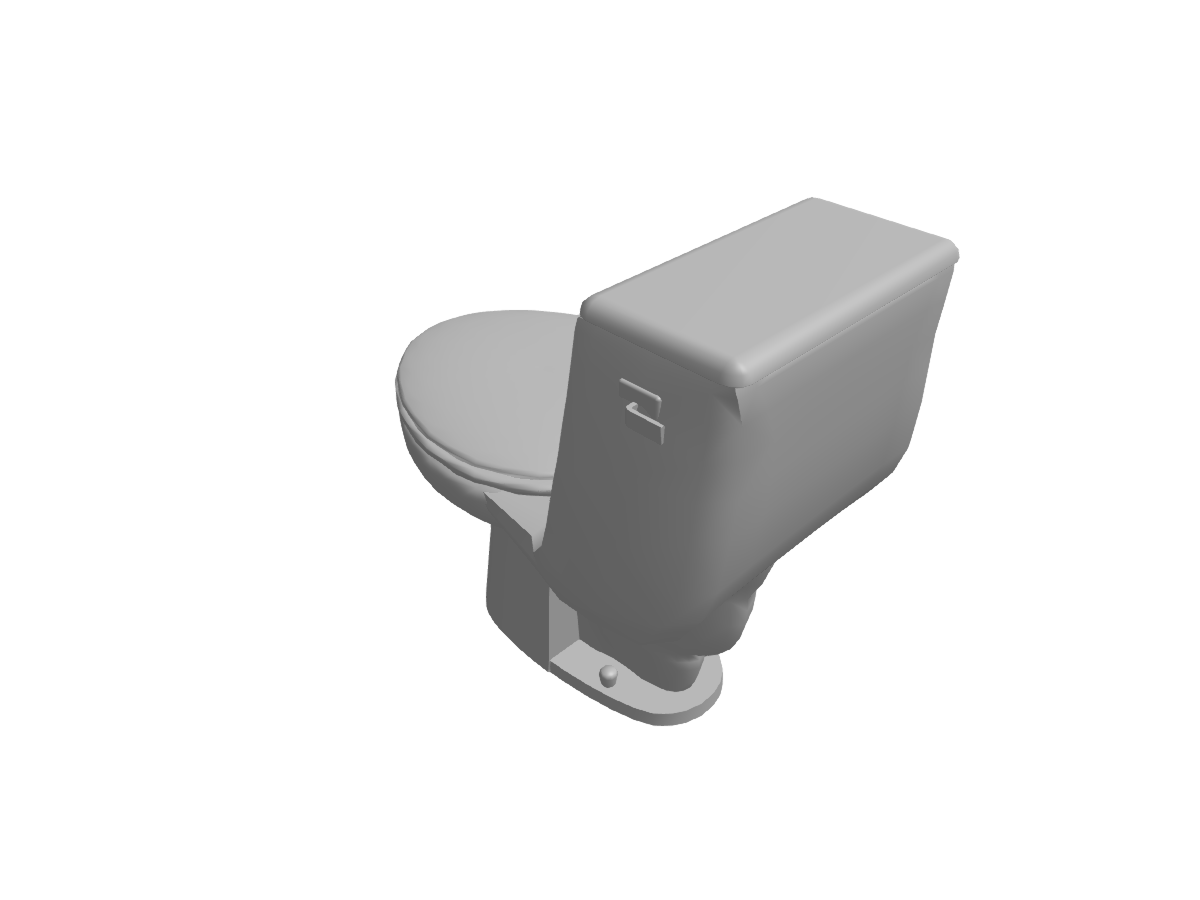}} \\
        \raisebox{6mm}[0mm][0mm]{(c)} & & \raisebox{6mm}[0mm][0mm]{(d)}\\[-10pt]
        \includegraphics[width=0.245\linewidth, clip=true, trim=200pt 50pt 200pt 50pt]{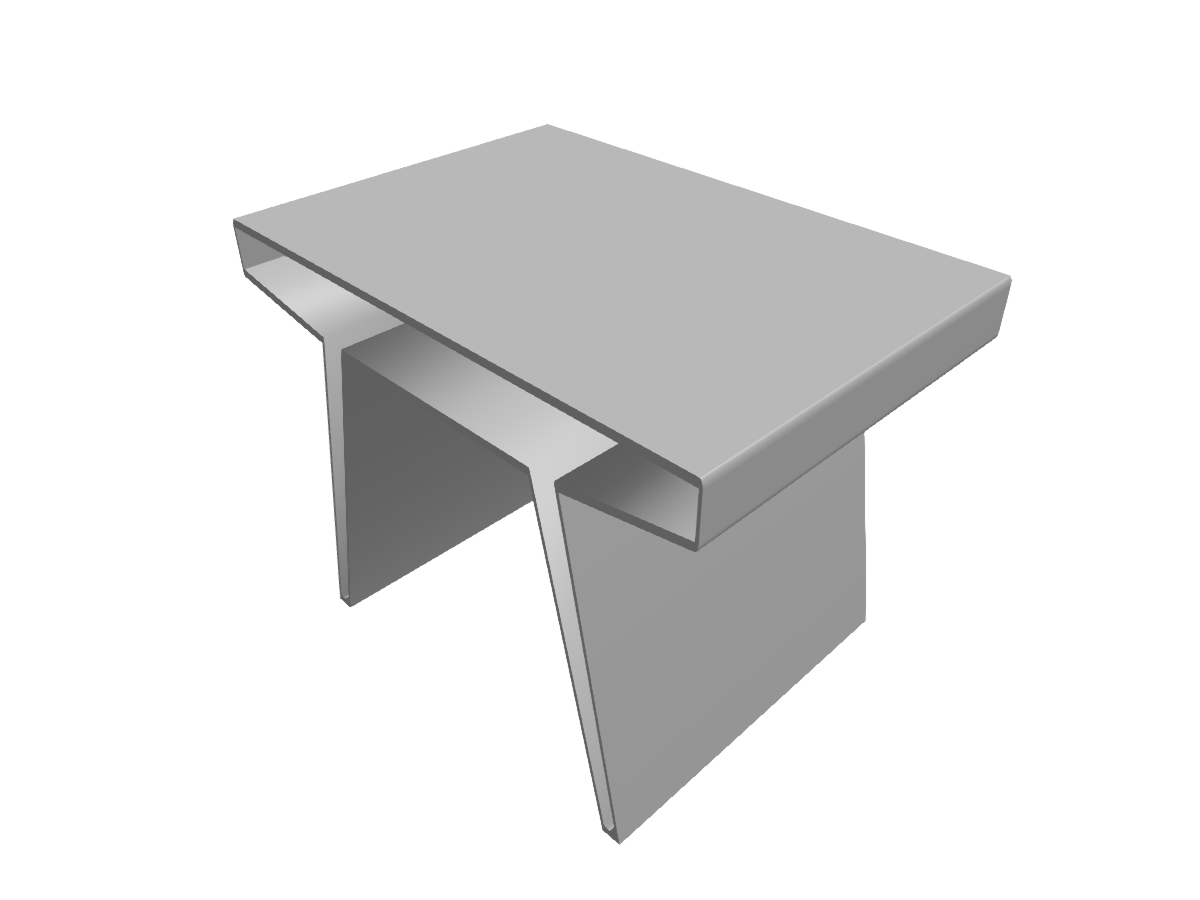}
        &\includegraphics[width=0.245\linewidth, clip=true, trim=200pt 50pt 200pt 50pt]{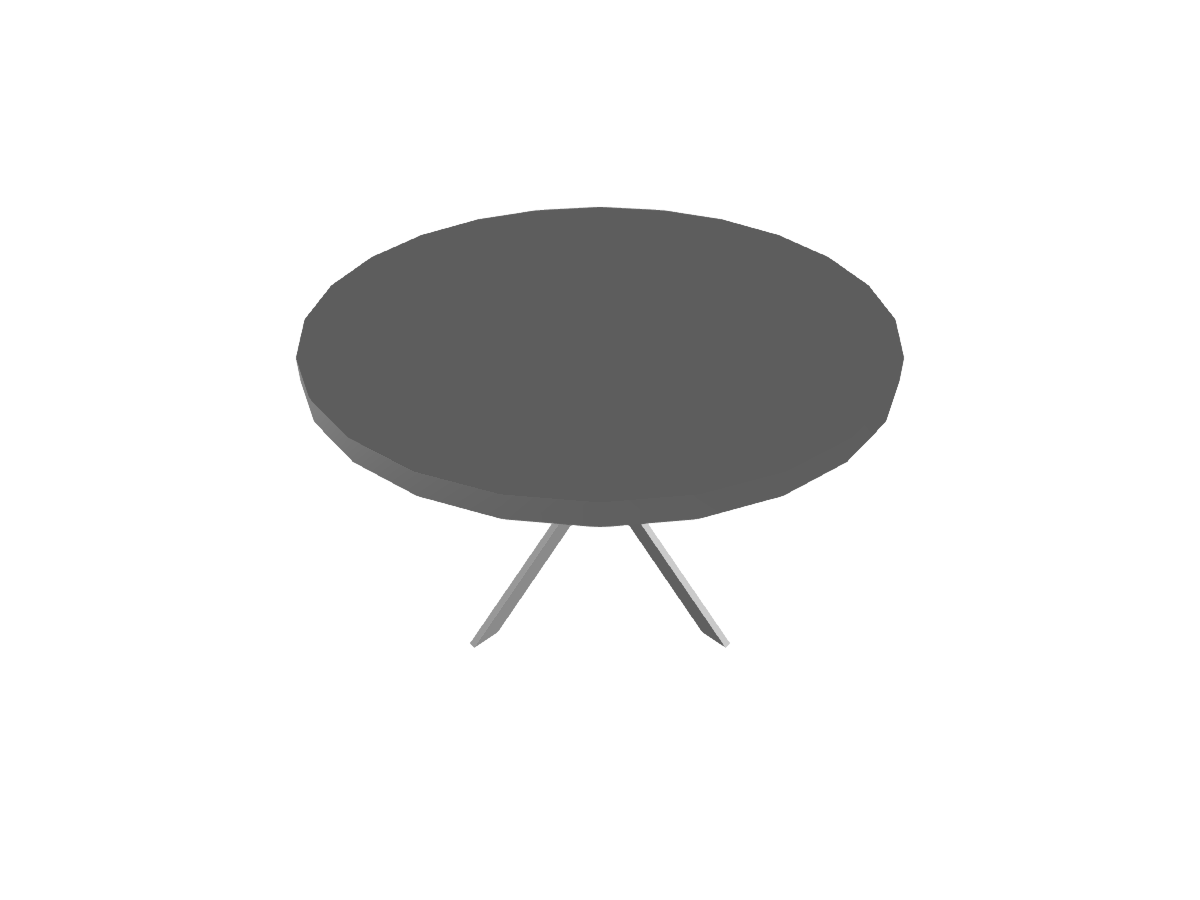} 
        &\includegraphics[width=0.245\linewidth, clip=true, trim=200pt 50pt 200pt 50pt]{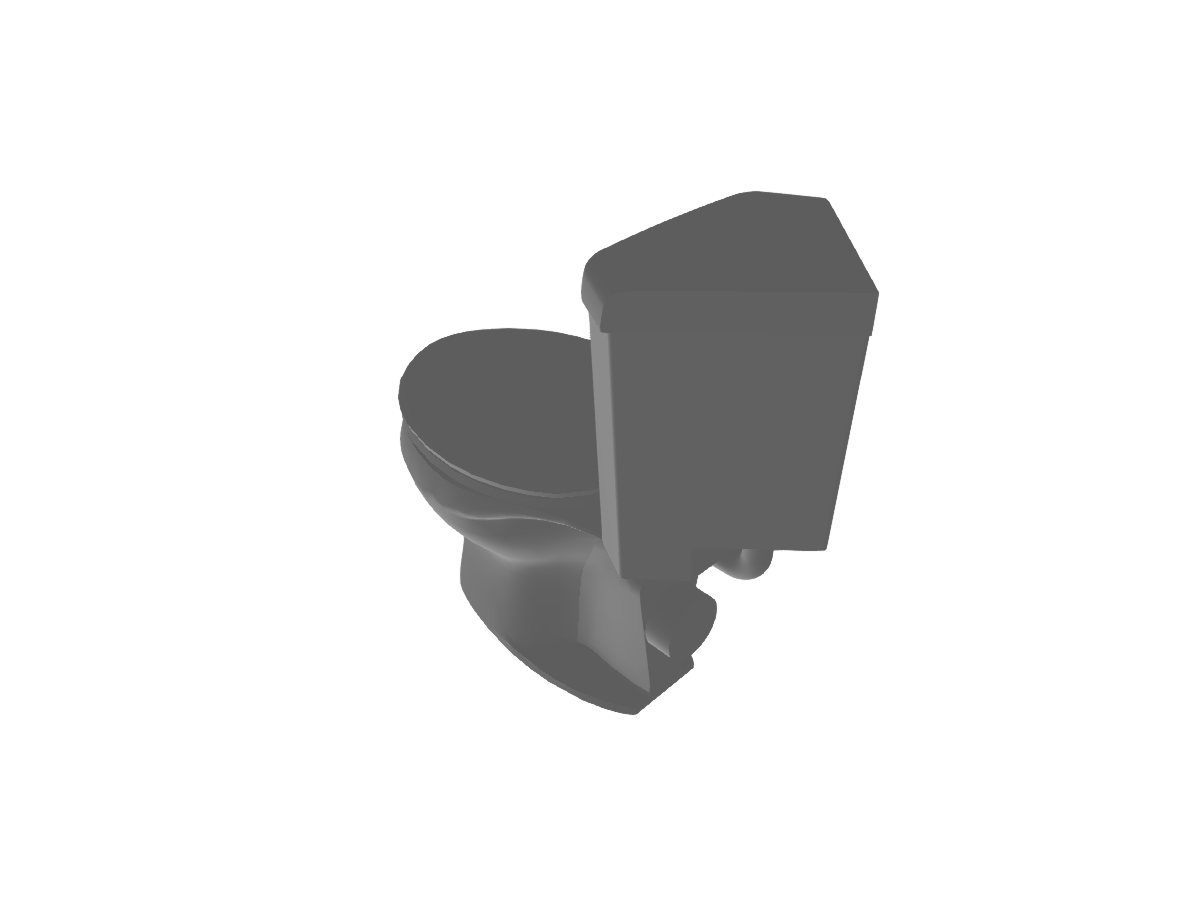}
        &\includegraphics[width=0.245\linewidth, clip=true, trim=200pt 50pt 200pt 50pt]{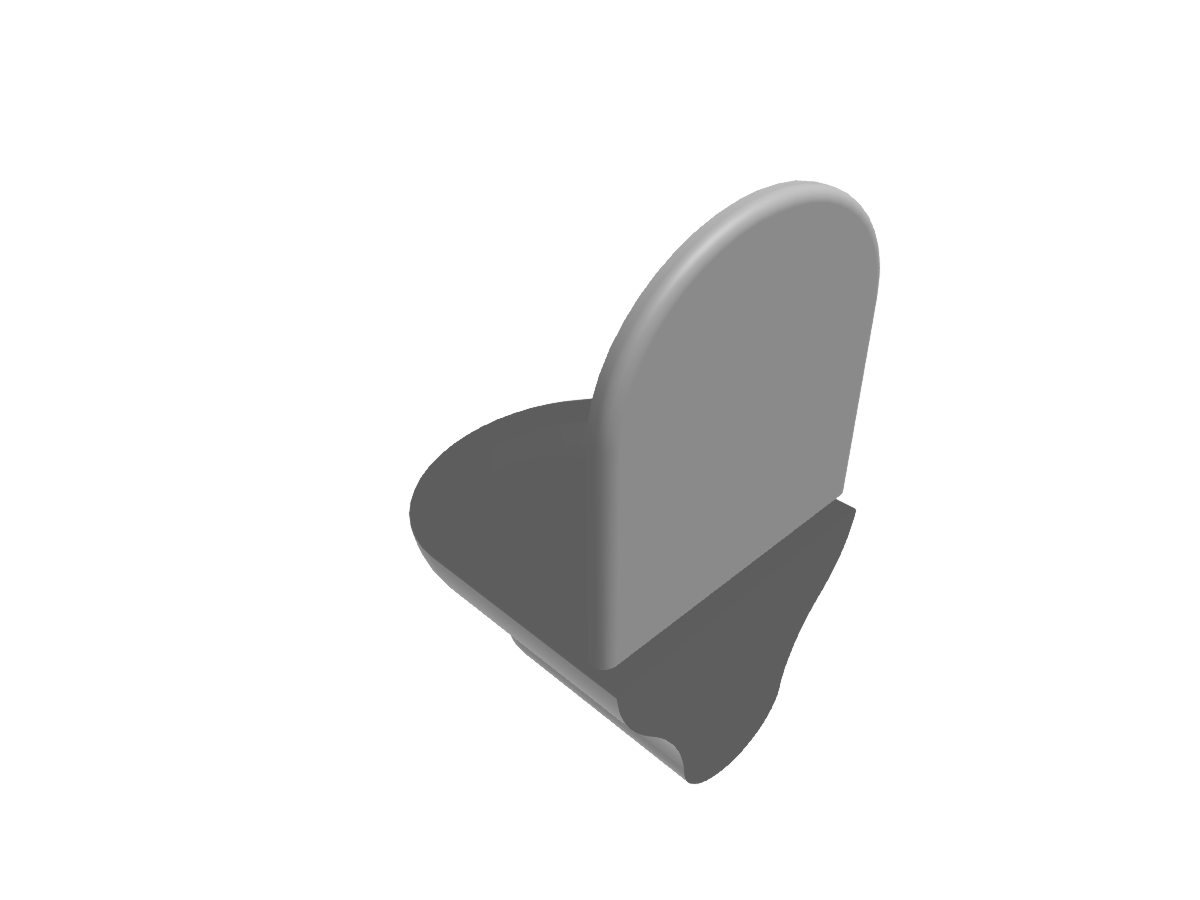} \\
        \includegraphics[width=0.245\linewidth, clip=true, trim=200pt 50pt 200pt 50pt]{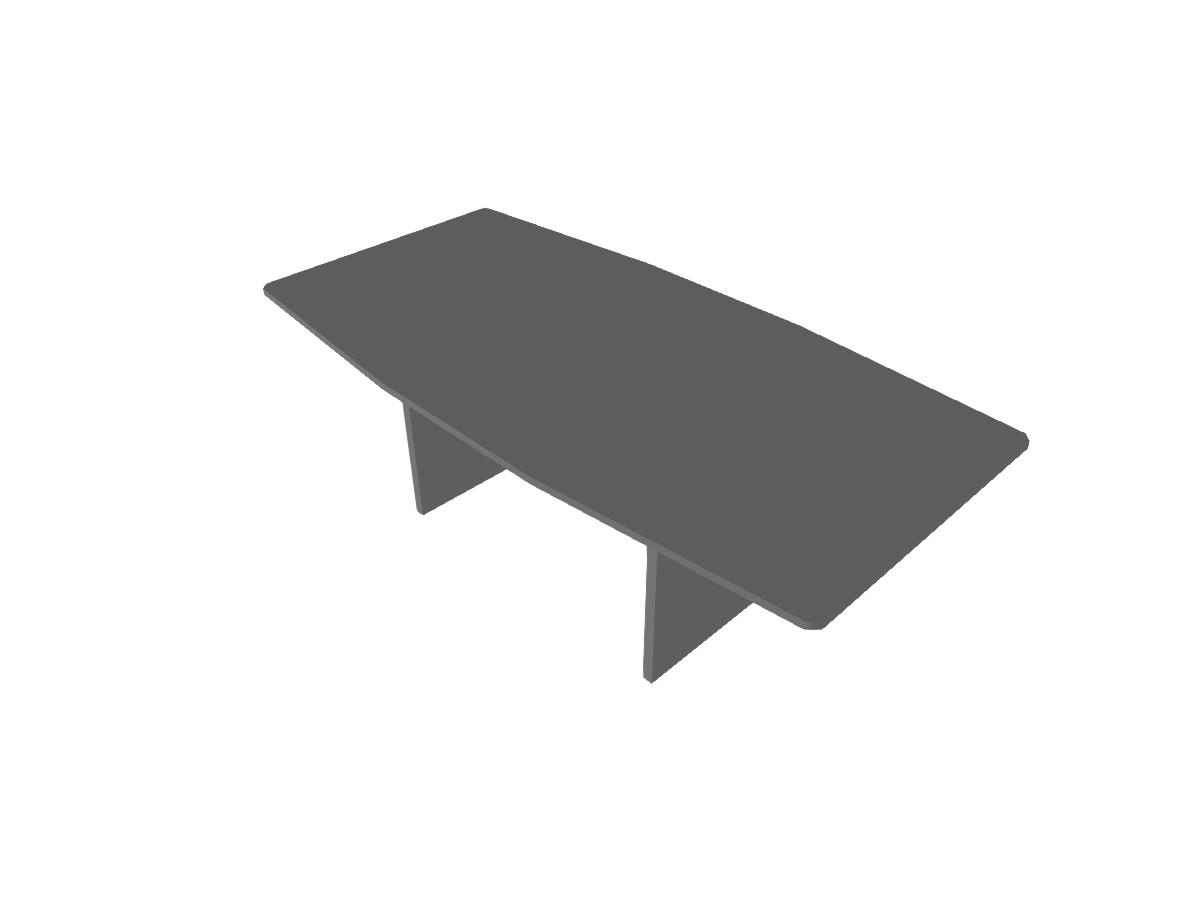}
        &\includegraphics[width=0.245\linewidth, clip=true, trim=200pt 50pt 200pt 50pt]{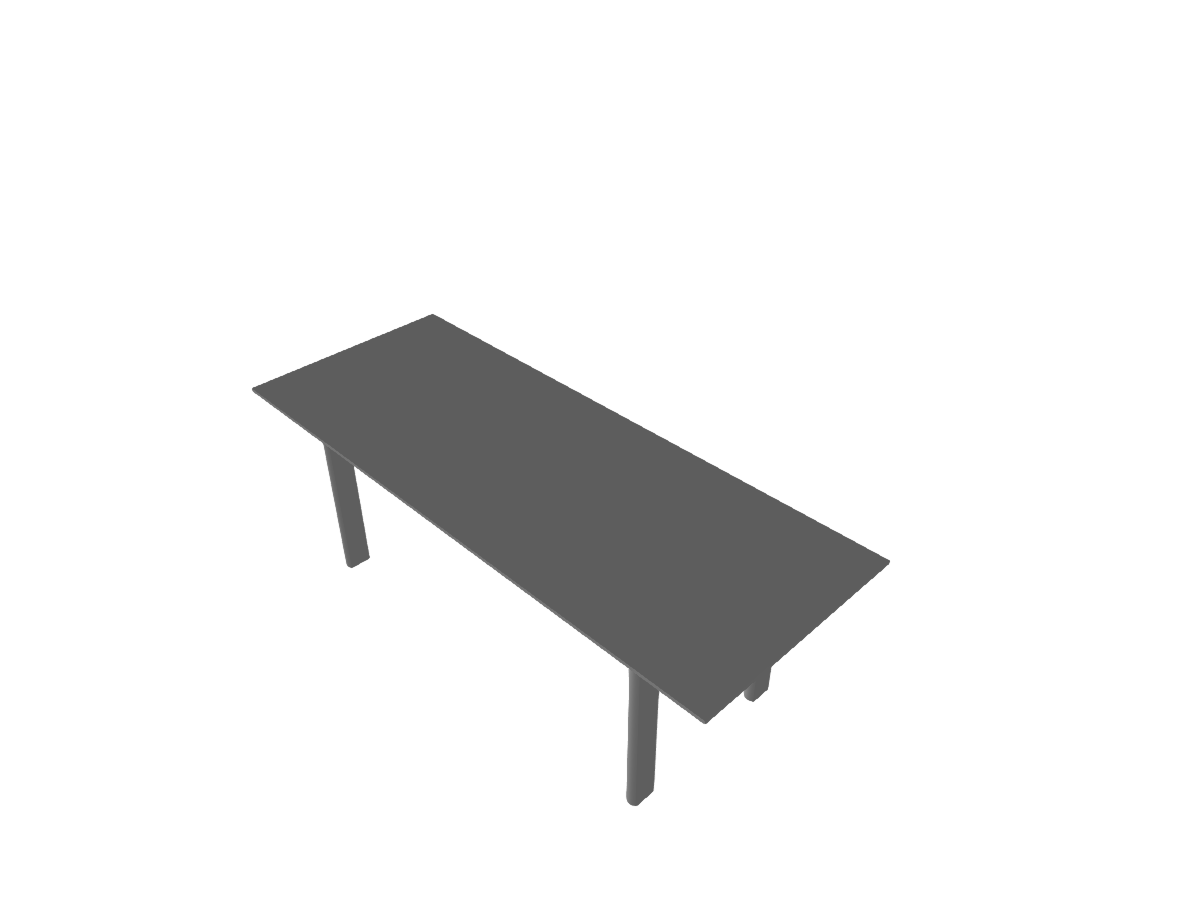} 
        &\includegraphics[width=0.245\linewidth, clip=true, trim=200pt 50pt 200pt 50pt]{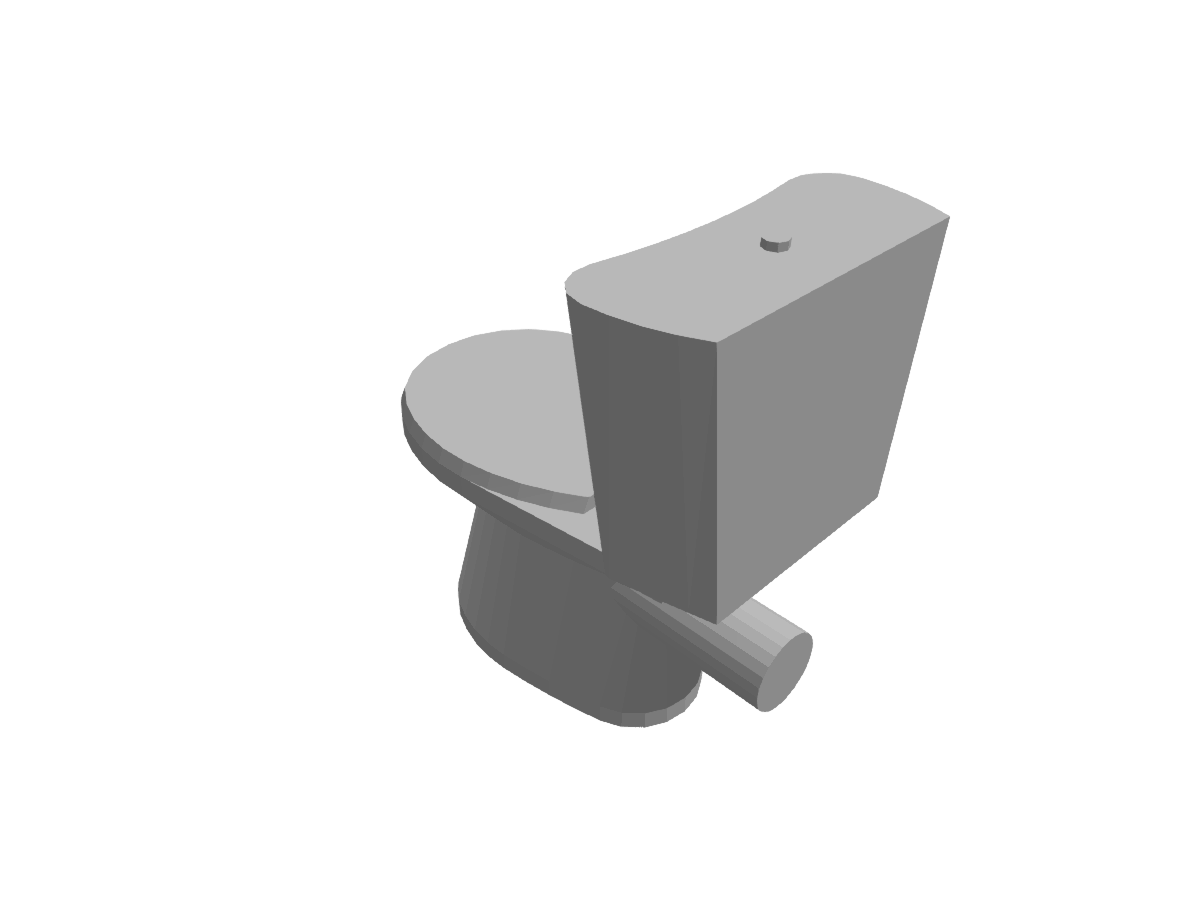}
        &\includegraphics[width=0.245\linewidth, clip=true, trim=200pt 50pt 200pt 50pt]{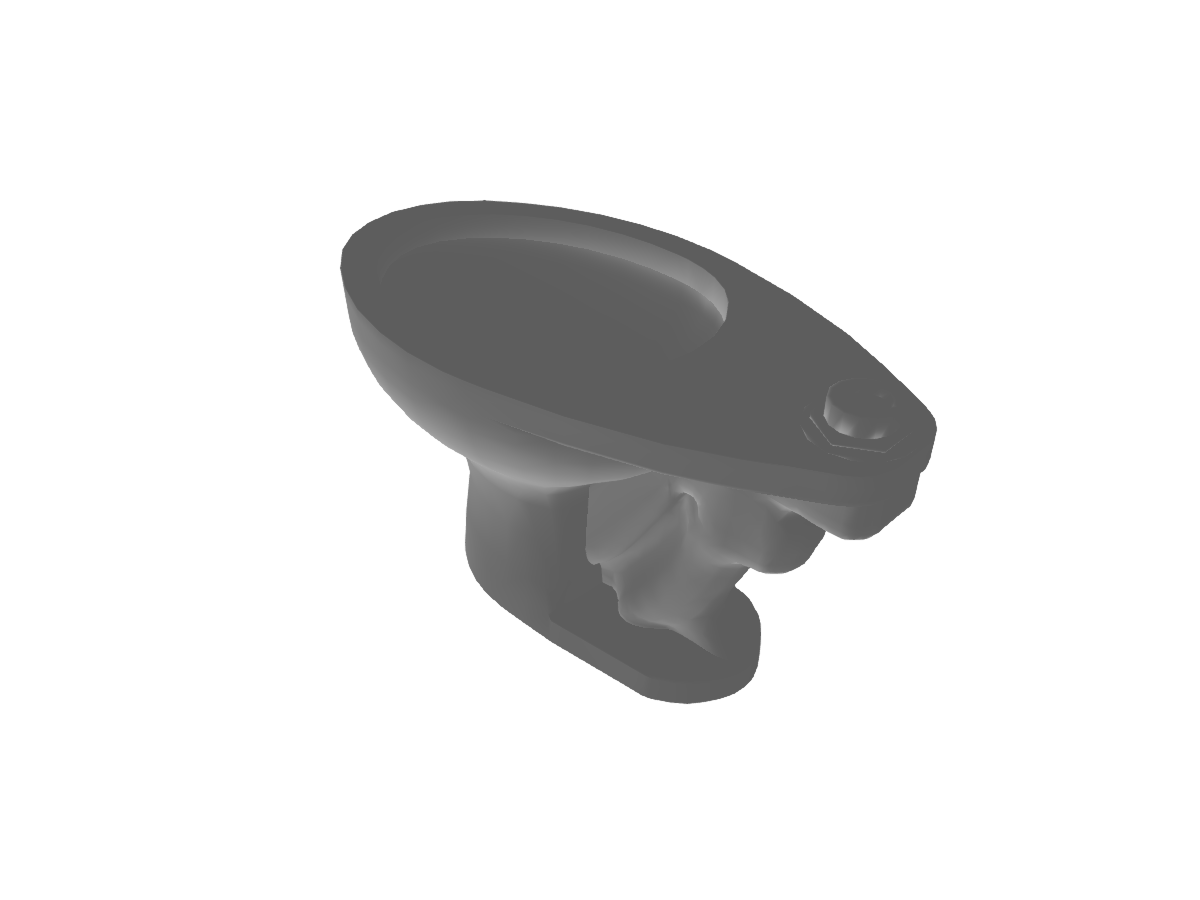}
    \end{tabular}}
    \caption{Selection of the first to fives samples 
    of the ModelNet10 dataset
    for (a) nightstands, (b) sofas, (c) tables, and (d) toilets.}
    \label{fig:modelnet10_2}
\end{figure}

We consider a subset from ModelNet10
comprising the first ten samples from each of the classes.
Initially,
we do not rely on affine object classes
for which a theoretical (linear) separability result holds.
Therefore, the considered experiments may validate
robustness of the transforms.
Afterwards, we consider classes of affinely transformed objects in Section~\ref{sec:affine_objects}.

\subsection{Discretization of the {Radon} transform}
\label{sec:discrete_model}

For a numerical realization of the {Radon} transform
on the voxelized objects,
we need a discrete model.
For $N\in\N$,
we define the symmetric index set
\begin{equation*}
    \mathcal I_N
    \coloneqq
    \left\{ \nicefrac{-N+1}{2}, \nicefrac{-N+3}{2}, \dots, \nicefrac{N-1}{2} \right\}
\end{equation*}
and the regular grid $\mathcal I_N^d$ as $d$-fold Cartesian product.
We consider a $d$-dimensional image $F\in\R^{\mathcal I_N^d}$ with $N$ voxels in each dimension.
Each voxel $\zbn=(n_1,\dots,n_d)\in \mathcal I_N^d$ is imagined
as cube with side length (i.e.\ voxel size) $s>0$ 
centered at $s\cdot\zbn\in\R^d$, i.e., 
the cube $(s(\zbn-\nicefrac12),s(\zbn+\nicefrac12)]$.
Then $F$ corresponds to the piecewice constant function
\begin{equation} \label{eq:f-discrete}
    f_F(\zbx)
    \coloneqq
    \sum_{\zbn\in\mathcal I_N^d} F(\zbn)\, \mathbf 1_{(s(\zbn-\nicefrac12),s(\zbn+\nicefrac12)]}(\zbx)
    ,\qquad \zbx\in\R^d.
\end{equation}
The cubes touch but do not overlap.

Let $\bftheta\in\S^{d-1}$ and $t\in\R$.
By linearity, 
we have
\begin{equation} 
    \mathcal R_{\bftheta}[f_F](t)
    =
    \sum_{\zbn\in\mathcal I_N^d} 
    F(\zbn)\, \mathcal R_{\bftheta}\big[\mathbf 1_{(s(\zbn-\nicefrac12),s(\zbn+\nicefrac12)]}\big](t).
\end{equation}
The shift identity for the {Radon} transform, 
$$
    \Radon_{\bftheta}[f(\cdot - \zby)](t)
    =
    \Radon_{\bftheta}[f](t-\langle \bftheta, \zby\rangle)
    ,\qquad \zby\in\R^d,
$$
follows directly from the definition \eqref{eq:radon}.
Together with the observation that 
$
\mathbf 1_{(s(\zbn-\frac12),s(\zbn+\frac12)]}(\zbx)
=
\mathbf 1_{\nicefrac s2 \zbe}(\zbx-s\cdot \zbn),
$
we obtain that
\begin{equation} 
\mathcal R_{\bftheta}[f_F](t)
=
\sum_{\zbn\in\mathcal I_N^d} 
F(\zbn)\, 
\mathcal R_{\bftheta}\big[\mathbf 1_{\nicefrac s2\zbe}\big](t-s\langle \zbn,\bftheta\rangle).
\end{equation}
Then we have
\begin{equation} \label{eq:Radon-discrete}
\mathcal R_{\bftheta}[f_F](t)
=
\sum_{\zbn\in\mathcal I_N^d} 
F(\zbn)\, A_{\bftheta}^{\nicefrac s2 \zbe}(t-s\langle \zbn,\bftheta\rangle)
\end{equation}
with $A_{\bftheta}^{\zba}$ given in \eqref{eq:Radon-cube}.
Since the box $\nicefrac s2(- \zbe, \zbe]$ has a diagonal of $\nicefrac {s\sqrt d}2 $,
the summands in \eqref{eq:Radon-discrete} vanish
if $| t - s\langle n,\bftheta\rangle| > \nicefrac {s\sqrt{d}}2 $.
For improved stability, we replace  $A_{\bftheta}^{\zba}(t)$ by its regularized variant $(2\varepsilon)^{-1} V_{\bftheta}^{\zba}(t-\varepsilon,t+\varepsilon)$, see \eqref{eq:volume-slab-general}.

\begin{remark}[Existing approaches]
\label{rem:binning}
Besides our approach, there are other ways to discretize the {Radon} transform.
Most work covers only the case $d=2$.
A simple, well-known approach \cite{Bracewell1995} projects the center of a voxel to the line $\bftheta^\perp$, and then uses binning in $t$.
For general $d$, we set $b$ as bin size in~$t$, 
and approximate
$$
\mathcal R_{\bftheta}[f_F](t)
\approx
\sum_{\substack{\zbn\in\mathcal I_N^d\\ |\langle \zbn,\bftheta\rangle -t| < b }} F(\zbn).
$$
This approach is usually implemented 
by first computing the projections $\langle \zbn,\bftheta\rangle$ 
for all voxels $\zbn$ and then binning them to a grid in $t$.
Therefore, this approach is generally faster than ours, but less accurate.
It can be improved by splitting each voxel into subvoxels.
The convergence of different approaches was analyzed in \cite{Huber25,BrediesHuber21}.

Another approach, known as slant stacking \cite{Toft1996} 
in 2D discretizes the line integral by a quadrature on the line $\bftheta^\perp$,
where the function values of $f$ are obtained via interpolation of $F$;
taking $f_F$ here corresponds to a the nearest neighbor interpolation.
Discrete convolutions in combination with data interpolation was utilized for the 2D {Radon} transform in \cite{Andersson2016}.
Monte Carlo approximations were discussed in \cite{Agarwal2019,ChyManRei2008,DeHoop1996}.
\end{remark}

Unfortunately, 
there are fewer readily available implementations of the 3D {Radon} transform.
Although some packages like \texttt{RadonKA}%
\footnote{\label{fn:RadonKA}Julia package \texttt{RadonKA}~\url{https://github.com/roflmaostc/RadonKA.jl}} 
\cite{Wechsler2024} provide ``3D {Radon} transforms'', 
these are vectorized versions of the 2D {Radon} transform 
(i.e.\ integrating along all lines perpendicular to the third coordinate) 
rather than 
integrating along planes.
In \cite{Toft1996},
closed forms for few simple objects, 
e.g.\ Gaußian bells
and ellipsoids in 3D 
are provided.

\subsection{Radon shape matching via feature extraction \cite{Daras2004}}
\label{sec:matching}

We first revise the construction 
of the feature extractors,
and then provide the numerical results 
for the shape matching.

\subsubsection{Preprocessing}
\label{sec:preprocessing}

According to
\cite{Daras2004},
the following preprocessing steps 
on the reduced ModelNet10 dataset from Section~\ref{sec:datasets}
are applied:
(i) centering the objects
by the means of mass points, 
i.e. centroid of the voxels,
(ii) \textit{Principal Component Analysis} (PCA) alignment 
using the covariance/inertia matrix
for rotating towards the principal axes,
(iii) scaling into the unit box,
and (iv) normalizing the mass 
after revoxelizing the objects.

\subsubsection{Construction of the features}
\label{sec:matching_construction}

Following the definition of the feature extractors from \cite[Eq.~(10--13)]{Daras2004},
the sphere is discretized 
by a uniform grid on the spherical coordinates given by
\begin{equation}
    \label{eq:Sd-cooridnates}
    \bftheta(\bfvartheta) \coloneqq \left[\begin{smallmatrix}
        \sin\vartheta_1\sin\vartheta_2 \\
        \cos\vartheta_{1}\sin\vartheta_{2} \\
        \cos\vartheta_{2}
    \end{smallmatrix}\right],
    \quad \left\{\begin{smallmatrix}
        \bfvartheta = (\vartheta_1,\vartheta_{2}), \\ 
        \vartheta_{1} \in [0, 2\pi)\\
        \vartheta_2\in [0,\pi].
    \end{smallmatrix}\right.
\end{equation}
The extractors $F_i\colon \R^d\to\R$ given by 
\begin{align*}
    F_1 (g)&\coloneqq \max_{i = 1...,D} g_i, &
    F_2 (g)&\coloneqq \sum_{i = 1}^{D-1} \tfrac{1}{2}|g_{i+1}-g_i|, \\
    F_3 (g)&\coloneqq \sum_{i = 1}^D g_i, &
    F_4 (g)&\coloneqq F_1(g) + F_1(-g),
\end{align*}
are applied recursively on 
\begin{align*}
    \Radon[f_F] : \R^{D_1} \times [0,2\pi)^{D_2} \times[0,\pi]^{D_3} &\to \R, \\
    (t, \bftheta(\vartheta_1,\vartheta_2)) &\mapsto \Radon_{\bftheta(\vartheta_1,\vartheta_2)}[f](t).
\end{align*}
The ordering $(F_{i_1}, F_{i_2}, F_{i_3})$
given by the permutation $\pi = (\pi_1, \pi_2, \pi_3) \in S_3(\{t,\vartheta_1, \vartheta_2\})$ 
of the $51$ selected extractors 
is summarized in Table~\ref{tab:ordering} \cite[Tab.~1]{Daras2004}.
We obtain the extractors
\begin{align*}
    D_{i_1,i_2,i_3} \coloneqq F_{i_3}(F_{i_2}(F_{i_1} \Radon[f_F]\pi(\cdot_1, \cdot_2, \cdot_3))\pi(\cdot_2, \cdot_3))(\cdot_3)
\end{align*}

\begin{table}[]
\caption{Summary of the combination for the extractors from \cite{Daras2004}.}
\vspace{-8pt}
\label{tab:ordering}
\resizebox{\linewidth}{!}{%
\begin{tabular}{c c c c c}
\toprule
$\pi(t,\vartheta_1,\vartheta_2)$ & $(t,\vartheta_1,\vartheta_2)$ & $(t,\vartheta_2,\vartheta_1)$ & $(\vartheta_2,t,\vartheta_1)$ & $(\vartheta_1,t,\vartheta_2)$ \\
\midrule
\multirow{17}{*}{\rotatebox{90}{$F_{i_1},F_{i_2},F_{i_3}$}}
&$F_1,F_1,F_1$ & $F_1,F_1,F_2$ & $F_1,F_1,F_4$ & $F_1,F_1,F_4$ \\
&$F_1,F_1,F_2$ & $F_1,F_1,F_3$ & $F_1,F_2,F_1$ & $F_1,F_2,F_1$ \\
&$F_1,F_1,F_4$ & $F_1,F_1,F_4$ & $F_1,F_2,F_4$ & $F_1,F_3,F_1$ \\
&$F_1,F_2,F_1$ & $F_1,F_2,F_1$ & $F_2,F_1,F_1$ & $F_1,F_3,F_2$ \\
& $F_1,F_3,F_1$ & $F_1,F_2,F_4$ & $F_2,F_2,F_1$ & $F_2,F_1,F_1$ \\
& $F_1,F_3,F_4$ & $F_1,F_4,F_1$ & $F_3,F_1,F_1$ & $F_2,F_2,F_1$ \\
& $F_1,F_4,F_1$ & $F_1,F_4,F_3$ & $F_3,F_1,F_2$ & $F_3,F_1,F_1$ \\
& $F_1,F_4,F_2$ & $F_2,F_1,F_1$ & $F_3,F_1,F_4$ & $F_3,F_2,F_1$ \\
& $F_2,F_1,F_1$ & $F_2,F_1,F_3$ & $F_3,F_2,F_4$ &  \\
& $F_2,F_1,F_4$ & $F_2,F_1,F_4$ & $F_3,F_3,F_4$ &  \\
& $F_2,F_3,F_4$ & $F_2,F_2,F_2$ &  &  \\
& $F_2,F_4,F_1$ & $F_2,F_4,F_1$ &  &  \\
& $F_2,F_4,F_2$ & $F_2,F_4,F_3$ &  &  \\
& $F_3,F_1,F_2$ & $F_3,F_1,F_1$ &  &  \\
& $F_3,F_1,F_4$ & $F_3,F_2,F_4$ &  &  \\
& $F_3,F_2,F_1$ & $F_3,F_4,F_4$ &  &  \\
& $F_3,F_4,F_1$ &  &  &  \\
\bottomrule
\end{tabular}}
\end{table}

\subsubsection{Numerical results}
\label{sec:matching_num}

After preprocessing the (reduced) ModelNet10 dataset
with $64\times64\times64$ voxels,
our discrete {Radon} transform $\mathcal R_{\bftheta(\bfvartheta)}[f_F](t)$ is applied
with different resolutions:
equispaced radii $t\in[-\nicefrac{\sqrt{3}}{2},\nicefrac{\sqrt{3}}{2}]$
and an equispaced grid in spherical coordinates $\bfvartheta=(\vartheta_1,\vartheta_2)$, see \eqref{eq:Sd-cooridnates}.

Instead of the similarity measure from \cite{Daras2004},
we perform a \textit{$1$-nearest neighbor} (1-NN) classification 
on the extracted object features
with different sizes of reference (sub)datasets:
for $N$ classes,
the $N\times10$ dataset is randomly divided into $N\times R$ train and $N\times(10-R)$ test samples,
and the data extractors 
are afterwards used 
to classify the test data.
Therefore each test sample is assigned to the reference sample 
with the smallest distance 
of the extractors.
Here, we consider the Manhattan norm $\|\cdot\|_1$,
the {Euclidean} norm $\|\cdot\|_2$,
and the {Chebyshev} norm $\|\cdot\|_\infty$.
The calculations 
are repeated 20 times with different random splits. 
The mean and standard deviation
are reported in Table~\ref{tab:knn_matching_small}
for the reduced ($N=3$) and in Table~\ref{tab:knn_matching_large}
for the entire ModelNet10 dataset ($N=10$).

We observe the best classification result of $71\%$,
cf.~Table~\ref{tab:knn_matching_small},
which is significantly better than guessing (33\%)
on the $3$-class problem with bathtubs, chairs, and dressers.
For the $10$-class problem in Table~\ref{tab:knn_matching_large},
we observe the best classification accuray of $36\%$
with the Manhattan norm. 
The respective distance map 
is provided in Fig.~\ref{fig:dist_map_shape_matching} (left),
from which the class structure is just slightly visible.
The confusion map in Fig.~\ref{fig:dist_map_shape_matching} (right)
indicates uncertainties between several classes;
the maximal rate of correct assignments is $64\%$.

The similarity measure used in \cite{Daras2004}
provides even worse results.
We were not able to reproduce the exact numbers from \cite{Daras2004}
due to 
unavailable code,
the lack of documentation of the preprocessing,
and the ambiguity 
how the {Radon} transform is calculated.
However,
compared to their dataset,
our choice is much more difficult 
with more varieties in the shapes 
within each class.

\begin{table}[]
    \caption{Accuracy results for the Radon shape matching \cite{Daras2004} with $1$-NN classification 
    on the reduced ModelNet10 dataset ($N=3$): bathtubs, chairs, and dressers.}
    \vspace{-8pt}
    \label{tab:knn_matching_small}
    \resizebox{\linewidth}{!}{%
    \begin{tabular}{ccc @{\quad} ccc @{\quad}}
        \toprule
        radii & angles & $R$ & $\|\cdot\|_1$ & $\|\cdot\|_2$ & $\|\cdot\|_\infty$ \\
        \midrule
        $512$ & $(20,16)$ & 1 & $0.451\pm0.156$ & $0.425\pm0.136$ & $0.412\pm0.163$ \\
                          && 3 & $0.480\pm0.127$ & $0.487\pm0.121$ & $0.457\pm0.115$ \\
                          && 5 & $0.633\pm0.100$ & $0.588\pm0.104$ & $0.544\pm0.097$ \\  
              & $(30,21)$ & 1 & $0.491\pm0.167$ & $0.484\pm0.133$ & $0.514\pm0.176$ \\
                          && 3 & $0.531\pm0.153$ & $0.507\pm0.126$ & $0.553\pm0.155$ \\
                          && 5 & $0.690\pm0.107$ & $0.635\pm0.102$ & $0.571\pm0.098$ \\ 
        $2048$& $(20,16)$ & 1 & $0.502\pm0.120$ & $0.432\pm0.145$ & $0.333\pm0.156$ \\
                          && 3 & $0.555\pm0.095$ & $0.461\pm0.111$ & $0.361\pm0.111$ \\
                          && 5 & $0.687\pm0.085$ & $0.556\pm0.067$ & $0.574\pm0.101$ \\ 
              & $(30,21)$ & 1 & $0.513\pm0.111$ & $0.415\pm0.164$ & $0.425\pm0.178$ \\
                          && 3 & $0.585\pm0.089$ & $0.451\pm0.137$ & $0.466\pm0.153$ \\
                          && 5 & $0.712\pm0.069$ & $0.680\pm0.116$ & $0.666\pm0.124$ \\ 
        \bottomrule
    \end{tabular}}
\end{table}

\begin{table}[]
    \caption{Accuracy results for the Radon shape matching \cite{Daras2004} with $1$-NN classification 
    on the entire ModelNet10 dataset ($N=10$).}
    \vspace{-8pt}
    \label{tab:knn_matching_large}
    \resizebox{\linewidth}{!}{%
    \begin{tabular}{ccc @{\quad} ccc @{\quad}}
        \toprule
        radii & angles & $R$ & $\|\cdot\|_1$ & $\|\cdot\|_2$ & $\|\cdot\|_\infty$ \\
        \midrule
        $2048$ & $(30,21)$ & 1 & $0.240\pm0.044$ & $0.222\pm0.045$ & $0.215\pm0.054$ \\
                          && 3 & $0.312\pm0.048$ & $0.312\pm0.054$ & $0.277\pm0.038$ \\
                          && 5 & $0.363\pm0.044$ & $0.351\pm0.060$ & $0.333\pm0.051$ \\ 
        \bottomrule
    \end{tabular}}
\end{table}

\begin{figure}
    \centering
    \includegraphics[width=0.49\linewidth]{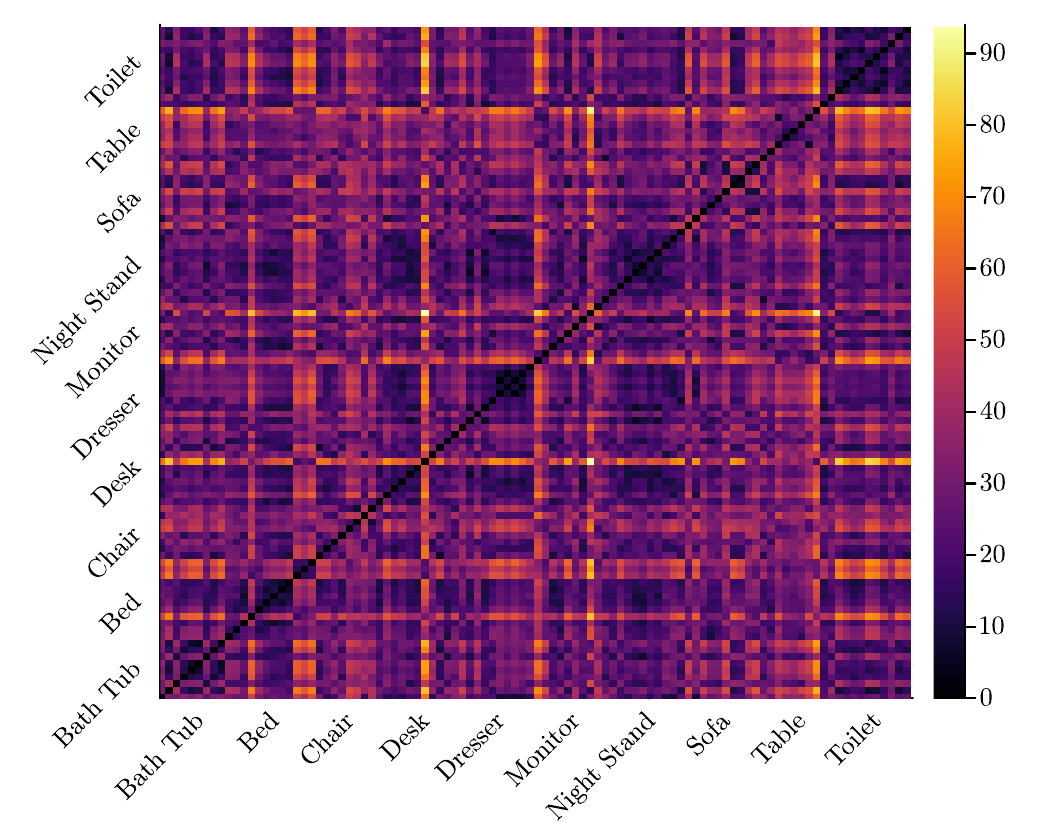}
    \includegraphics[width=0.49\linewidth]{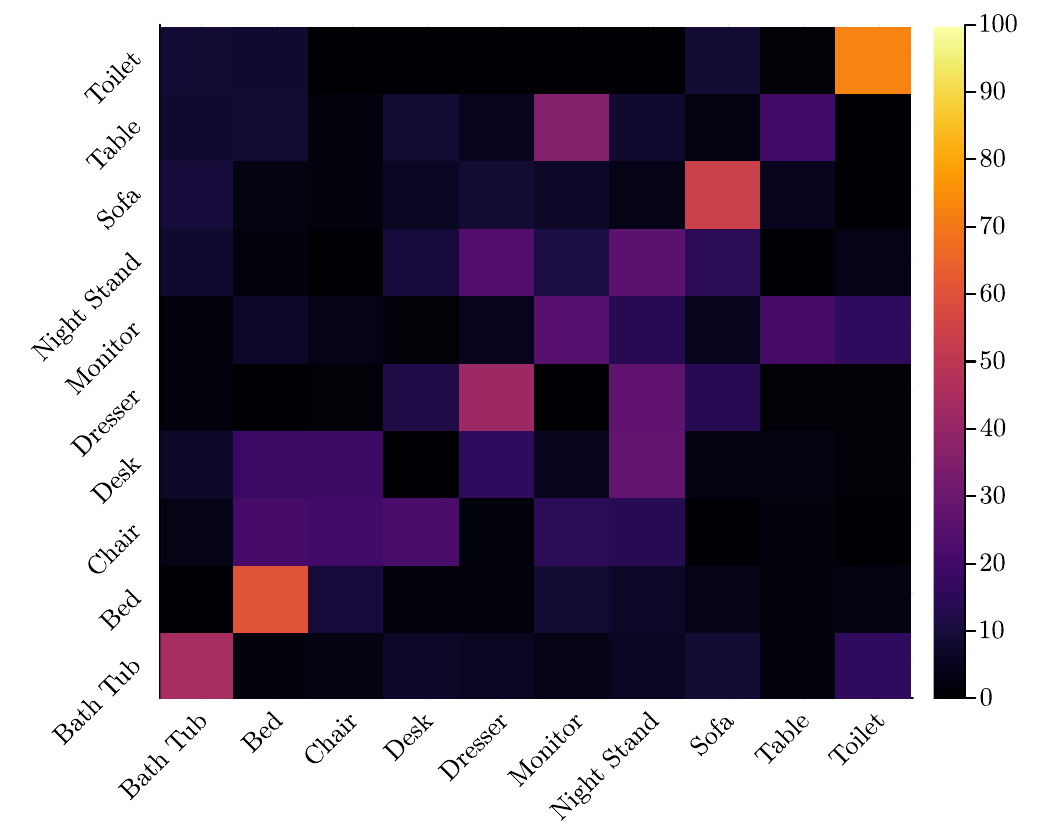}
    \vspace{-13pt}
    \caption{Distance map (left)
    and confusion map (in \%, right) for the Radon shape matching of
    the entire ModelNet10 dataset 
    with respect to the Manhattan norm $\|\cdot\|_1$,
    i.e. the best performing distance in Table~\ref{tab:knn_matching_large}.}
    \label{fig:dist_map_shape_matching}
\end{figure}

\subsection{Classification via NR-CDT \cite{Beckmann2025a}}
\label{sec:nrcdt}

The {Radon} transform can be utilized 
for the classification of images 
using the so-called R-CDT approach \cite{Kolouri2016}. 
Its extension, the NR-CDT \cite{Beckmann2024a, Beckmann2025}, 
is invariant to affine transformations,
and was applied to
pattern recognition \cite{Beckmann2025a}.
Differently,
we apply the theory on voxels
and not on meshes (i.e. point clouds),
which makes the computation of the 3D {Radon} transform crucial.
We first revise the construction of the NR-CDT
for functions,
and then provide the numerical results 
on the ModelNet10 dataset 
for comparison with the feature extractors from Section~\ref{sec:matching}.

\subsubsection{Normalized {Radon} cumulative distribution transform}
\label{sec:nrcdtthm}

Let $f \in \Lebesgue^1(\R^d)$ be nonnegative
and $\int_{\R^d} f(\zbx) \d \zbx = 1$.
Following \cite{Beckmann2025a},
we define the \emph{Radon cumulative distribution transform} (Radon-CDT)
of $f$ by 
\begin{equation}
    \widehat{\Radon_{\bftheta}[f]} \colon \R \to \R,
    \quad \xi \mapsto \inf \Bigl\{s \in \R \Bigm| \int_{-\infty}^s \Radon_{\bftheta}[f](t) \d t > \xi\Bigr\}.
\end{equation}
This is the quantile function, i.e., the generalized inverse
of the cumulative distribution function 
of $\Radon_{\bftheta}[f]$, see \cite{Beckmann2024a,Beckmann2025a}.
We define the \emph{mean} and \emph{standard derivation} of a function $g \in \Lebesgue^\infty([0,1])$ by
\begin{equation*}
    \mathrm{mean}(g) \coloneqq \int_0^1 g(s) \d s,
    \;\; 
    \mathrm{std}(g)^2 \coloneq \int_0^1 (g(s) - \mathrm{mean}(g))^2 \d s.
\end{equation*}
Following \cite{Beckmann2025a},
the \emph{normalized Radon-CDT} (NR-CDT) 
\begin{equation}
    \NRCDT_{\zb\theta} [f] \colon \R \to \R,
    \; 
    t 
    \mapsto
    \frac
    {\widehat{\Radon_{\bftheta}[f]}(t) 
    - 
    \mathrm{mean}(\widehat{\Radon_{\bftheta}[f]})}
    {\mathrm{std}(\widehat{\Radon_{\bftheta}[f]})},
\end{equation}
is well defined 
if $\supp(f) \subset \R^d$ is compact 
and $\dim(\supp(f)) > d-1$, cf.~\cite[Prop.~4, Lem.~1]{Beckmann2025a}.

Finally, we define the \emph{max-normalized} R-CDT (\mNRCDT)
\begin{equation}
    \maxNRCDT [f] \colon \R \to \R,
    \quad 
    \xi \mapsto  
    \sup_{\bftheta \in \S^{d-1}} \NRCDT_{\bftheta} [f](\xi).
\end{equation}
The \mNRCDT{} is invariant 
under affine transformations, see \cite[Thm.~1]{Beckmann2024a}.
More specifically,
consider for a function $f$
its affine transformation 
\begin{equation} \label{eq:affine}
    f_{\zbA, \zby} : \R^d \to \R,
    \; \zbx \mapsto f(\zbA^{-1}(\zbx - \zby)),
\end{equation}
for $\zbA \in \GL(d)$ and $\zby \in \R^d$,
then it holds
\begin{equation} \label{eq:mnrcdt-affine}
    \maxNRCDT[f_{\zb{A}, \zby}] = \maxNRCDT[f].
\end{equation}

\subsubsection{Numerical results}
\label{sec:nrcdtnum}

We apply the \mNRCDT{} 
on both datasets,
i.e. the reduced (bathtubs, chairs, and dressers)
and entire ModelNet10 dataset,
where each class consists of ten samples.
The discretization 
uses three parameters:
i) the number of evaluation points for the CDT;
ii) the equispaced radii
as in Section~\ref{sec:matching_num};
iii) the angles $\bftheta$,
which are not restricted to grids
unlike in Section~\ref{sec:matching}.
There are many notions of ``almost equispaced'' points $\bftheta\in\S^2$ 
such as spherical designs \cite{Womersley2018,Graef2011}
or quasi-Monte Carlo designs \cite{Brauchart2014,Hertrich2024}.
We chose
the {Fibonacci} points \cite{Hannay2004,Gonzalez2010}
given in spherical coordinates \eqref{eq:Sd-cooridnates} by
\begin{align}
    \label{eq:fibo_spherical}
    \bfvartheta 
    = \left[\begin{smallmatrix}
        \vartheta_1 \\
        \vartheta_2
    \end{smallmatrix}\right]
    \in \left\{\left[\begin{smallmatrix}
        i \varphi \\
        \arccos\bigl(1- \tfrac{2 (i-1) + 1}{n}\bigr)
    \end{smallmatrix}\right] \biggm| i \in \ii{n}\right\},
\end{align}
with the golden angle $\varphi \coloneqq \pi(3 - \sqrt{5})$.
The resulting \mNRCDT{}s
are compared  
with different norms,
cf.~Section~\ref{sec:matching_num}.

For the shape matching
on the ModelNet10 dataset,
we perform a $1$-NN classification
with splitting sizes $R\in\{1,3,5\}$
as in Section~\ref{sec:matching_num}.
The accuracies (mean$\pm$std.)
for different discretizations of the sphere $\sphere^2$
and the radii 
based on $20$ runs of the $1$-NN are reported 
in Table~\ref{tab:knn_small} for the reduced dataset ($N=3$)
and in Table~\ref{tab:knn_large} for the entire dataset ($N=10$).

We observe accuracies up to $93\%$ for the shape matching,
see Table~\ref{tab:knn_small}, in the reduced classes (bathtubs, chairs, and dressers),
which is a significant improvement 
to the Radon shape matching from Section~\ref{sec:matching},
cf.~Table~\ref{tab:knn_matching_small}.
Especially, in the \mNRCDT{} space, 
the chairs are perfectly distinguished from both other classes.

For the entire dataset,
we obtain an accuracy up to $47\%$,
which is significantly better than random guessing ($10\%$),
as well as the accuracies from Section~\ref{sec:matching}, cf.~Table~\ref{tab:knn_matching_large}.
The separability of the classes in the \mNRCDT-space 
is clearly visible from the distance map 
in Fig.~\ref{fig:dist_map} (left).
Moreover, in the confusion map, see Fig.~\ref{fig:dist_map} (right),
there are few classes (nightstands, tables, sofas, and bathtubs)
that cannot be assigned well.
However, the diagonal is more dominant than in Fig.~\ref{fig:dist_map_shape_matching} (right).

\begin{table}[]
    \caption{Accuracy results for the \mNRCDT{} \cite{Beckmann2025a} with the $1$-NN classification 
    on the reduced ModelNet10 dataset $(N=3)$: bathtubs, chairs, and dressers.}
    \vspace{-8pt}
    \label{tab:knn_small}
    \resizebox{\linewidth}{!}{%
    \begin{tabular}{ccc @{\quad} lll}
        \toprule
        radii & angles & $R$
        & $\|\cdot\|_1$ & $\|\cdot\|_2$ & $\|\cdot\|_\infty$ 
        \\
        \midrule
        $64$ & $64$ 
        & $1$ & $0.853\pm0.089$ & $0.826\pm0.080$ & $0.703\pm0.085$ 
        \\
        && $3$ & $0.890\pm0.061$ & $0.852\pm0.075$ & $0.776\pm0.087$ 
        \\
        && $5$ & $0.922\pm0.067$ & $0.900\pm0.063$ & $0.783\pm0.110$ 
        \\
        \midrule
        $128$ & $128$ 
        & $1$ & $0.901\pm0.040$ & $0.855\pm0.073$ & $0.757\pm0.078$ 
        \\
        && $3$ & $0.904\pm0.048$ & $0.866\pm0.059$ & $0.759\pm0.094$ 
        \\
        && $5$ & $0.900\pm0.066$ & $0.869\pm0.100$ & $0.820\pm0.094$ 
        \\
        \midrule 
        $256$ & $256$ 
        & $1$ & $0.892\pm0.031$ & $0.835\pm0.100$ & $0.696\pm0.083$ 
        \\
        && $3$ & $0.914\pm0.093$ & $0.869\pm0.068$ & $0.785\pm0.079$ 
        \\
        && $5$ & $0.933\pm0.094$ & $0.896\pm0.070$ & $0.836\pm0.076$ 
        \\
        \bottomrule
    \end{tabular}}
\end{table}

\begin{table}[]
    \caption{Accuracy results for the \mNRCDT{} \cite{Beckmann2025a} 
    with the $1$-NN classification 
    on the entire ModelNet10 dataset ($N=10$).}
    \vspace{-8pt}
    \label{tab:knn_large}
    \resizebox{\linewidth}{!}{%
    \begin{tabular}{ccc @{\quad} lll}
        \toprule
        radii & angles & $R$ 
        & $\|\cdot\|_1$ & $\|\cdot\|_2$ & $\|\cdot\|_\infty$ 
        \\
        \midrule 
        $256$ & $256$ 
        & $1$ & $0.307\pm0.040$ & $0.312\pm0.051$ & $0.245\pm0.037$ 
        \\
        && $3$ & $0.375\pm0.041$ & $0.401\pm0.052$ & $0.312\pm0.046$ 
        \\
        && $5$ & $0.419\pm0.047$ & $0.472\pm0.042$ & $0.388\pm0.060$ 
        \\
        \bottomrule
    \end{tabular}}
\end{table}

\begin{figure}
    \centering
    \includegraphics[width=0.49\linewidth]{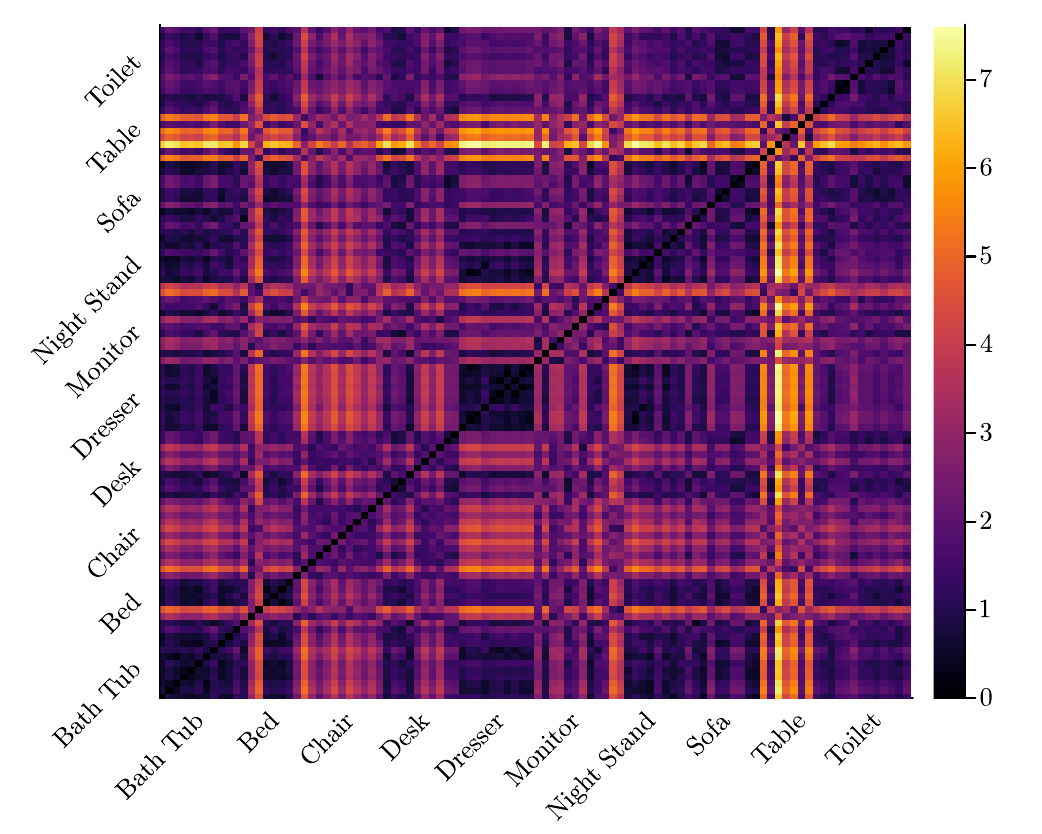}
    \includegraphics[width=0.49\linewidth]{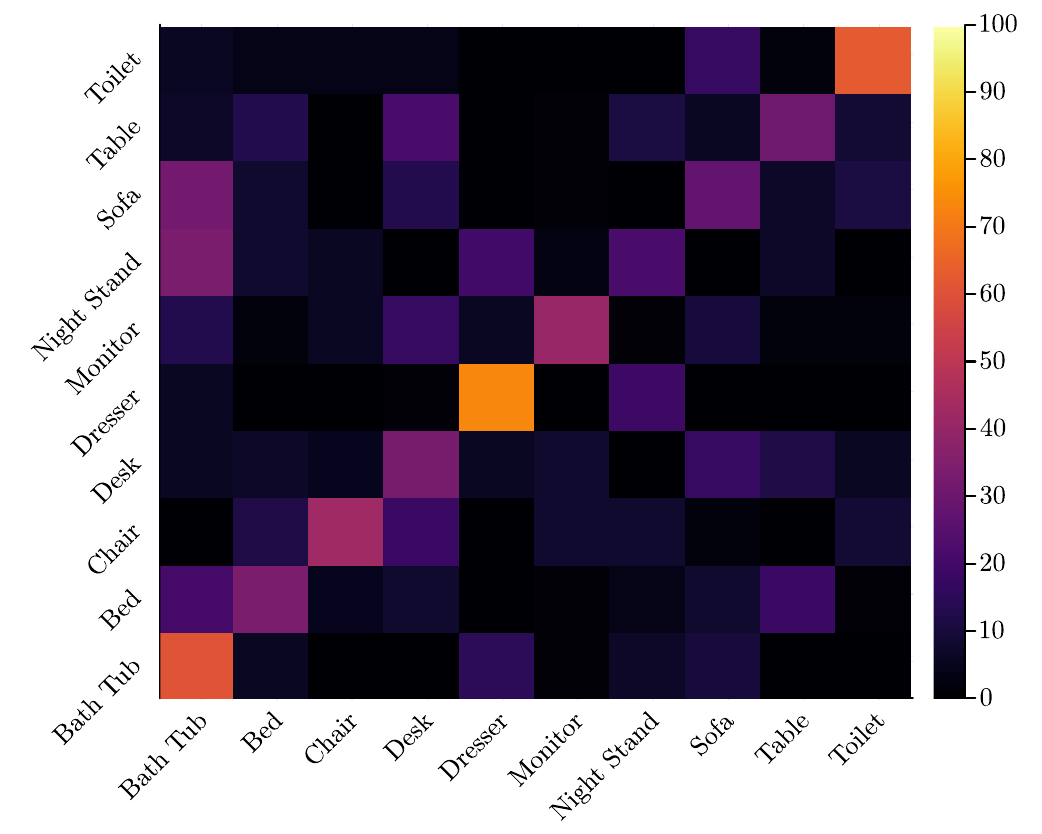}
    \vspace{-13pt}
    \caption{Distance map (left)
    and confusion map (in \%, right)
    of the entire ModelNet10 dataset in the \mNRCDT-space 
    with respect to the Euclidean norm $\|\cdot\|_2$,
    i.e. the distance with the best outcome of the accuracies
    from Table~\ref{tab:knn_large}.}
    \label{fig:dist_map}
\end{figure}

\subsection{Classification of affinely transformed objects}
\label{sec:affine_objects}

Contrary to the previous subsections,
we consider affienly transformed objects in 3D.
Here, we take the first sample from each class of the ModelNet10 dataset, 
cf.~Fig.~\ref{fig:modelnet10_1} and~\ref{fig:modelnet10_2},
as a template,
and generate ten objects by applying random affine transformations
(anisotropic scaling, shearing, shifting, rotation) to the template, see \eqref{eq:affine}.
The numerical studies ($1$-NN) from Sections~\ref{sec:matching_num} and~\ref{sec:nrcdtnum}
are repeated
for the densest discretization. 
We obtain an accuracy for the Radon shape matching \cite{Daras2004}
of up to $42\%$,
which is outperformed by the \mNRCDT{} embedding \cite{Beckmann2025a}
with a nearly perfect classification result of $99\%$,
as theoretically expected, cf.~\eqref{eq:mnrcdt-affine}.

\section{Approximation and Means of Empirical Data}
\label{sec:continuous}

In contrast to Section~\ref{sec:discrete},
which relies on a voxel-based data structure,
here we are interested in a free-support setting for measures. 
Firstly, we consider clustering in 2D and 3D,
more precisely, we aim to fit a given empirical measure
by a mixture of measures on cubes 
via the sliced {Wasserstein} distance,
see Section~\ref{sec:clustering_empirical_measure}.
Secondly,
we compute sliced {Wasserstein} barycenters,
see Section~\ref{sec:sliced_wasserstein_barycenter}.

\subsection{Optimal histogram via sliced {Wasserstein} distance}
\label{sec:clustering_empirical_measure}

\begin{figure}
    \centering
    \includegraphics[width=0.49\linewidth, clip=true, trim=15pt 10pt 5pt 10pt]{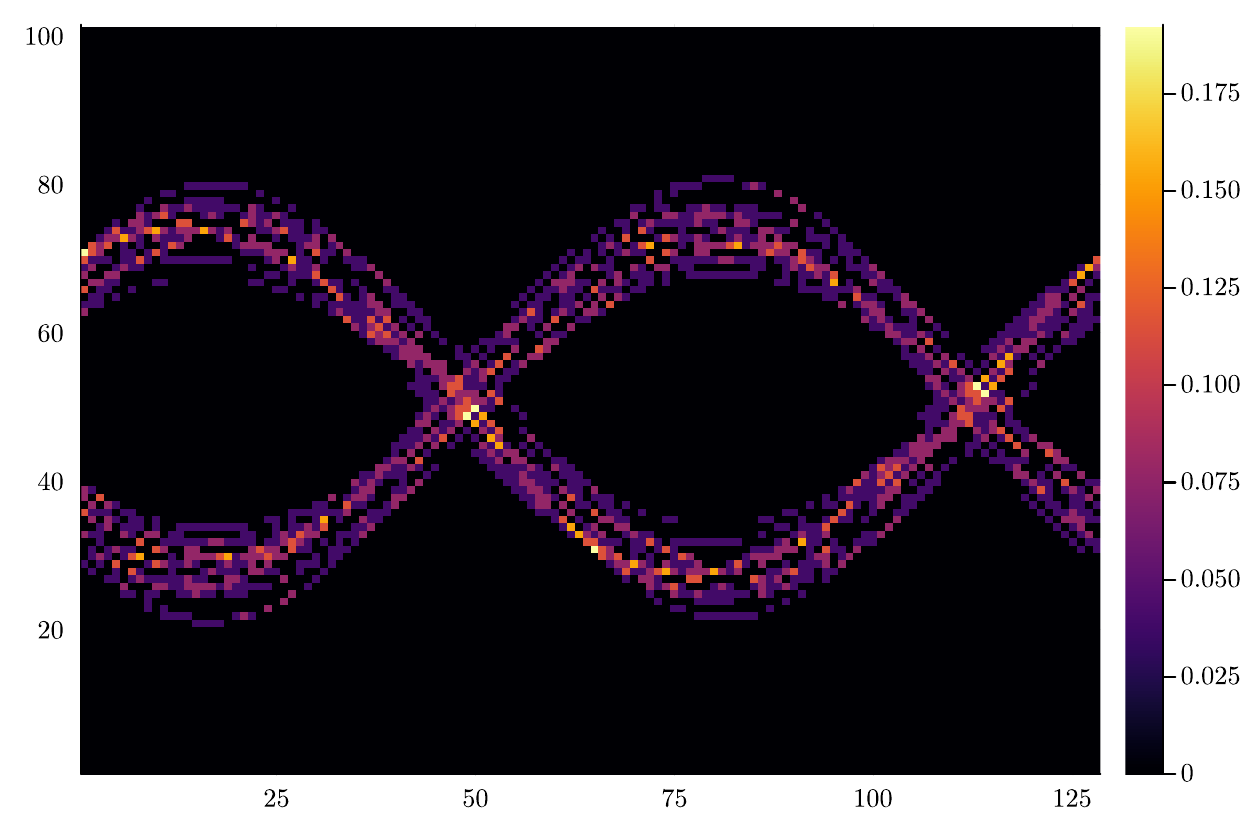} 
    \includegraphics[width=0.49\linewidth, clip=true, trim=15pt 10pt 5pt 10pt]{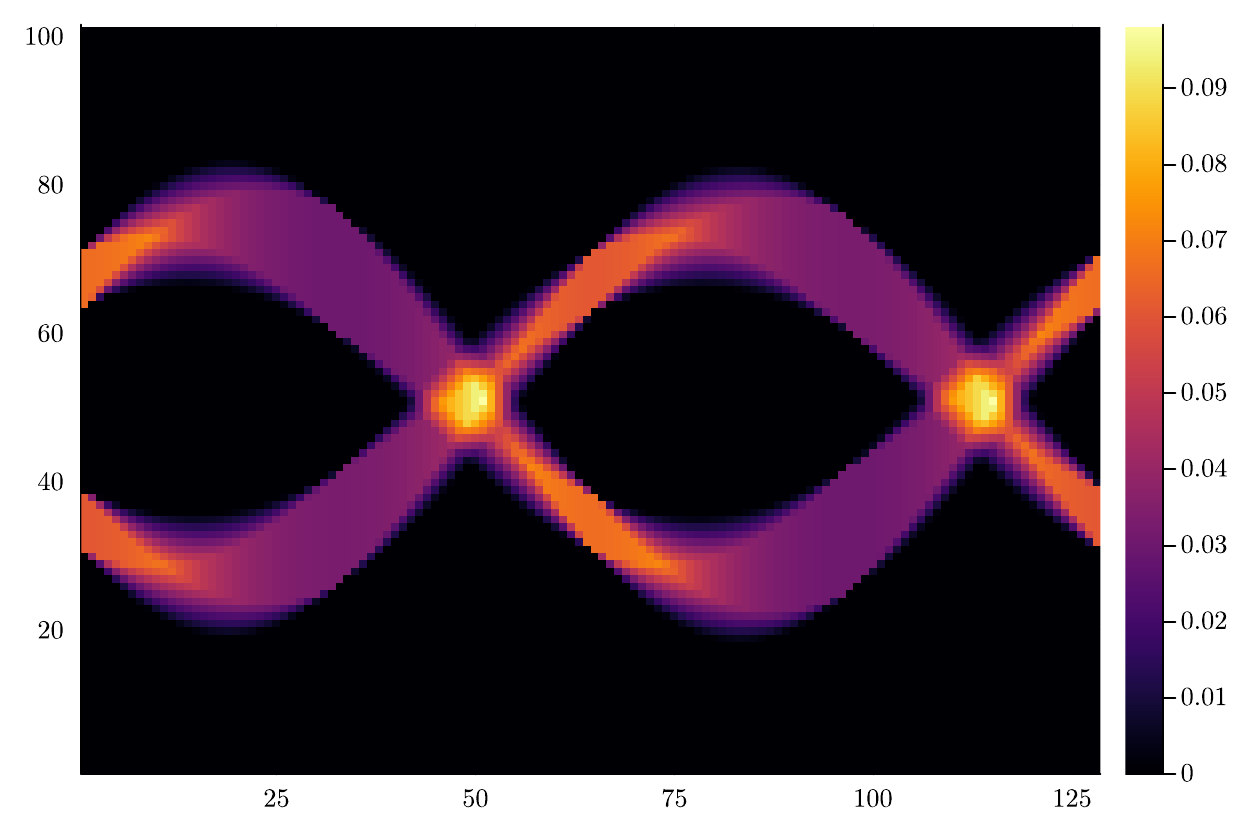} 
    \caption{2D Radon transforms
    of the empirical (target) measure $\Delta_n$ (left)
    and the uniform mixture $\Upsilon_k$
    solving \eqref{eq:minSW}, 
    corresponding to Fig.~\ref{fig:SWiter} (top).}
    \label{fig:SWsinograms}
\end{figure}

\begin{figure}
    \centering
    \includegraphics[width=0.89\linewidth, clip=true, trim=0pt 0pt 0pt 0pt]{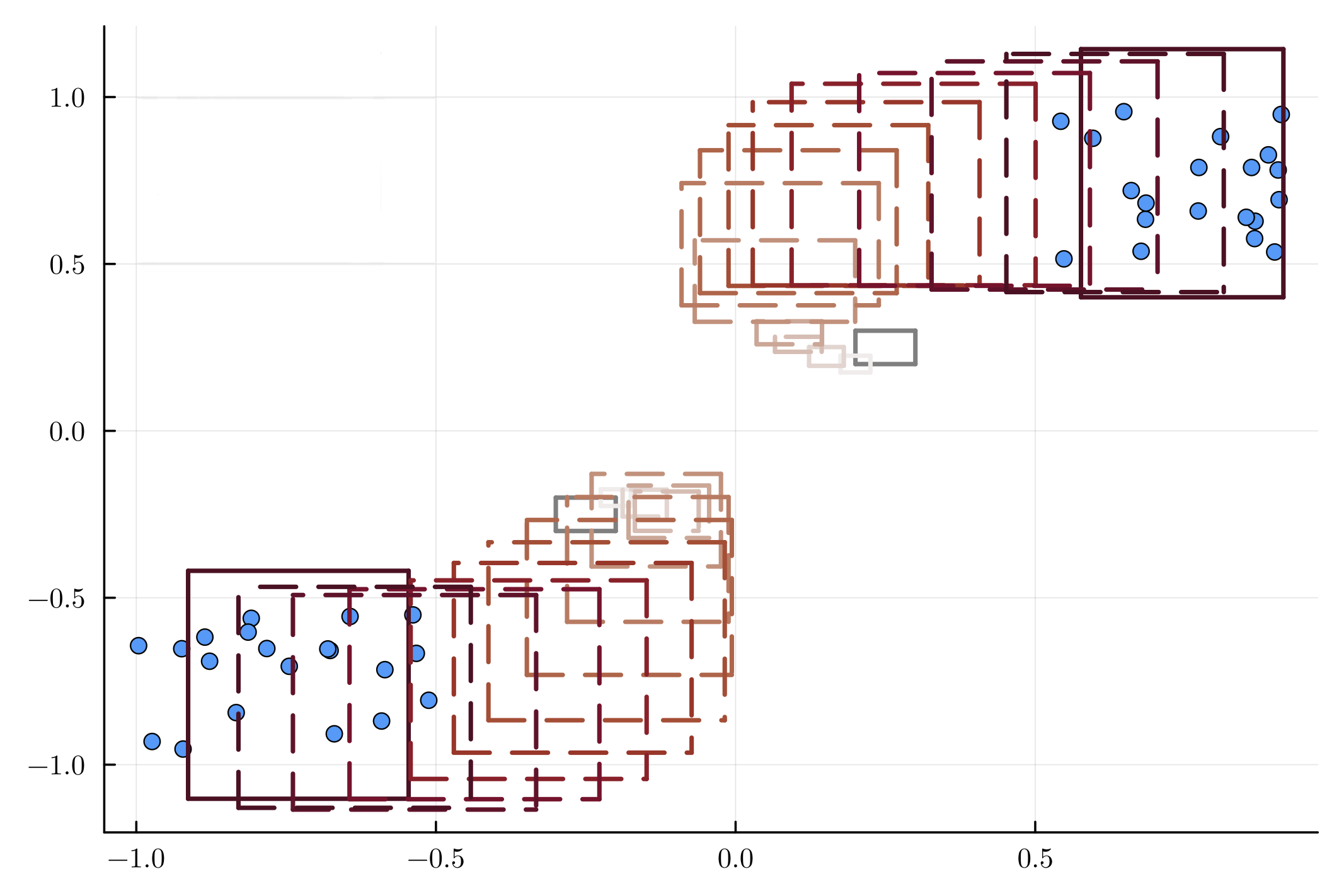}
    \includegraphics[width=0.09\linewidth, clip=true, trim=10pt -10pt 5pt 0pt]{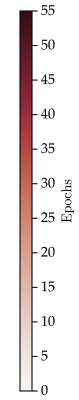} \\
    \includegraphics[width=0.89\linewidth, clip=true, trim=130pt 20pt 110pt 45pt]{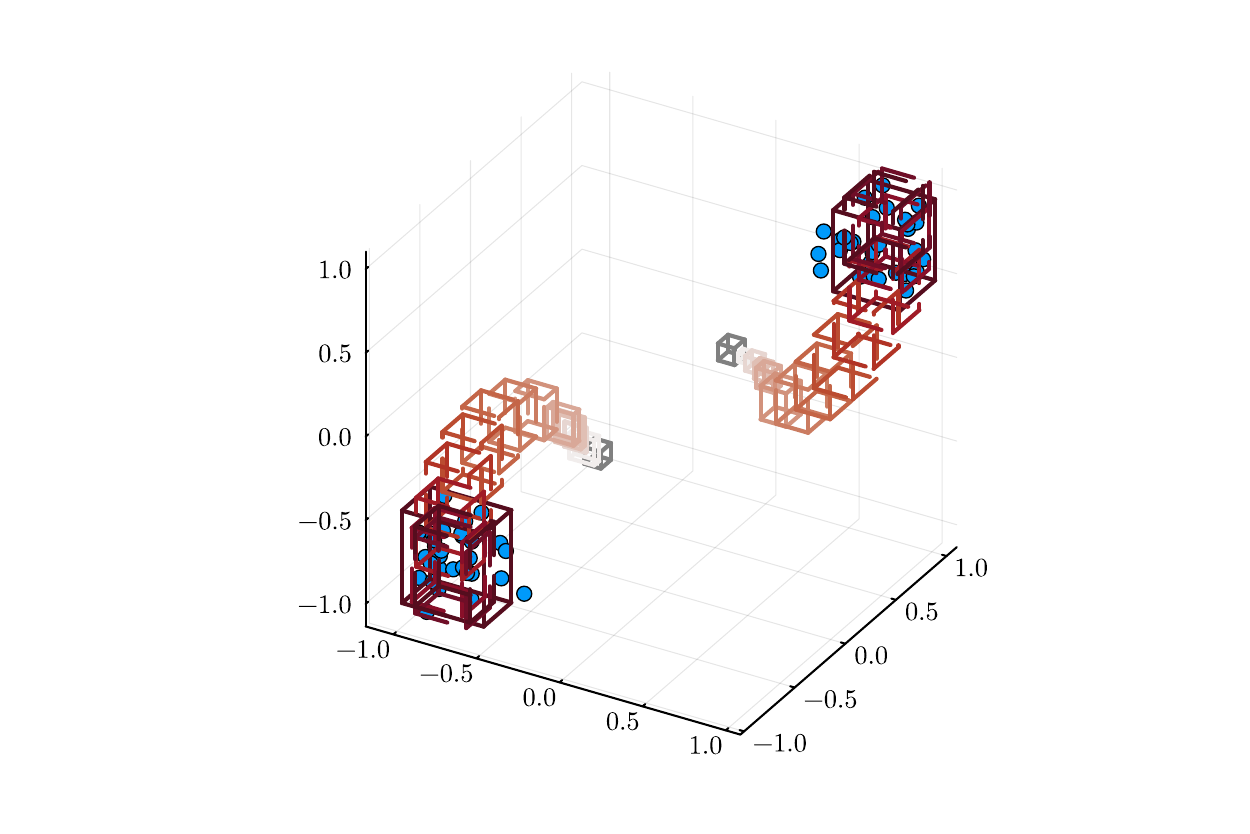}
    \includegraphics[width=0.09\linewidth, clip=true, trim=10pt -100pt 5pt 0pt]{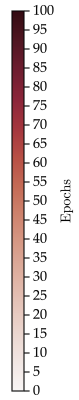}
    \caption{Visualization of the iterates in $\R^d$ with
    $d = 2$ (top) and $d=3$ (bottom),
    corresponding to Fig.~\ref{fig:SWsinograms}.
    The uniform mixture $\Upsilon_2$ (red cubes)
    approximates the optimal histogram for the
    empirical measure $\Delta_n$ (blue dots).}
    \label{fig:SWiter}
\end{figure}

Approximating an empirical measure 
$\Delta_n \coloneqq \nicefrac{1}{n}\sum_{i = 1}^n \delta_{\zbx_i}$,
where $\delta_{\zbx}$ denotes the {Dirac} measure 
at a point $\zbx \in \X\subset\R^d$,
by another distribution 
is a important task \cite{BRPP15,KolNadSimBadRoh19}.
We are interested in linear combinations 
\begin{equation} \label{eq:upsilon}
\Upsilon_k \coloneqq \sum_{j = 1}^k \gamma_j u_j,
\end{equation}
where $u_j$ is the uniform measure on a cube $$(\zbc_j - \zbw_j, \zbc_j+\zbw_j]$$ with center $\zbc_j \in \X$ 
and width $\zbw_j \in \X$,
and $\gamma_j \coloneqq \nicefrac{u_j(\X)}{\sum_{\ell = 1}^k u_\ell(\X)}$.
Note that the ansatz of minimizing the {Kullback}--{Leibler} divergence or any other Csiszár divergence to the given measure
is not defined in this case \cite{Csi1972}.

An alternative approach is to consider 
the {sliced {Wasserstein} distance}.
To this end, the \emph{{Wasserstein} distance} \cite{San15}
of measures $\chi, \eta \in \P(\R)$ is defined by
\begin{equation}
    \Wasserstein_2^2(\chi, \eta) 
    \coloneqq 
    \inf_{\gamma \in \Pi(\chi, \eta)} \int_{\R\times\R} |x-y|^2 \d \gamma(x,y),
\end{equation}
where $\Pi$ is the set of measures $\gamma\in\P(\R\times\R)$ with marginals $\chi$ and $\eta$.
We have
\begin{align}
    \label{eq:W}
    \Wasserstein_2^2(\chi, \eta) = \int_0^1 |F_{\chi}^{[-1]}(s) - F_{\eta}^{[-1]}(s)|^2 \d \rho(s)
\end{align}
for some reference measure $\rho \in \P(\R)$
with $\chi \ll \rho$,
with the cumulative distribution function $F_{\chi}(s) \coloneqq \chi((-\infty, s])$
and its generalized inverse $g^{[-1]}(t) \coloneqq \inf\{s \mid g(s) > t\}$.
The \emph{sliced {Wasserstein} distance} 
of $\mu, \nu \in \P(\R^d)$
is defined by
\begin{equation}
    \label{eq:SW}
    \SWasserstein_2^2(\mu, \nu) 
    \coloneqq \int_{\sphere^{d-1}} \Wasserstein_2^2(\Radon_{\bftheta}[\mu], \Radon_{\bftheta}[\nu])
    \d u_{\sphere^{d-1}}(\bftheta),
\end{equation}
where $u_{\sphere^{d-1}}$ is the uniform measure on $\sphere^{d-1}$
and the Radon transform of measures,
corresponding to \eqref{eq:radon} for the density,
is defined by
\begin{equation} \label{eq:Radon-measure}
    \Radon_{\bftheta}[\mu] 
    \coloneqq \langle \cdot , \bftheta\rangle_{\#} \mu
    = \mu\circ(\langle \cdot , \bftheta\rangle)^{-1}
    \in \P(\R),
    \;\; \bftheta \in \sphere^{d-1}.
\end{equation}
Compared to the Wasserstein distance, its sliced version combines similar properties with faster computation, cf.\ \cite{nguyen2025introduction,PeyCut19,QueBeiSte23,QueBueSte24}.

Then, we propose to approximate the empirical measure $\Delta_n$ by solving
\begin{equation}
    \label{eq:minSW}
    \argmin_{\substack{\zbc, \zbw \in \X^k }} \SWasserstein_2^2(\Upsilon_k, \Delta_n), 
\end{equation}
where $\Upsilon_k$ is parameterized by $\zb c$ and $\zb w$.
For the numerical computation, 
the Wasserstein distance $\Wasserstein_2^2$ is computed by \eqref{eq:W},
the Radon transform of $\Upsilon_k$ is computed by a linear combination of \eqref{eq:Radon-cube},
and the Radon transform \eqref{eq:Radon-measure}
of the empirical measure is given by
\begin{equation} \label{eq:Radon-empirical}
    \Radon_{\bftheta}[\Delta_n] = \sum_{i = 1}^n \delta_{\langle \zbx_i, \bftheta\rangle},
\end{equation}
where we apply nearest neighbor interpolation.

We provide two experiments
for $d \in\{2,3\}$
and consider $\Delta_n$ sampled from the uniform measure on $[-1,-\nicefrac{1}{2}]^d \cup [\nicefrac{1}{2},1]^d$
with $n = 40$ samples for $d = 2$ and $n = 60$ for $d = 3$.
We use $101$ radii equispaced in $[-\sqrt{d}, \sqrt{d}]$,
and $128$ angles
equispaced on $\sphere^1$, 
and {Fibonacci} points~\eqref{eq:fibo_spherical} on $\sphere^2$.
We solve \eqref{eq:minSW},
with the ADAM optimizer \cite{Kingma2014} using 100 epochs, 
learning rate $0.05$, smoothing parameters $(0.9, 0.99)$,
and initialization $\zbc_{1,2} = \pm \nicefrac{1}{4}\zbe \in \R^d$
and $\zbw_{1,2} = \nicefrac{1}{10} \zbe \in \R^d$.
The Radon transforms of the target $\Delta_n$ and optimal measure $\Upsilon_2$
are visualized in Fig.~\ref{fig:SWsinograms}.
The iterates are visualized in Fig.~\ref{fig:SWiter}.
We observe that the computed measures $\Upsilon_n$
perfectly cluster in both settings.
The corresponding Radon transforms
reproduce the main character 
of the target.

\subsection{Sliced {Wasserstein} barycenter}
\label{sec:sliced_wasserstein_barycenter}

The \textit{sliced {Wasserstein} barycenter} \cite{BRPP15}
of two probability measures $\mu_1, \mu_2 \in \P_2(\R^d)$
with finite second moment,
such that $\Radon[\mu_i] \in \P_2(\sphere^{d-1}\times\R)$ for $i \in \ii{2}$,
is defined for some $\lambda \in (0,1)$ by
\begin{align}
    \label{eq:BSW}
    \argmin_{\mu \in \P(\R^d)} 
    \lambda \SWasserstein_2^2(\mu_1, \mu) 
    + (1-\lambda) \SWasserstein_2^2(\mu_2, \mu),
\end{align}
where the sliced Wasserstein distance is defined 
in \eqref{eq:SW}.

The computation of
the {Radon} transform 
is performed in the manner of Section~\ref{sec:clustering_empirical_measure}
with $42$ angles $\bftheta$ and $31$ radii $t$.
The reference measures $\mu_j$ are either two hemispheres rotated by $90^\circ$ around the z-axis,
or a sphere and a hemisphere, see Fig.~\ref{fig:vis_barycenter}.
The minimization 
uses the same setting as in Section~\ref{sec:clustering_empirical_measure}
and relies on a free support discretization
of $\mu$ in the form $\Upsilon_k$, see \eqref{eq:upsilon}, 
with $k=200$ centers in $\R^3$,
The barycenters are visualized in Fig.~\ref{fig:vis_barycenter},
and the corresponding Radon transform in Fig.~\ref{fig:sw_bary_sino}.

\begin{figure}
    \resizebox{\linewidth}{!}{%
    \begin{tabular}{@{} l @{} l @{}}
    \includegraphics[width=0.6\linewidth,clip=true,trim=20pt 0pt 10pt 0pt]{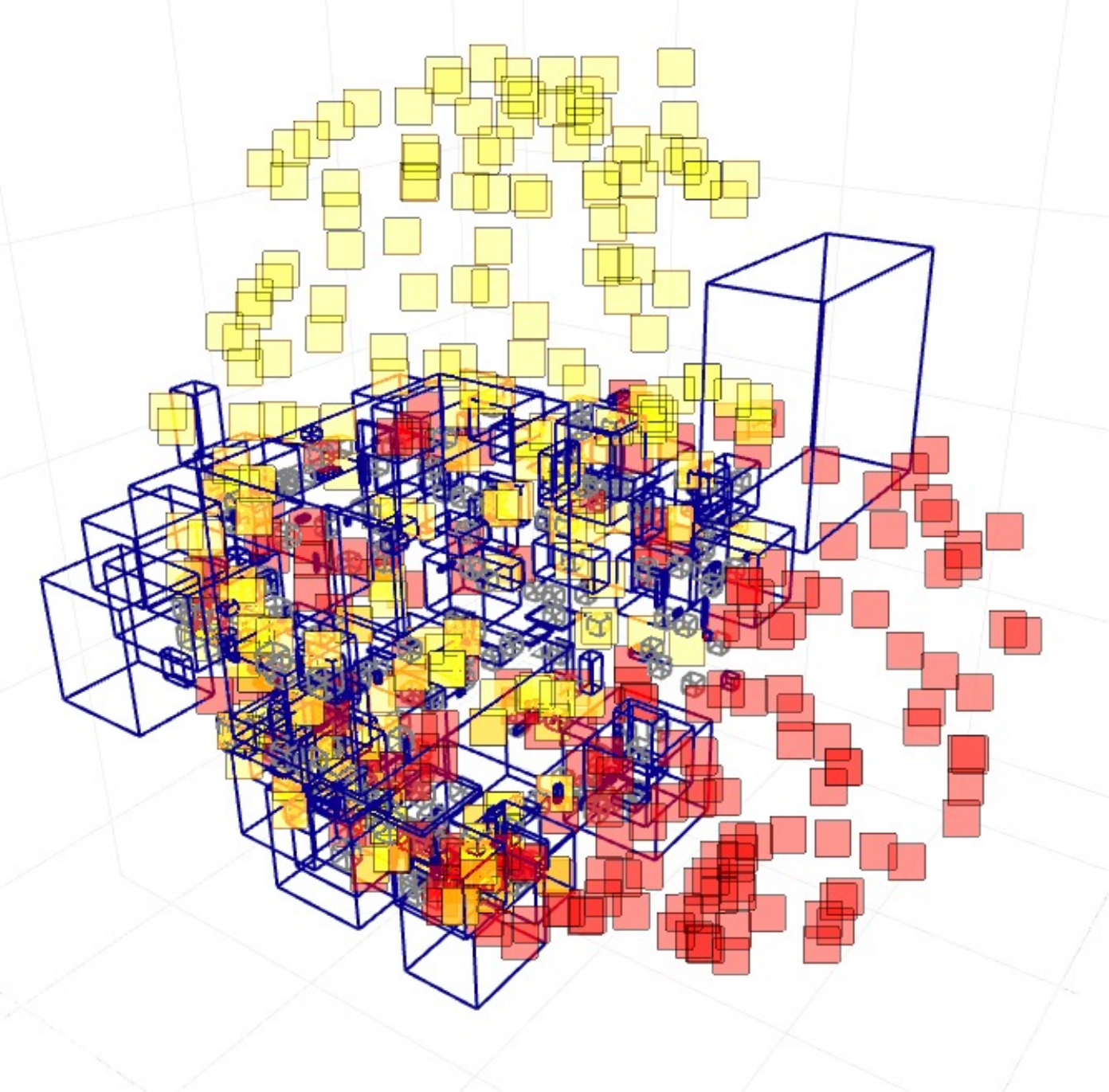}
    & \includegraphics[width=0.6\linewidth,clip=true,trim=20pt 0pt 10pt 0pt]{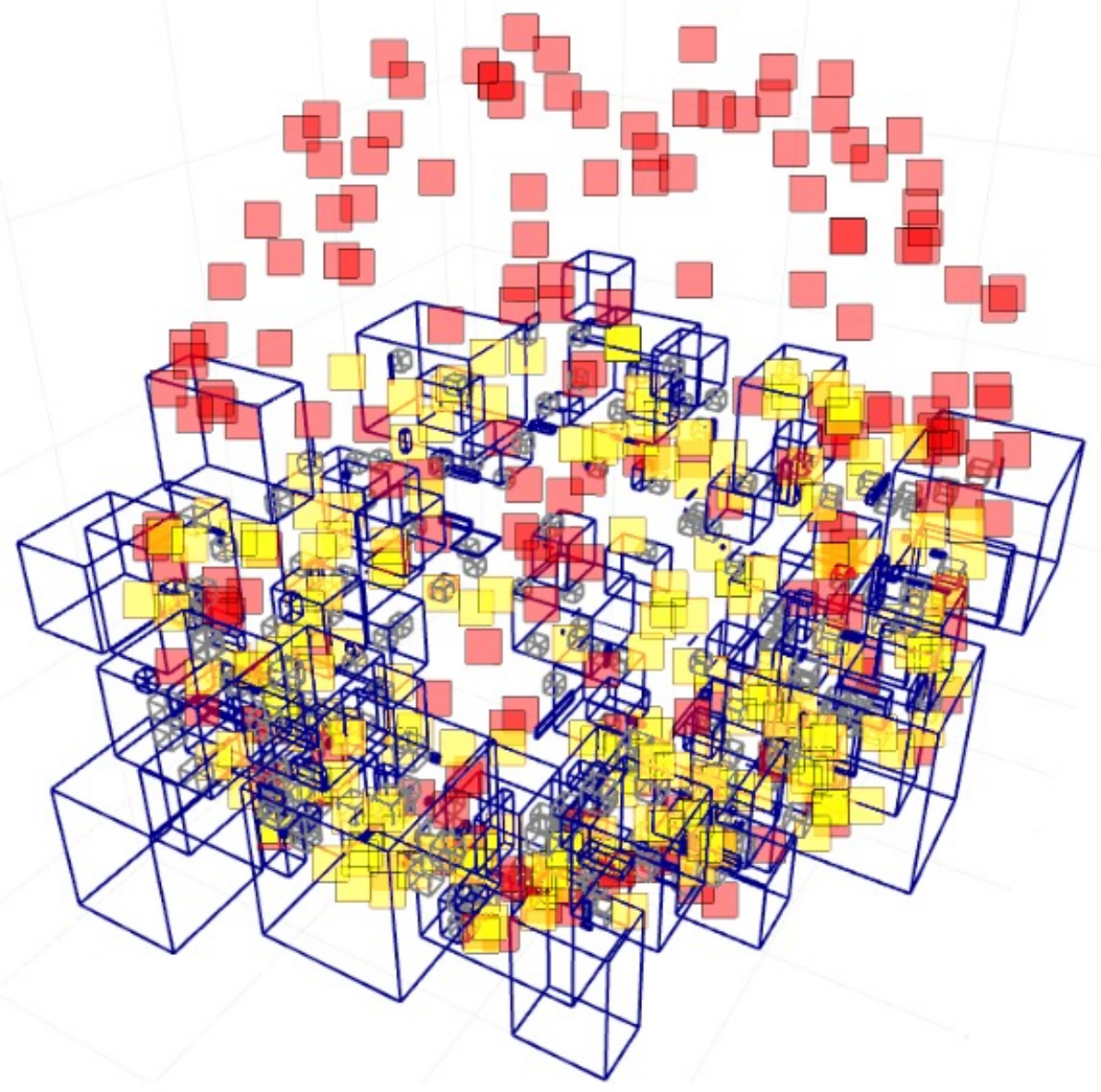} \\
    \raisebox{6mm}[0mm][0mm]{(a)} & \raisebox{6mm}[0mm][0mm]{(b)}\\[-15pt]
    \end{tabular}}
    \caption{Visualization of 
    the calculated sliced Wasserstein barycenters (blue boxes)
    between the voxelized target measures
    for (a) two rotated hemispheres (red, yellow)
    and (b) sphere (red) and hemisphere (yellow).}
    \label{fig:vis_barycenter}
\end{figure}

\begin{figure}
    \resizebox{\linewidth}{!}{%
    \begin{tabular}{@{} c @{} c @{} c @{}}
        \includegraphics[width=0.5\linewidth,clip=true,trim=20pt 20pt 120pt 20pt]{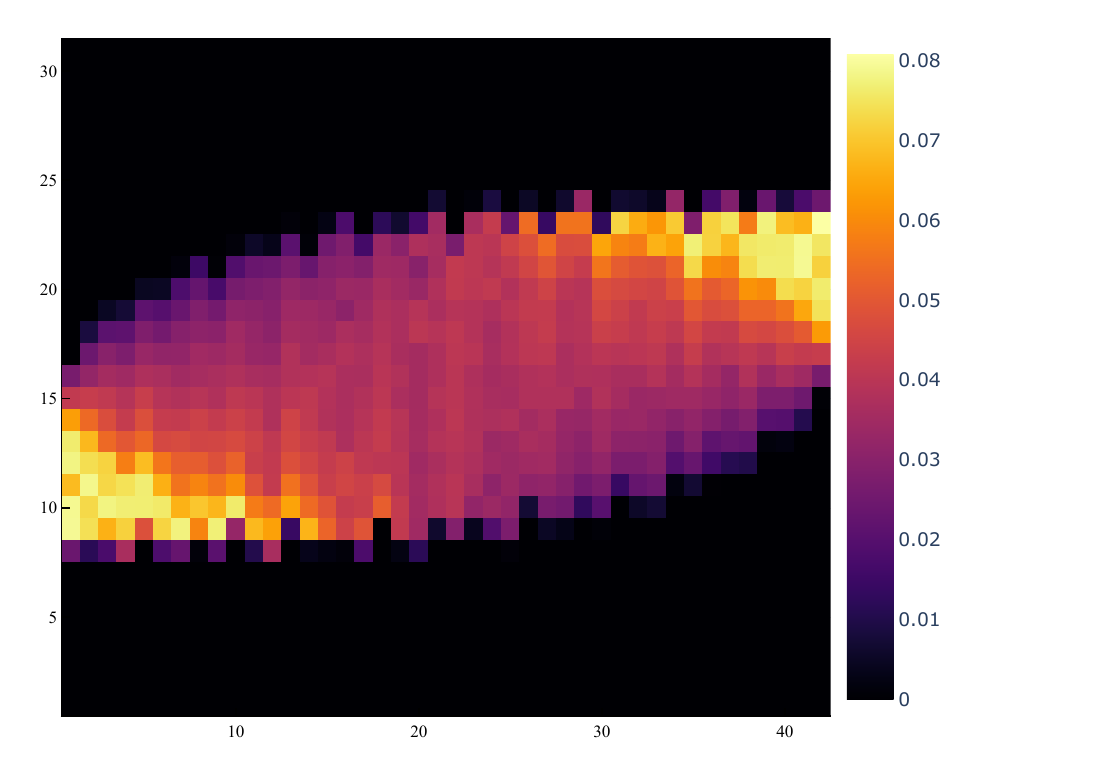}
        \includegraphics[width=0.5\linewidth,clip=true,trim=20pt 20pt 120pt 20pt]{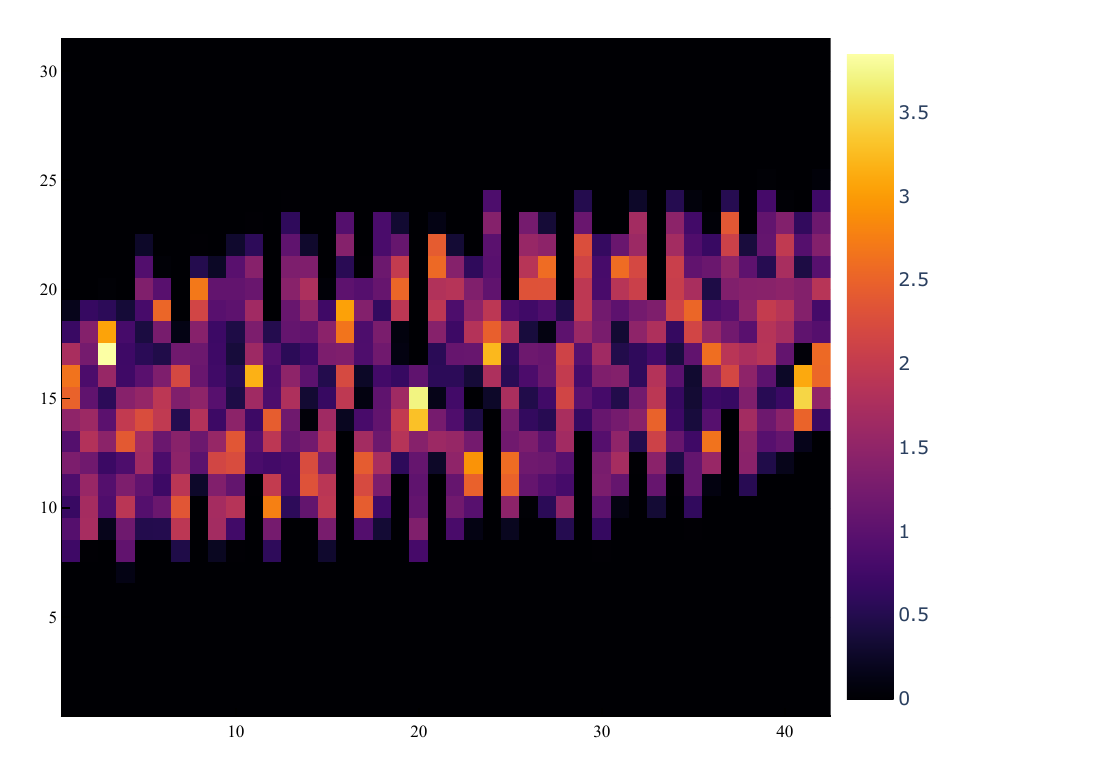}
        \includegraphics[width=0.5\linewidth,clip=true,trim=20pt 20pt 120pt 20pt]{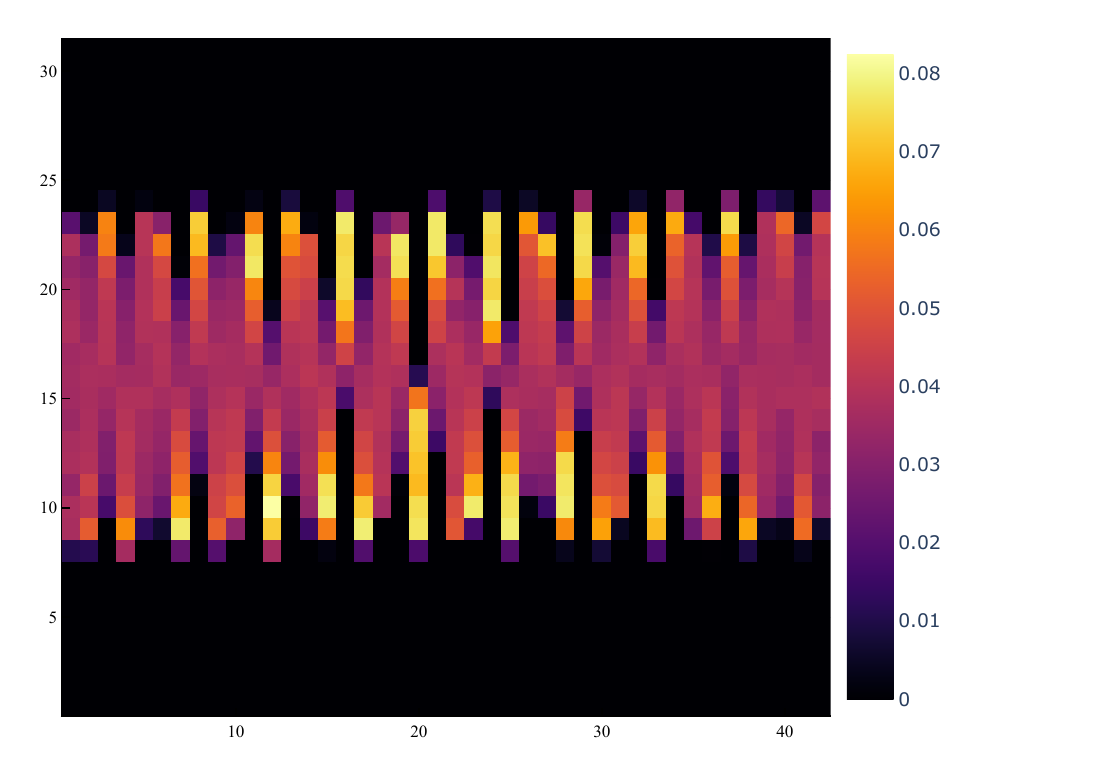} \\
        \includegraphics[width=0.5\linewidth,clip=true,trim=20pt 20pt 120pt 20pt]{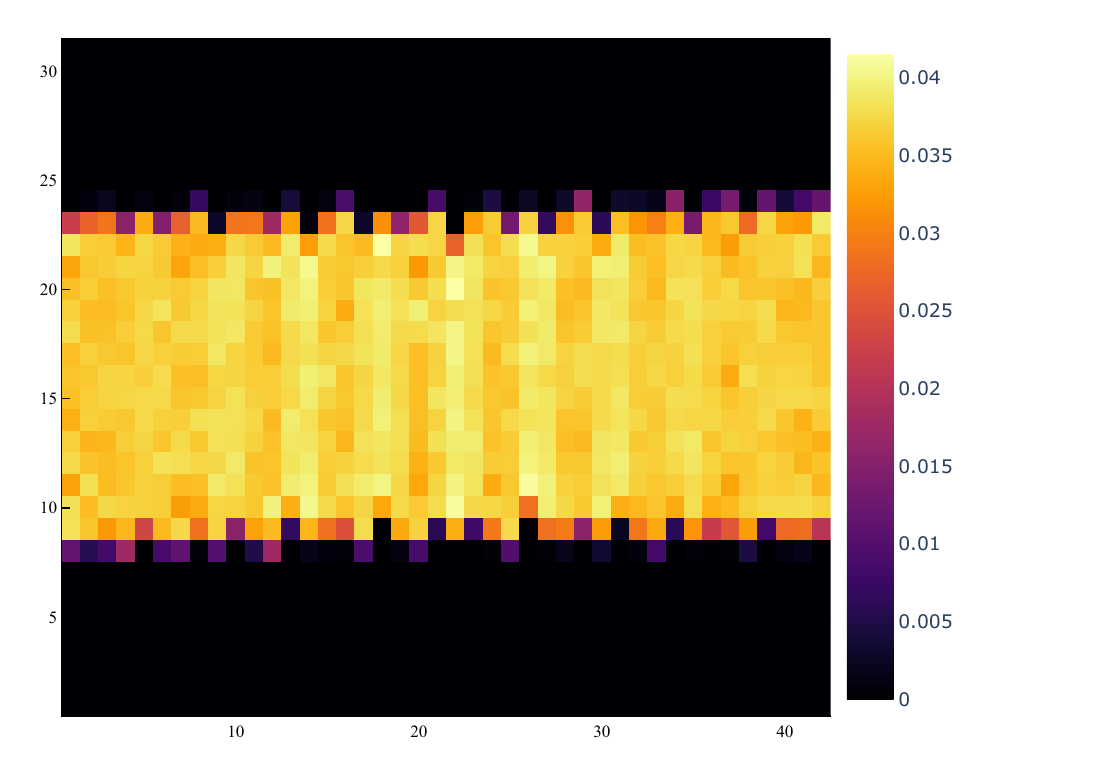}
        \includegraphics[width=0.5\linewidth,clip=true,trim=20pt 20pt 120pt 20pt]{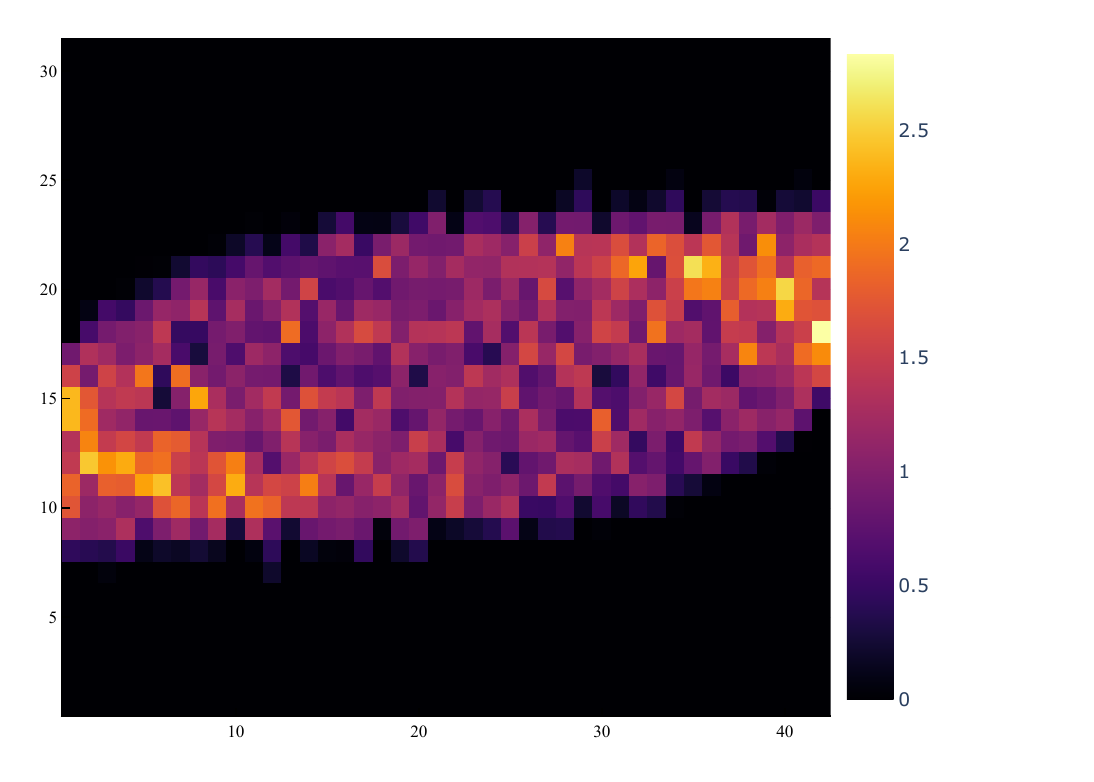}
        \includegraphics[width=0.5\linewidth,clip=true,trim=20pt 20pt 120pt 20pt]{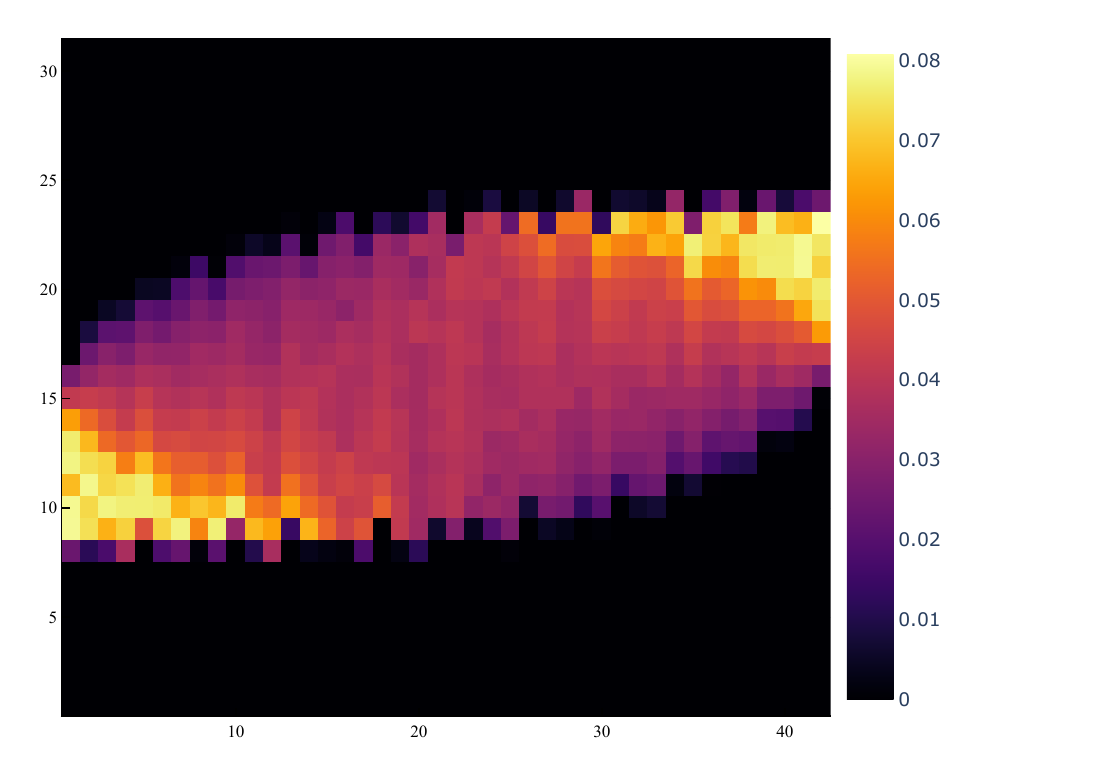}
    \end{tabular}}
    \caption{3D Radon transform of the two reference measures $\mu_1$ (left) and $\mu_2$ (right),
    and the computed barycenter ($\lambda = \nicefrac{1}{2}$) 
    from \eqref{eq:BSW} (middle).
    The measures are visualized in Fig.~\ref{fig:vis_barycenter}.}
    \label{fig:sw_bary_sino}
\end{figure}

\section{Comparison with Monte Carlo integration}
\label{sec:MC}

Another implementation of the {Radon} transform,
cf.~Rem.~\ref{rem:binning},
is based on a Monte Carlo integration, see \cite{Agarwal2019,ChyManRei2008,DeHoop1996}.
Similar to \eqref{eq:limSlap},
we approximate the intersection area of  $(-1,1]^d$ 
and the hyperplane $H_{\bftheta}(t)$
by
\begin{equation}
    A_{\bftheta}^{\zbe}(t)
    \approx \tfrac{1}{2\varepsilon} V_{\bftheta}^{\zbe}(t-\varepsilon,t+\varepsilon)
    = \tfrac{\E_{\zbx \sim \mathcal U((-1,1]^d)}[|\langle \bftheta,\zbx\rangle - t| \leq \varepsilon]}{2^{1-d}\varepsilon}
    \label{eq:RadonMC}
\end{equation}
with $N$ random samples of $\zbx$ from the uniform distribution $\mathcal U((-1,1]^d)$.
We use $d=4$,
$N = 2^q$ samples with $q \in \{16,20,24\}$, 
and $\varepsilon = 10^{-3}$ for the Monte Carlo sampling.
We generate quasi-uniform points on the sphere $\sphere^{3}$
by applying the inverse {Gaußian} {cumulative distribution function}
to {Sobol'} points \cite{Sobol1967} in $[0,1]^4$
and normalizing to the sphere, 
following the construction in \cite[§~5.2]{Lemieux2009}.
Notably,
the evaluation of our formula \eqref{eq:Radon-cube} 
for $128$ angles $\zb\theta$ 
and $128$ radii $t$
takes less than 1 second,
whereas
the Monte Carlo sampling 
for just $1$ angle and $128$ radii takes about $5$ seconds
with $N = 2^{24}$ samples for an appropriate result with less than $1\%$ error.
The error depending on $N$ is plotted in Fig.~\ref{fig:MC}.

\begin{figure}\centering
    $A_{(\pm1,\pm1,\pm1,\pm1)^\perp}^{\zbe}$ \\
    \includegraphics[width=0.99\linewidth]{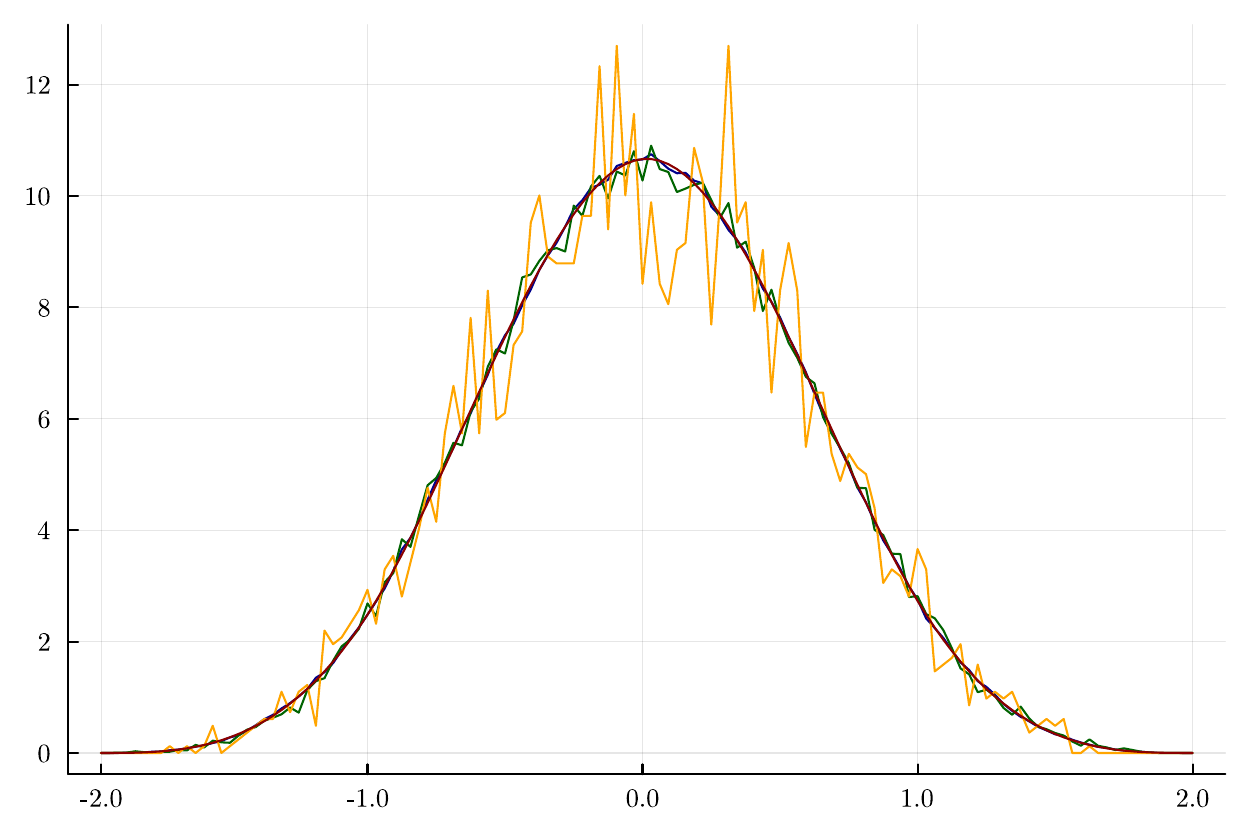}
    \\
    mean absolute difference
    \\
    \includegraphics[width=0.99\linewidth]{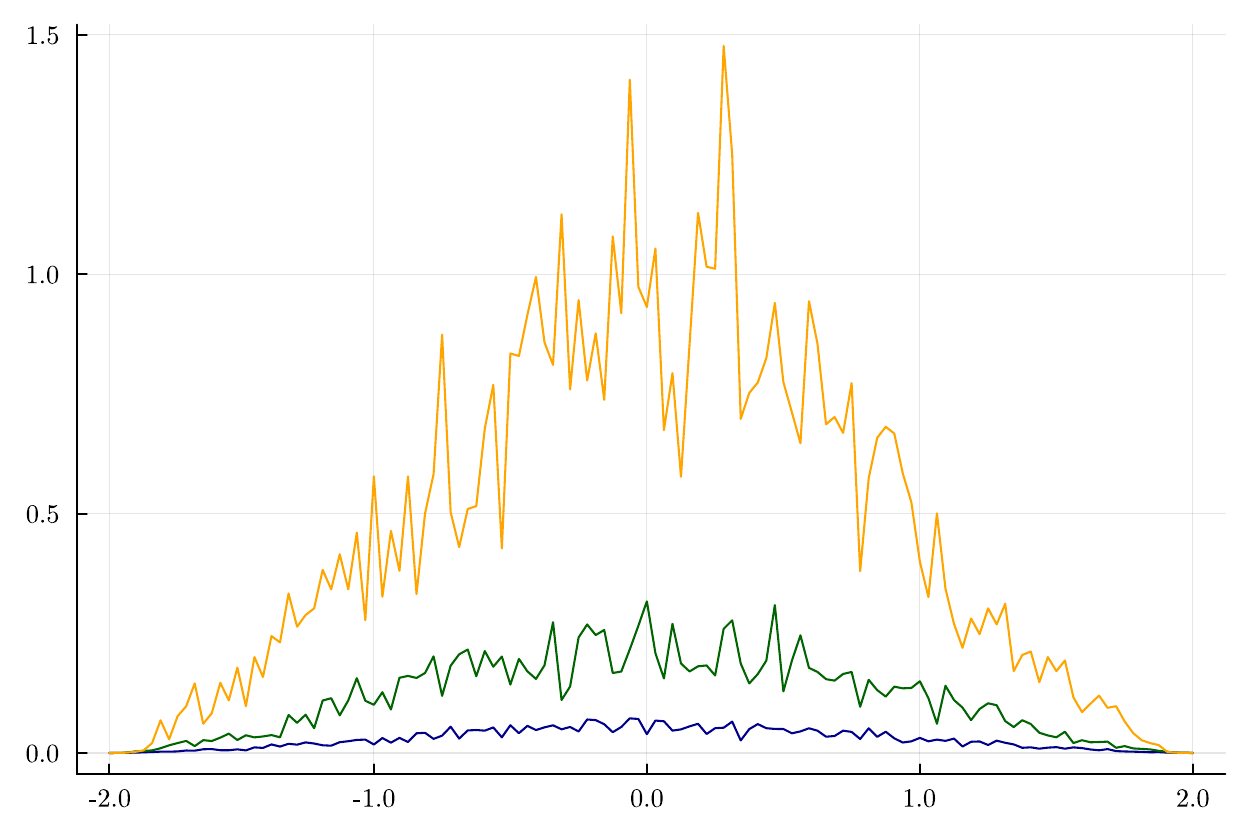} 
    \\
    {\includegraphics[width=0.22\linewidth, clip=true, trim=480pt 358pt 35pt 30pt]{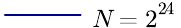}
    \hfill \includegraphics[width=0.22\linewidth, clip=true, trim=480pt 345pt 35pt 43pt]{numerics/exact_GMC_comp_765.pdf}
    \hfill \includegraphics[width=0.22\linewidth, clip=true, trim=480pt 332pt 35pt 56pt]{numerics/exact_GMC_comp_765.pdf}
    \hfill \includegraphics[width=0.22\linewidth, clip=true, trim=480pt 320pt 35pt 70pt]{numerics/exact_GMC_comp_765.pdf}}
    \caption{Comparison of the exact formula \eqref{eq:Radon-cube} 
    and the Monte Carlo approximation \eqref{eq:RadonMC}
    of the Radon transform $A^{\zbe}_{\bftheta}(t)$ in $d=4$ (top),
    and the mean absolute difference for 20 repetitions (bottom), both depending on $t$.}
    \label{fig:MC}
\end{figure}

\section*{Details of the Implementation}

We implemented the multidimensional discrete {Radon} transform  
and the experiments in Julia.%
\footnote{Code available at
\url{https://github.com/JJEWBresch/XdRadonTransform}.}
They are performed on a MacBookPro 2020 
with an Intel Core i5 CPU (4 cores, 1.4~GHz) and 8~GB~RAM.

\section*{Acknowledgments}
The authors thank M.~Piening for drawing the attention 
on the experiment in Section~\ref{sec:clustering_empirical_measure}.

\section*{Funding}
Funding by the DFG under the SFB 10.55776/F68 “Tomography Across the Scales” (STE 571/19-1, project number: 495365311) is gratefully acknowledged. For the purpose of open access, the authors have applied a CC BY public copyright license to any authors-accepted manuscript version arising from this submission.

\bibliographystyle{abbrvurl}
\bibliography{ref}

\newpage

\end{document}